\documentclass[11pt,reqno]{amsart}%amsart

\usepackage{amsmath,amsthm,amssymb,amsfonts, comment,fullpage}
\usepackage{braket}
\usepackage{mathtools}
\usepackage{bbm}
\usepackage{caption}
\usepackage{times}
\usepackage{mathrsfs}
\usepackage{latexsym}
\usepackage[dvips]{graphics}
\usepackage{epsfig}
\usepackage{amsmath,amsfonts,amsthm,amssymb,amscd}
\input amssym.def
\input amssym.tex
\usepackage{color}
\usepackage{hyperref}
\usepackage{url}
\newcommand{\bburl}[1]{\textcolor{blue}{\url{#1}}}
\usepackage{diagbox}
\usepackage{tikz}
\usepackage{tkz-tab}
\usepackage{tkz-graph}
\usetikzlibrary{shapes.geometric,positioning}
\usetikzlibrary{decorations.pathreplacing}

\newcommand{\burl}[1]{\textcolor{blue}{\url{#1}}}

\numberwithin{equation}{section}

\newtheorem{thm}{Theorem}[section]

\newtheorem{cor}[thm]{Corollary}
\newtheorem{lem}[thm]{Lemma}
\newtheorem{prop}[thm]{Proposition}

\newtheorem{defi}[thm]{Definition}

\newtheorem{que}[thm]{Question}

\theoremstyle{plain}

\newtheorem{corollary}[thm]{Corollary}
\newtheorem{definition}[thm]{Definition}
\newtheorem{example}[thm]{Example}
\newtheorem{lemma}[thm]{Lemma}
\newtheorem{proposition}[thm]{Proposition}
\newtheorem{theorem}[thm]{Theorem}

\newtheorem{remark}[thm]{Remark}

\newcommand\be{\begin{equation}}
\newcommand\ee{\end{equation}}
\newcommand\bee{\begin{equation*}}
\newcommand\eee{\end{equation*}}
\newcommand\bea{\begin{eqnarray}}
\newcommand\eea{\end{eqnarray}}
\newcommand\beae{\begin{eqnarray*}}
\newcommand\eeae{\end{eqnarray*}}
\newcommand\bi{\begin{itemize}}
\newcommand\ei{\end{itemize}}
\newcommand\ben{\begin{enumerate}}
\newcommand\een{\end{enumerate}}
\newcommand\bc{\begin{center}}
\newcommand\ec{\end{center}}
\newcommand\ba{\begin{array}}
\newcommand\ea{\end{array}}

\newcommand{\R}{\ensuremath{\mathbb{R}}}

\newcommand{\Z}{\ensuremath{\mathbb{Z}}}

\newcommand{\N}{\mathbb{N}}

\newcommand\frakfamily{\usefont{U}{yfrak}{m}{n}}
\DeclareTextFontCommand{\textfrak}{\frakfamily}

\newcommand{\hr}[1]{\href{#1}{\url{#1}}}

\newcommand{\E}[1]{\mathbb{E}[#1]}

\title{Paper template}
\author{Glenn Bruda}
\email{\textcolor{blue}{\href{mailto:glenn.bruda@ufl.edu}{glenn.bruda@ufl.edu}}}
\address{Department of Mathematics, University of Florida, Gainesville, FL 32611}

\author{Bruce Fang}
\email{\textcolor{blue}{\href{mailto:bf8@williams.edu}{bf8@williams.edu}}}
\address{Department of Mathematics and Statistics, Williams College, Williamstown, MA 01267}

\author{Raul Marquez}
\email{\textcolor{blue}{\href{mailto:raul.marquez02@utrgv.edu}{raul.marquez02@utrgv.edu}}}
\address{School of Mathematical and Statistical Sciences, University of Texas Rio Grande Valley, Brownsville, TX 78520}

\author{Steven J. Miller}
\email{\textcolor{blue}{\href{mailto:sjm1@williams.edu}{sjm1@williams.edu}},  \textcolor{blue}{\href{Steven.Miller.MC.96@aya.yale.edu}{Steven.Miller.MC.96@aya.yale.edu}}}
\address{Department of Mathematics and Statistics, Williams College, Williamstown, MA 01267}

\author{Beni Prapashtica}
\email{\textcolor{blue}{\href{mailto:bp492@cam.ac.uk}{bp492@cam.ac.uk}}}
\address{Department of Pure Mathematics and Mathematical Statistics, University of Cambridge, Cambridge, UK}

\author{Vismay Sharan}
\email{\textcolor{blue}{\href{mailto:vismay.sharan@yale.edu}{vismay.sharan@yale.edu}}}
\address{Department of Mathematics, Yale University, New Haven, CT 06511}

\author{Daeyoung Son}
\email{\textcolor{blue}{\href{mailto:ds15@williams.edu}{ds15@williams.edu}}}
\address{Department of Mathematics and Statistics, Williams College, Williamstown, MA 01267}

\author{Saad Waheed}
\email{\textcolor{blue}{\href{mailto:sw21@williams.edu}{sw21@williams.edu}}}
\address{Department of Mathematics and Statistics, Williams College, Williamstown, MA 01267}

\author{Janine Wang}
\email{\textcolor{blue}{\href{mailto:jjw3@williams.edu}{jjw3@williams.edu}}}
\address{Department of Mathematics and Statistics, Williams College, Williamstown, MA 01267}

\thanks{This research was supported by the National Science Foundation grant DMS-2241623, Churchill College Cambridge, the Finnerty fund, and Williams College.}

\subjclass[2010]{}


\date{\today}

\title{The Limiting Spectral Distribution of Various Matrix Ensembles Under the Anticommutator Operation}
\begin{document}

\begin{abstract}
Inspired by the quantization of classical quantities and Rankin Selberg convolution, we study the anticommutator operation $\{\cdot, \cdot\}$, where $\{A,B\} = AB + BA$, applied to real symmetric random matrix ensembles including Gaussian orthogonal ensemble (GOE), the palindromic Toeplitz ensemble (PTE), the $k$-checkerboard ensemble, and the block $k$-circulant ensemble ($k$-BCE). Using combinatorial and topological techniques related to non-crossing and free matching properties of GOE and PTE, we obtain closed-form formulae for the moments of the limiting spectral distributions of $\{\textup{GOE, GOE}\}$, $\{\textup{PTE, PTE}\}$, $\{\textup{GOE, PTE}\}$ and establish the corresponding limiting spectral distributions with generating functions and convolution. On the other hand, $\{\textup{GOE, $k$-checkerboard}\}$ and $\{\textup{$k$-checkerboard, $j$-checkerboard}\}$ exhibit entirely different spectral behavior than the other anticommutator ensembles: while the spectrum of $\{\textup{GOE, $k$-checkerboard}\}$ consists of 1 bulk regime of size $\Theta(N)$ and 1 blip regime of size $\Theta(N^{3/2})$, the spectrum of $\{\textup{$k$-checkerboard, $j$-checkerboard}\}$ consists of 1 bulk regime of size $\Theta(N)$, 2 intermediary blip regimes of size $\Theta(N^{3/2})$, and 1 largest blip regime of size $\Theta(N^2)$. In both cases, with the appropriate weight function, we are able to isolate the largest regime for other regime(s) and analyze its moments and convergence results via combinatorics. We end with numerical computation of lower even moments of $\{\textup{GOE, $k$-BCE}\}$ and $\{\textup{$k$-BCE, $k$-BCE}\}$ based on genus expansion and discussion on the challenge with analyzing the intermediary blip regimes of $\{\textup{$k$-checkerboard, $j$-checkerboard}\}$.
\end{abstract}

\maketitle

\tableofcontents

%%%%%%%%%%%%%%%%%%%%%%%%%%%%%%%%%%%%%%%%%%%%%%%%%%%%%%%%%%%%%%%%%%%%%%%%%%%%%%%%%%%%%%%%%%%%%%%%%%%%%%%%%%%%%%%%%%%%%%%%%%%%%%%%%%%%
%%%%%%%%%%%%%%%%%%%%%%%%%%%%%%%%%%%%%%%%%%%%%%%%%%%%%%%%%%%%%%%%%%%%%%%%%%%%%%%%%%%%%%%%%%%%%%%%%%%%%%%%%%%%%%%%%%%%%%%%%%%%%%%%%%%%
%%%%%%%%%%%%%%%%%%%%%%%%%%%%%%%%%%%%%%%%%%%%%%%%%%%%%%%%%%%%%%%%%%%%%%%%%%%%%%%%%%%%%%%%%%%%%%%%%%%%%%%%%%%%%%%%%%%%%%%%%%%%%%%%%%%%
\section{Introduction}

\subsection{Background}
Random matrices were introduced by Wishart \cite{Wishart} in the 1920s to estimate the sample covariance matrices of large samples. Since then, random matrices and statistics of their eigenvalues have been widely studied, with numerous important applications to number theory, quantum mechanics, condensed matter physics, wireless communications \cite{Mezzadri, Mehta, Forrester, Couillet}, etc. Discovered by Wigner while he was investigating nuclear resonance levels \cite{Wigner1,Wigner2}, Wigner's semi-circle law is a central result in random matrix theory and states that under certain conditions, the normalized spectral distribution of symmetric random matrices with i.i.d. entries converges weakly almost surely to a semi-circle as the size of the matrices tends to infinity. The limiting normalized spectral distribution of different types of random matrix ensembles has been extensively researched, see for example \cite{BasBo1,BasBo2,BLMST,BHS1,BHS2,FM,GKMN,Toeplitz,McK,Me}.

Even though the spectral distribution of matrices from various random matrix ensembles converges to the semicircle distribution, additional symmetries on the entries of a matrix ensemble often allow us to observe new behavior. Examples of such ensembles include Toeplitz \cite{Toeplitz}, palindromic Toeplitz \cite{palindromicToeplitz}, $k$-checkerboard \cite{split}, adjacency matrices of $d$-regular graphs \cite{GKMN}, and block circulant \cite{Block Circulant} ensembles. In general, it is rare to find a well-known, closed-form expression of the limiting spectral distribution of a random matrix ensemble.

The new construction in this paper was inspired by the notion of quantization in quantum mechanics. In contrast to classical mechanics where commutativity holds for the product of two classical variables $ab$ (i.e., $ab=ba$), in quantum mechanics the composition (or product) of two quantum operators $\hat{A}\hat{B}$ is generally not commutative (i.e., $\hat{A}\hat{B}\neq \hat{B}\hat{A}$). A natural way to resolve this issue is to quantize $\hat{A}\hat{B}$ by the anticommutator product $(\hat{A}\hat{B}+\hat{B}\hat{A})/2$. As $1/2$ is only a scaling factor, we drop it from our definition of anticommutatr product \cite{Quantum}.

Our study of anticommutator of random matrices was also motivated by the deep connections between random matrix theory and number theory, which originated from a chance meeting between Montgometry and Dyson. They observed that the pair correlation of the nontrivial zeros of the Riemann zeta function, which encodes information about the spacing between adjacent zeros, matches with that of the eigenvalues of random Hermitian matrices in the Gaussian Unitary Ensemble (GUE) \cite{BFMT}. This observation started a long and fruitful relationship between these two areas. Subsequent work by Rudnick and Sarnak \cite{RS} generalized this connection to zeros of automorphic $L$-functions, which are generalization of Riemann zeta function in number theory. In the theory of integral representation of $L$-functions, Rankin-Selberg convolution is an important method to obtain a new family of $L$-function from given $L$-function families. Considering fruitful connections between random matrix theory and number theory, it is a natural question to ask if there is an analogue of such convolution in the context of random matrix theory.

Previous research on analogues of Rankin-Selberg convolution has considered the Hadamard product, the Kronecker product, the ``swirl'' operation, and the ``disco'' operation. All these operations take as input different types of random matrices and combine them together to form new random matrices \cite{disco, BLMST, Swirl, Convolution}. In our paper, we consider the convolution realized through anticommutating of two random matrices of the same size, defined as
\begin{align*}
    \{A, B\} \ := \ AB+BA,
\end{align*}
where $A,B$ are real symmetric random matrices of the same dimension. We call $\{A, B\}$ the \textbf{anticommutator} of $A$ and $B$. We choose to sample $A$ and $B$ from real symmetric random matrix ensembles known either for lack of symmetry or for additional symmetrical structure, such as GOE (lack of additional symmetry), palindromic Toeplitz (additional symmetry), block circulant (additional symmetry), and $k$-checkerboard (additional symmetry), in hope to see how the additional symmetry affects the spectral distribution of the anticommutator.

The anticommutator of random matrices has previously been studied by Nica and Speicher in \cite{commutator} using the machinery of free probability. They provided a general analytic formula for the moment series of the spectral distribution of the anticommutator of random matrices in terms of expressions involving compositional inversion and $R$-transform. The analytic formula has successfully yielded the moment series, and thus the spectral distribution of the anticommutator of random matrices whose spectral distributions are semicircle, free Poisson, arcsine, or Bernoulli, etc. However, this method generally fails in the case of matrix ensembles with additional symmetries, due to the intractability of the $R$-transform and the compositional inversion computation. We demonstrate that we can study the anticommutator of these matrix ensembles using combinatorial, topological, and moment methods. We define these matrix ensembles and summarize the results in the next section.

\subsection{Preliminaries}

We study the anticommutators of the matrices from real symmetric matrix ensembles as defined in this section. We choose to form anticommutators from these ensembles because they all share nice combinatorial structures and closed form expressions for limiting expected moments or even limiting spectral distribution, which would increase our chance of being able to study their anticommutator through combinatorics and find closed form expressions for the limiting expected moments or the limiting spectral distributions the anticommutator. Moreover, we choose to include the $k$-checkerboard ensemble, whose spectrum typically consists of two regimes of different orders, because we are interested in what the splitting behavior of the spectrum would look like if we take the anticommutator of the $k$-checkerboard ensemble and another matrix ensemble.

\begin{definition} The $N\times N$ \textbf{Gaussian Orthogonal Ensemble (GOE}) is a random matrix ensemble whose matrices $A_N=(a_{ij})$ are given by 
\begin{align}
a_{ij} \ = \ a_{ji} \ \sim \ \begin{cases}
    b_{ij} &\textup{if $i\neq j$} \\
    c_{ii} &\textup{if $i=j$},
\end{cases}
\end{align}
where $b_{ij}\sim N(0,1)$ and $c_{ii}\in N(0,2)$ for all $1\leq i,j\leq N$.
\end{definition}

\begin{definition}
The $N\times N$ \textbf{real symmetric palindromic Toeplitz ensemble (PTE)} (where $N$ is assumed to be even for simplicity) is a random matrix ensemble whose matrices $M_N$ have entries paramatrized by $b_0, b_1,\dots,b_{N/2-1}$, where the $b_i$'s are i.i.d. random variables with mean 0, variance 1, and finite higher moments:
\begin{align}
a_{ij} \ = \ \begin{cases}
b_{|i-j|}& \text{if } 0\leq |i-j|\leq \frac{N}{2}-1\\
b_{N-1-|i-j|}& \text{if } \frac{N}{2}\leq |i-j|\leq N-1.
\end{cases}
\end{align}
These matrices are therefore of the form
\begin{align}
\begin{pmatrix}
b_0 & b_1 & b_2 & \cdots & b_2 & b_1 & b_0\\
b_1 & b_0 & b_1 & \cdots & b_3 & b_2 & b_1\\
b_2 & b_1 & b_0 & \cdots & b_4 & b_3 & b_2\\
\vdots & \vdots & \vdots & \ddots & \vdots & \vdots & \vdots\\
b_2 & b_3 & b_4 & \cdots & b_0 & b_1 & b_2\\
b_1 & b_2 & b_3 & \cdots & b_1 & b_0 & b_1\\
b_0 & b_1 & b_2 & \cdots & b_2 & b_1 & b_0
\end{pmatrix}.
\end{align}
\end{definition}
\begin{definition}
Let $k|N$. The $N\times N$ \textbf{real symmetric $k$-block circulant ensemble ($k$-BCE)} is a random matrix ensemble whose matrices are block matrices of the form
\begin{equation}
\left(\begin{array}{ccccc}
B_0 & B_1 & B_2 & \cdots & B_{\frac{N}{k}-1}\\
B_{-1} & B_0 & B_1 & \cdots & B_{\frac{N}{k}-2}\\
B_{-2} & B_{-1} & B_0 & \cdots & B_{\frac{N}{k}-3}\\
\vdots & \vdots & \vdots & \ddots & \vdots\\
B_{1-\frac{N}{k}} & B_{2-\frac{N}{k}} & B_{3-\frac{N}{k}} & \cdots & B_0
\end{array}\right),
\end{equation}
where each $B_i$ is an $k\times k$ real matrix, each $B_{-i} = B_i^T$, and $B_0$ is symmetric. Unless otherwise specified (i.e., the restriction that $B_{-i}=B_i^T$), the entries of these block matrices are i.i.d. with mean 0, variance 1, and finite higher moments.
\end{definition}

\begin{definition} For $k\in\mathbb{Z}_{>0}$ and $w\in\mathbb{R}$, the $N\times N$ \textbf{$(k,w)$-checkerboard ensemble} is a random matrix ensemble whose matrices $M_N=(m_{ij})$ are given by
\begin{align}
m_{ij} \ = \
\begin{cases}
a_{ij},& \text{if }i\not\equiv j\mod{k}\\
w,& \text{if }i\equiv j\mod{k},
\end{cases}
\end{align}
where $a_{ij}=a_{ji}$ and all of the distinct $a_{ij}$ terms are sampled from a distribution with mean $0$, variance $1$, and finite higher moments. For example, $N\times N$ $(2,w)$-checkerboard matrices are of the form
\begin{align*}
\begin{pmatrix}
w & a_{0\textup{ }1} & w & a_{0\textup{ }3} & w & \cdots & a_{0\textup{ }N-1} \\
a_{0\textup{ }1} & w & a_{1\textup{ }2} & w & a_{1\textup{ }4} & \cdots & w \\
w & a_{1\textup{ }} & w & a_{2\textup{ }3} & w & \cdots & a_{2\textup{ }N-1} \\
\vdots & \vdots & \vdots & \vdots & \vdots & \ddots & \vdots \\
a_{0\textup{ }N-1} & w & a_{2\textup{ }N-1} & w & a_{4\textup{ }N-1} & \cdots & w
\end{pmatrix}.
\end{align*}

We refer to the $(k, 1)$-checkerboard ensemble as the $k$-checkerboard ensemble. Unless specified otherwise, we always assume that the weight of the checkerboard be 1.
\end{definition}

\begin{definition}
The \textbf{mean matrix} $\overline{M}_N=(m_{ij})$ of the $N\times N$ $k$-checkerboard ensemble is given by
\begin{align}
\overline{m}_{ij} \ = \
\begin{cases}
0,& \text{if }i\not\equiv j\mod{k}\\
1,& \text{if }i\equiv j\mod{k},
\end{cases}
\end{align}
where we note that the rank of $\overline{M}_N$ is $k$.
\end{definition}

\begin{remark}
We say that $f(n)=O(g(n))$ or $f(n)\ll g(n)$ if there exist a positive real number $C$ and a real number $M$ such that $f(n)\leq Cg(n)$ for all $n\geq M$. If $f(n)=O(g(n))$ and $g(n)=O(f(n))$, then we say that $f(n)=\Theta(g(n))$. Moreover, we say that $f(n)=o(g(n))$ if $\lim_{n\rightarrow\infty}\frac{f(n)}{g(n)}=0$. On the other hand, we say that $f(n)=\omega(g(n))$ if $\lim_{n\rightarrow\infty}\frac{f(n)}{g(n)}=\infty$.
\end{remark}

\begin{definition}
As we see later, the eigenvalues of the anticommutator of two random matrices from the above ensembles (except for the $k$-checkerboard ensemble) are typically of $\Theta(N)$. Hence, we define the \textbf{empirical spectral measure} associated to such an $N\times N$ random matrix $C_N$ as
\begin{align}
\mu_{C_N, N}(x)dx \ = \ \frac{1}{N}\sum_{i=1}^N \delta\left(x-\frac{\lambda_i}{N}\right)dx,
\end{align}
where $\{\lambda_i\}_{i=1}^N$ are the eigenvalues of $C_N$ and $\delta$ is the Dira-delta functional. The \textbf{empirical spectral distribution} $F^{C_N/N}$ of $C_N/N$ is defined as
\begin{align}
F^{C_N/N}(x) \ := \ \int_{-\infty}^x\mu_{C_N, N}(y)dy \ = \ \frac{\#\{i\leq N:\lambda_{i}/N\leq x\}}{N}.
\end{align}
If $F^{C_N/N}$ is differentiable, then the \textbf{empirical spectral density} $f^{C_N/N}$ of $C_N/N$ is defined as the derivative of $F^{C_N/N}$. If as $N\rightarrow\infty$ we have $F^{C_N/N}$ converges in some sense (in probability or almost surely) to a distribution $F$, then we say that $F$ is the \textbf{limiting spectral distribution} of the matrix ensemble.
The \textbf{m\textup{\textsuperscript{th}} moment of the empirical spectral distribution $F^{C_N/N}$ of $C_N/N$}, denoted by $M_m(C_N)$, is naturally defined as
\begin{align}
M_m(C_N) \ := \ \int_{-\infty}^\infty x^mdF^{C_N/N}(x) \ = \ \int_{-\infty}^\infty x^m\mu_{C_N, N}(x)dx \ = \ \frac{\sum_{i=1}^N \lambda_i^m}{N^{k+1}}.
\end{align}
Moreover, we define the \textbf{expected m\textup{\textsuperscript{th}} moment of the empirical spectral distribution} of $C_N/N$, denoted by $M_m(N)$, as the average of $M_m(C_N)$ over all the $C_N$ in our chosen anticommutator matrix ensemble, i.e., $M_m(N)=\mathbb{E}\left[M_m(C_N)\right]$. Finally, we let $M_m$ to be the limit of $M_m(N)$ as $N\rightarrow\infty$.
\end{definition}

To establish the limiting spectral distribution of a random matrix ensemble, it is crucial to understand the limiting expected moments of the empirical spectral distribution. The formula below, known as the Eigenvalue Trace Formula, allows us to reduce the moment calculation into a combinatorial problem:

\begin{prop}
Let $A_N$ be an $N\times N$ real symmetric matrix, then
\begin{align}
M_m(A_N) \ = \ \frac{1}{N^{m+1}}\textup{Tr}(A^m_N) \ = \frac{1}{N^{m+1}}\sum_{1\leq i_1, \cdots, i_m\leq N} a_{i_1i_2}a_{i_2i_3}\cdots a_{i_mi_{1}},
\end{align}
where $\textup{Tr}(\cdot)$ denotes the trace of a matrix and $a_{ij}$'s are the entries of $A_N$ indexed by $ij$. Similarly, if $A_N$ is a random matrix drawn from an $N\times N$ real symmetric matrix ensemble, then
\begin{align}
M_m(N) \ = \ \frac{1}{N^{m+1}}\mathbb{E}\left[\textup{Tr}(A^m_N)\right] \ = \frac{1}{N^{m+1}}\sum_{1\leq i_1,\cdots, i_m\leq N}\mathbb{E}\left[a_{i_1i_2}a_{i_2i_3}\cdots a_{i_mi_1}\right],
\end{align}
where by $\mathbb{E}\left[\textup{Tr}(A^m_N)\right]$ we mean averaging over the $N\times N$ random matrix ensemble with each matrix $A_N$ weighted by its probability of occurring. We refer to each $a_{i_1i_2}a_{i_2i_3}\cdots a_{i_mi_1}$ as a \textbf{cyclic product} in the expected $m$\textup{\textsuperscript{th}} moment of the random matrix ensemble.
\end{prop}

\subsection{Results}
In this section, we summarize our findings about anticommutator ensembles of random matrices sampled from ensembles including GOE, the real symmetric PTE, the real symmetric $k$-BCE, and the $k$-checkerboard ensemble. We will study not only the \textbf{homogeneous} anticommutator ensembles $\{A_N,B_N\}$, where $A_N$ and $B_N$ are matrices from the same ensemble, but also look at inhomogeneous anticommutator ensembles $\{A_N,B_N\}$, where $A_N$ and $B_N$ are matrices from different ensembles. Using the method of genus expansion as we shall develop in Section \ref{sec: anticommutator Combinatorics}, the non-crossing property of the cyclic products of GOE, and the free matching property of the cyclic products of PTE, we are able to obtain closed-form formulae for the limiting expected moments of the empirical spectral distributions of $\{\textup{GOE, GOE}\}$, $\{\textup{PTE, PTE}\}$, $\{\textup{GOE, PTE}\}$. Specifically, the closed-form formulae for $\{\textup{GOE, GOE}\}$ and $\{\textup{PTE, PTE}\}$ are given by explicit, simple expressions, while that for $\{\textup{GOE, PTE}\}$ is given by a recurrence relation. We also provide genus expansion formulae for the limiting expected moments of the empirical spectral distributions of $\{\textup{GOE, $k$-BCE}\}$ and $\{\textup{$k$-BCE, $k$-BCE}\}$ based on their respective matching rules. We only concern ourselves with the limiting expected even moments, since as we shall see in Section \ref{sec: anticommutator Combinatorics}, the limiting expected odd moments vanish for all these ensembles. Despite the fact that it is generally rare to find a closed form expression for the limiting spectral density of any random matrix ensemble, we are able to obtain nice closed form expressions for the limiting spectral densities of $\{\textup{GOE, GOE}\}$, $\{\textup{PTE, PTE}\}$ using analytic methods including generating functions and convolution of distributions. Comparisons between the empirical spectral density and the theoretiacl spectral density for these ensembles are shown in Figure \ref{PTExPTE figure} and \ref{GOExGOE figure}. Note that from now on, unless otherwise specified, we refer to the limiting expected moments of the empirical spectral distribution of a matrix ensemble simply as the limiting expected moments of the matrix ensemble.

\begin{theorem}
The limiting expected $2m$\textup{\textsuperscript{th}} moment $M_{2m}$ of $\{\textup{GOE, GOE}\}$ is given by the explicit formula
\begin{align}
M_{2m} \ = \ \frac{1}{m}\sum_{k=1}^m2^k\binom{2m}{k-1}\binom{m}{k}.
\end{align}
\end{theorem}

\begin{cor}
The limiting spectral density of $\{\textup{GOE, GOE}\}$ is given by
\begin{align}
\mu(x) \ = \ -\frac{\sqrt{3}}{2\pi|x|}\left(\frac{3x^2+1}{9h(x)}-h(x)\right),\quad |x| \ \leq \ \sqrt{\frac{11+5\sqrt{5}}{2}},
\end{align}
where
\begin{align}
h(x) \ = \ \sqrt[3]{\frac{18x^2+1}{27}+\sqrt{\frac{x^2(1+11x^2-x^4)}{27}}}.
\end{align}
\end{cor}

\begin{theorem}\label{PTEPTE moments intro}
The limiting expected $2m$\textup{\textsuperscript{th}} moment $M_{2m}$ of $\{\textup{PTE, PTE}\}$ is $4^m((2m-1)!!)^2$.
\end{theorem}

\begin{cor}
The limiting spectral density of $\{\textup{PTE, PTE}\}$ is given by the convolution
\begin{align}
\mu(x) \ = \ f(x) * g(x),
\end{align}
where $f(x)$ is the probability density function of a chi-squared random variable with 1 degree of freedom and $g(x)$ is the probability density function of the negative of a chi-squared random variable with 1 degree of freedom.
\end{cor}

\begin{theorem}
The limiting expected $2m$\textup{\textsuperscript{th}} moment $M_{2m}$ of $\{\textup{GOE, PTE}\}$ is given by $\sigma_{m, 0}$, where $\sigma_{n, s}$ satisfies the initial conditions $\sigma_{0, s}=(2s-1)!!$ for $s\geq 1$ and $\sigma_{0,0}=1$, and the recurrence relation
\begin{align}
\sigma_{n, s} \ = \ \sum_{k=1}^n \sigma_{k-1,1}\cdot \sigma_{n-k,s}+\sigma_{k-1,0}\cdot \sigma_{n-k,s+1}.
\end{align}
\end{theorem}

\begin{cor}
The generating function $F(z,w):=\sum_{n,s\geq0}\sigma_{n,s}z^n w^s/s!$ satisfies the partial differential equation
\begin{align}
F(z,w) \ = \ (1-2w)^{-1/2}+z\left(\frac{\partial F(z,0)}{\partial w}F(z,w)+F(z,0)\frac{\partial F(z,w)}{\partial w}\right).
\end{align}
\end{cor}

\begin{theorem}
The limiting expected $2m$\textup{\textsuperscript{th}} moment $M_{2m}$ of $\{\textup{GOE, $k$-BCE}\}$ is given by the genus expansion formula
\begin{align}
M_{2m} \ = \ \sum_{C\in\mathcal{C}_{2,4m}}\sum_{\pi_C\in NCF_{2,C}(4m)}k^{\#(\gamma_{4m}\pi)-(2m+1)},
\end{align}
and the limiting expected $2m$\textup{\textsuperscript{th}} moment $M'_{2m}$ of $\{\textup{$k$-BCE, $k$-BCE}\}$ is given by the genus expansion formula
\begin{align}
M'_{2m} \ = \ \sum_{C\in\mathcal{C}_{2,4m}}\sum_{\pi_C\in \mathcal{P}_{2,C}(4m)}k^{\#(\gamma_{4m}\pi)-(2m+1)}.
\end{align}
\end{theorem}

\begin{figure}[ht]
\centering
\includegraphics[width=0.7\textwidth]{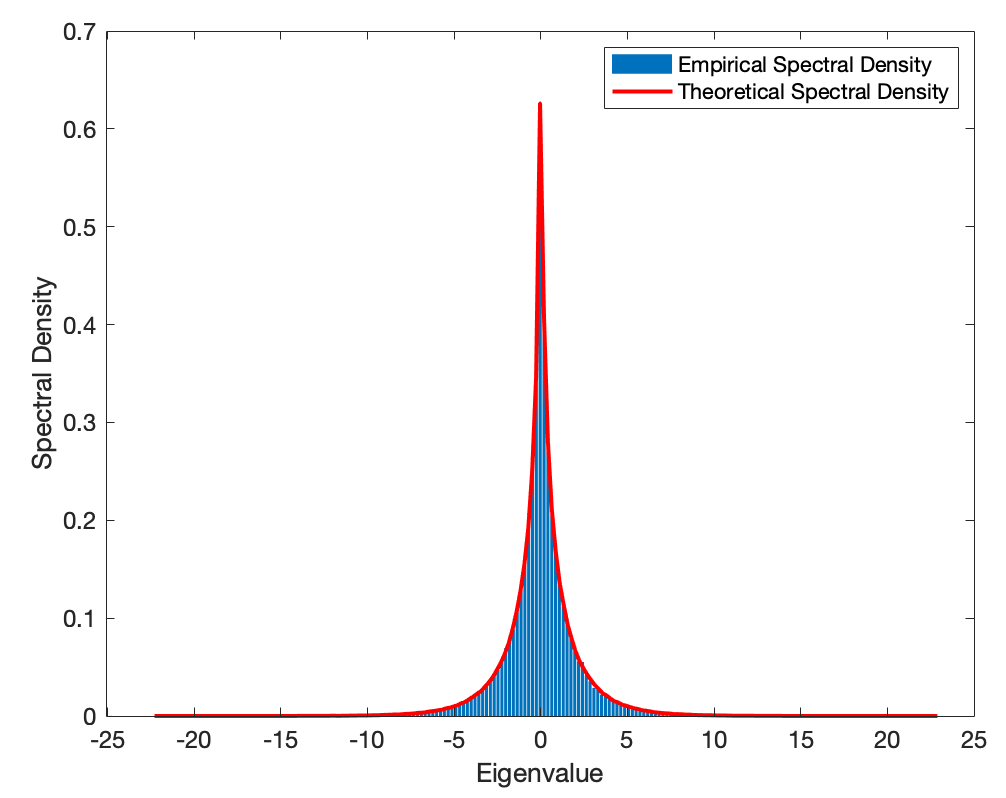}
\caption{Comparison between the normalized empirical spectral density for one hundred 1000$\times$1000 matrices from $\{\textup{PTE, PTE}\}$ and the theoretical spectral density.}
\label{PTExPTE figure}
\end{figure}

\begin{figure}[ht]
\centering
\includegraphics[width=0.7\textwidth]{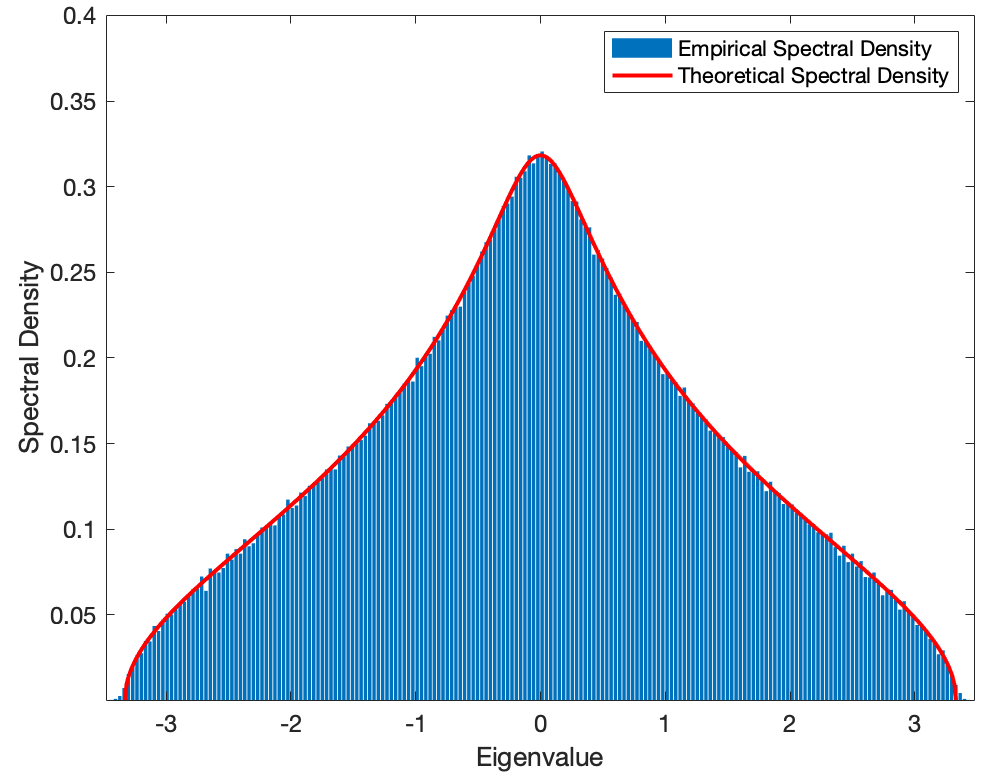}
\caption{Comparison between normalized empirical spectral density for one hundred 1000$\times$1000 matrices from $\{\textup{GOE, GOE}\}$ and the theoretical spectral density.}
\label{GOExGOE figure}
\end{figure}

\begin{figure}[ht]
\centering
\includegraphics[width=0.7\textwidth]{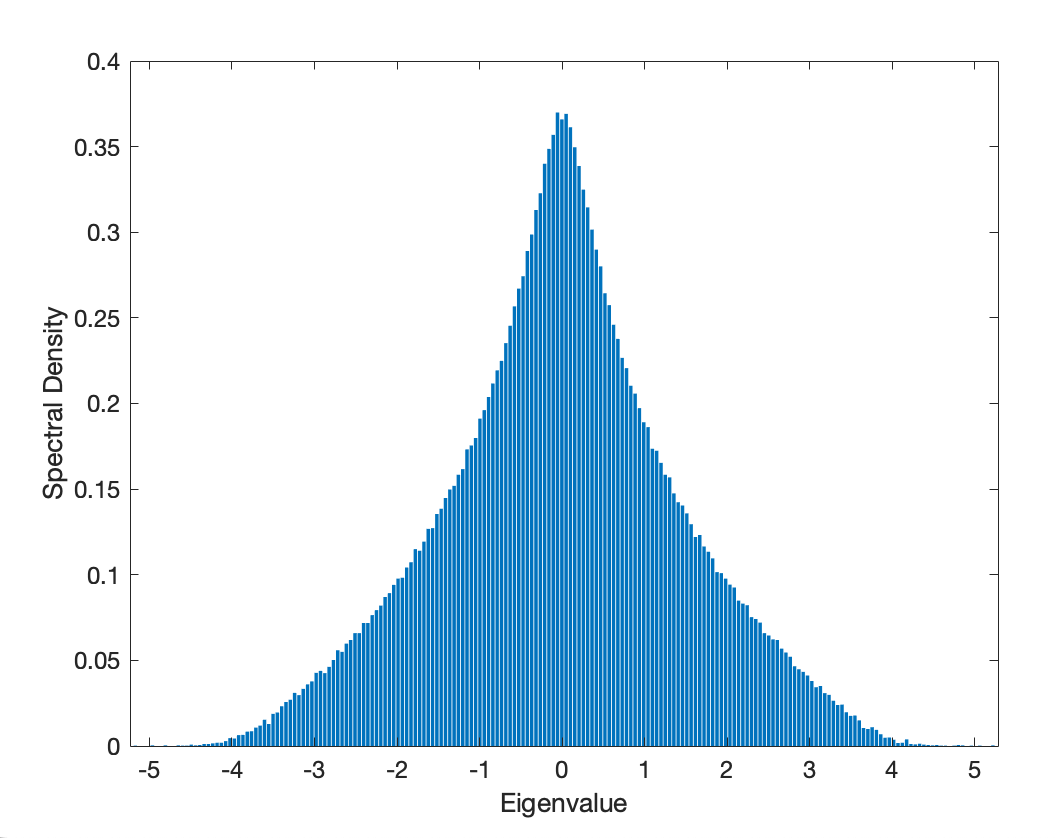}
\caption{Normalized empirical spectral density for one hundred 1000$\times$1000 matrices from $\{\textup{GOE, PTE}\}$.}
\label{GOExPTE figure}
\end{figure}

\begin{figure}[ht]
\centering
\begin{minipage}[b]{0.4\textwidth}
\includegraphics[width=\textwidth]{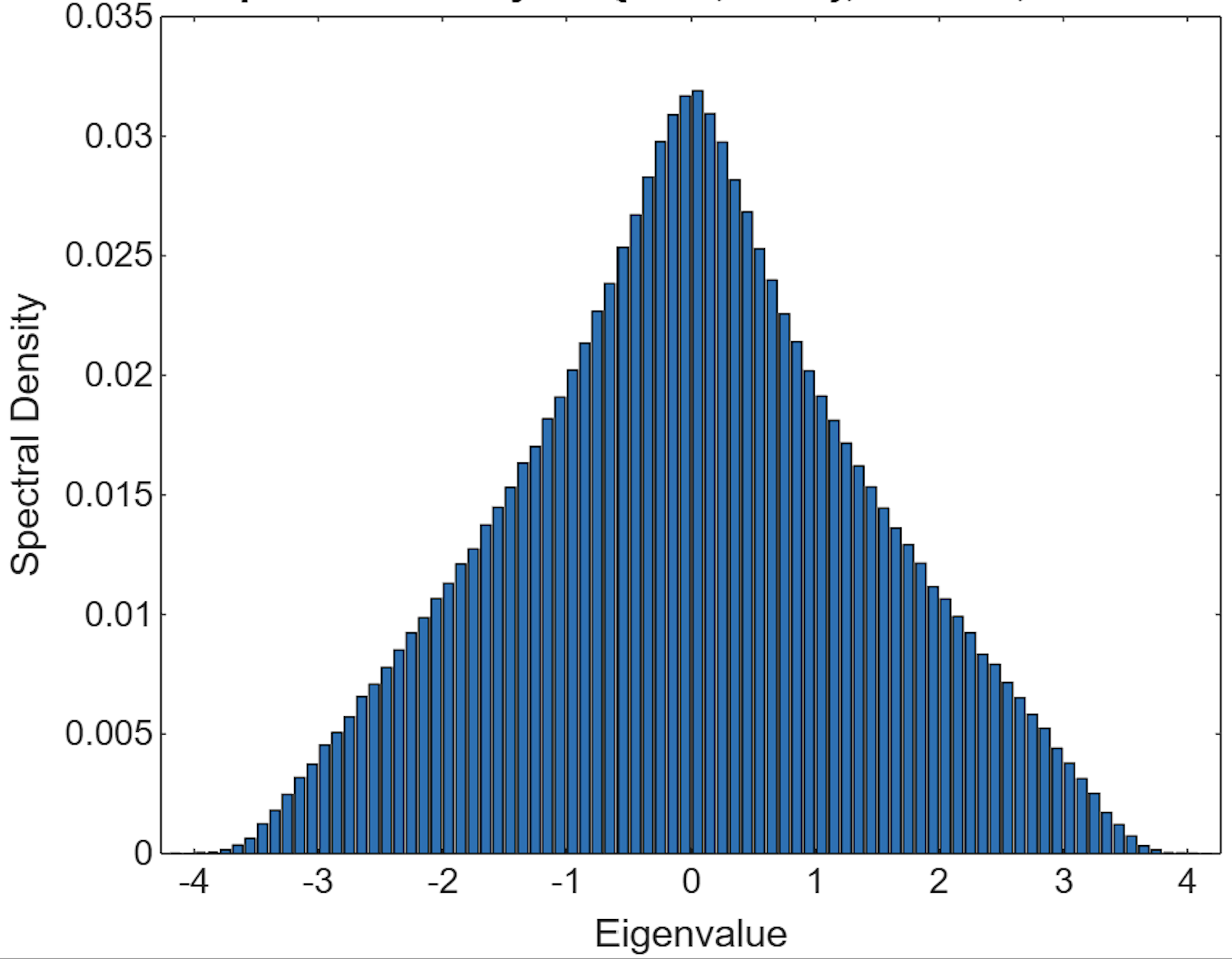}
\caption{Normalized empirical spectral density for one hundred 1500$\times$1500 matrices from $\{\textup{GOE, $2$-BCE}\}$.}
\label{GOExBCE figure}
\end{minipage}
\begin{minipage}[b]{0.4\textwidth}
\includegraphics[width=\textwidth]{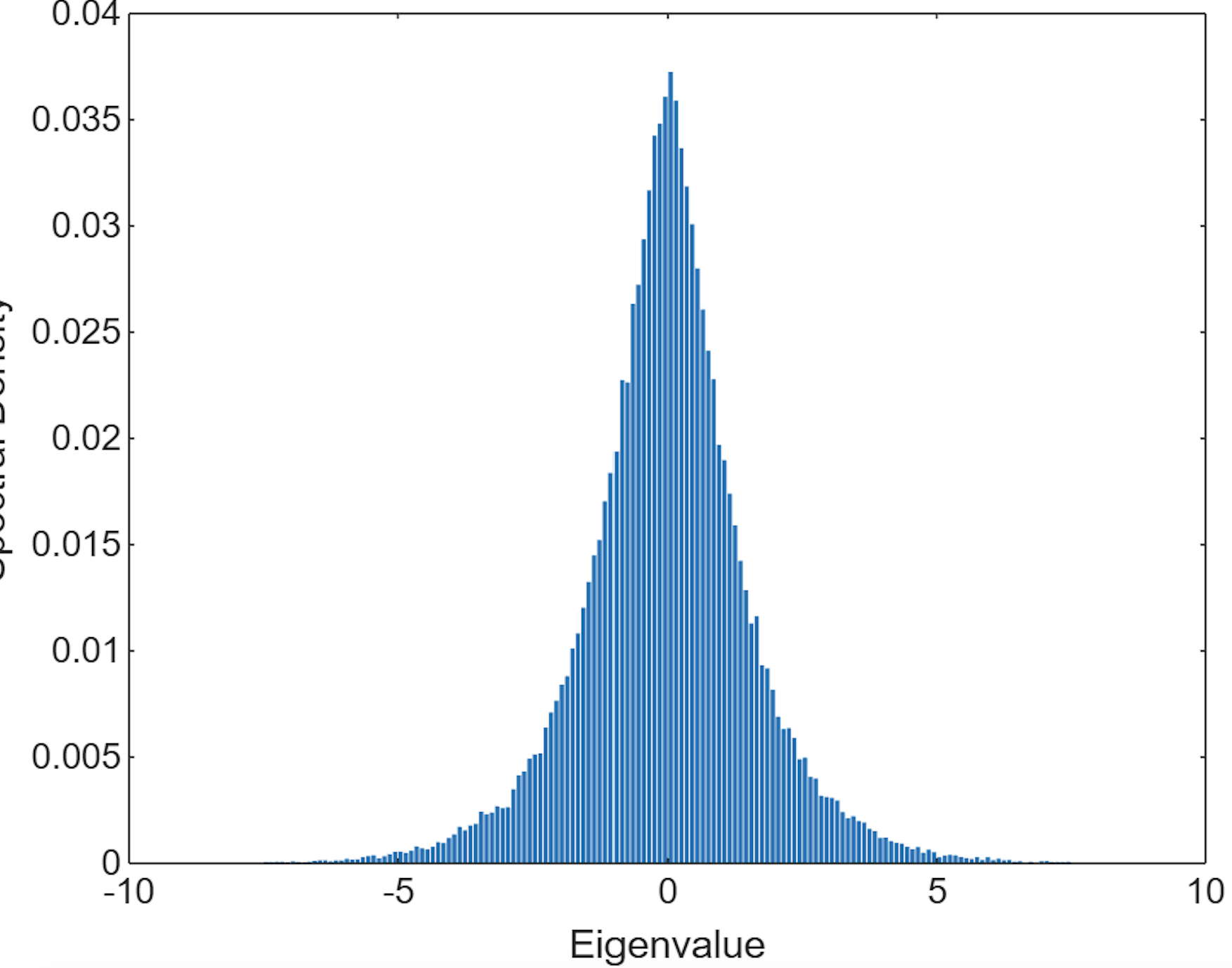}
\caption{Normalized empirical spectral density for one hundred 1500$\times$1500 matrices from $\{\textup{$2$-BCE, $2$-BCE}\}$.}
\label{BCExBCE figure}
\end{minipage}
\end{figure}

In contrast to anticommutator ensembles of random matrices sampled from GOE and the real symmetric PTE, anticommutator ensembles where one of the random matrix is sampled from the real symmetric $k$-BCE generally have more complex matrix structures and matching rules for the expected moments. Hence, for $\{\textup{GOE, $k$-BCE}\}$ and $\{\textup{$k$-BCE, $k$-BCE}\}$, even recurrence relations for the limiting expected moments are hard to come. Instead, we provide genus expansion formulae for these moments and numerical computation for lower even moments in Appendix \ref{lower even moments involving BCE}.

So far, all the random matrix ensembles we have considered have eigenvalues almost surely $\Theta(N)$. However, we begin to observe new behaviors as we consider anticommutator ensembles where at least one of the random matrices is sampled from the $k$-checkerboard. We restrict our attention to the following two ensembles: $\{\textup{GOE, $k$-checkerboard}\}$ and $\{\textup{$k$-checkerboard, $j$-checkerboard}\}$. If $k\mid N$, then the empirical spectral distribution of a matrix from an $N\times N$ $\{\textup{GOE, $k$-checkerboard}\}$ consists of one bulk regime of size $\Theta(N)$ and one blip regime of size $\Theta(N^{3/2})$. On the other hand, if we take $k,j>1$, $\textup{gcd}(k,j)=1$ and $k,j\mid N$, then the empirical spectral distribution of a matrix from an $N\times N$ $\{\textup{$k$-checkerboard, $j$-checkerboard}\}$ consists of one bulk regime of size $\Theta(N)$, one largest blip regime that contains the largest eigenvalue (which is positive) of size $\Theta(N^2)$, and two regimes whose magnitude is in between the bulk regime and the largest blip regime and the number of eigenvalues they contain depends on $k$ and $j$. We call these regimes \textbf{intermediary blip regimes} and their eigenvalues \textbf{intermediary blip eigenvalues}. 

Based on numerical simulation, for $\{\textup{GOE, $k$-checkerboard}\}$, the blip regime contains $2k$ eigenvalues positioned symmetrically around $0$ at $\pm N^{3/2}/k+\Theta(N)$, as illustrated in Figure \ref{GOE5Checkerboard}. On the other hand, for $\{\textup{$k$-checkerboard, $j$-checkerboard}\}$, the largest blip regime contains one eigenvalue at $2N^2/jk+\Theta(N)$; among the two intermediary blip regimes, one regime contains $2k-2$ eigenvalues and is positioned symmetrically around 0 at $\pm \frac{N^{3/2}}{k}\sqrt{1-\frac{1}{j}}+\Theta(N)$, and the other regime contains $2j-2$ eigenvalues and is positioned symmetrically around 0 at $\pm \frac{N^{3/2}}{j}\sqrt{1-\frac{1}{k}}+\Theta(N)$. An example of these blip regimes is illustrated in Figure \ref{Intermediary Blip figure} and \ref{Largest Blip figure}, where we keep the normalization factor for both figures to be $N$ to highlight the difference in the order of $N$ of these regimes. Furthermore, the existence of these blip regimes is established in Appendix \ref{multipleregimes} through Weyl's inequality.

\begin{figure}[ht]
\centering
\includegraphics[width=0.7\textwidth]{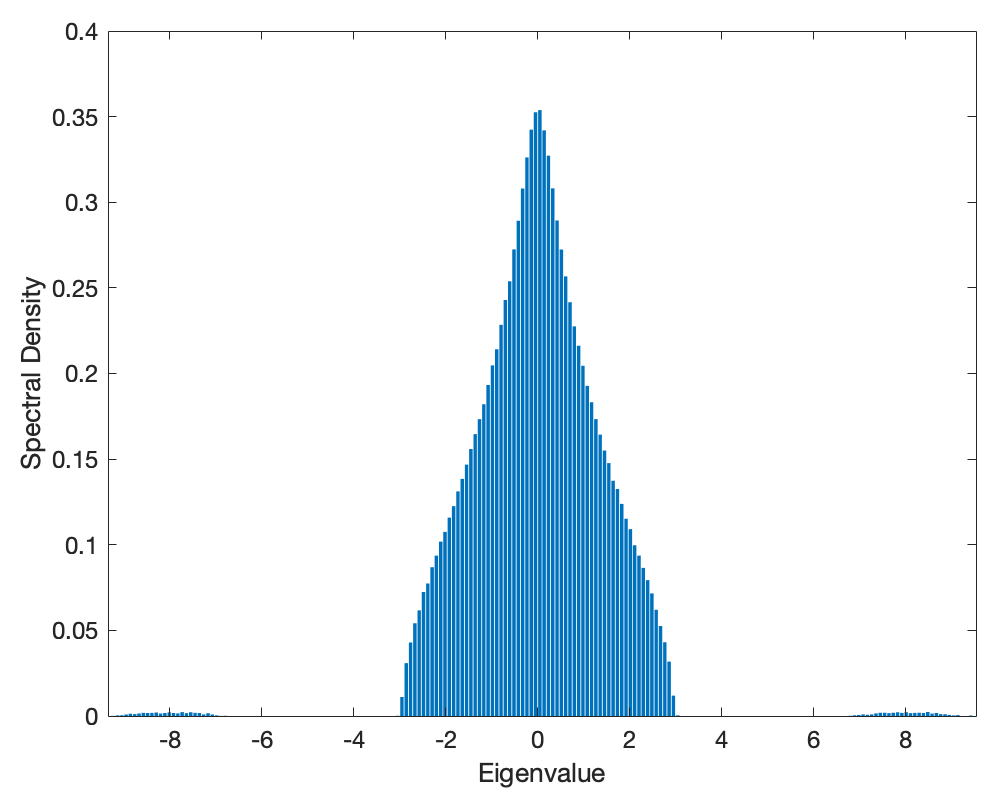}
\caption{Normalized empirical spectral density for one hundred 1500$\times$1500 matrices from $\{\textup{GOE, $5$-checkerboard}\}$. }
\label{GOE5Checkerboard}
\end{figure}

\begin{figure}[ht]
\centering
\includegraphics[width=0.7\textwidth]{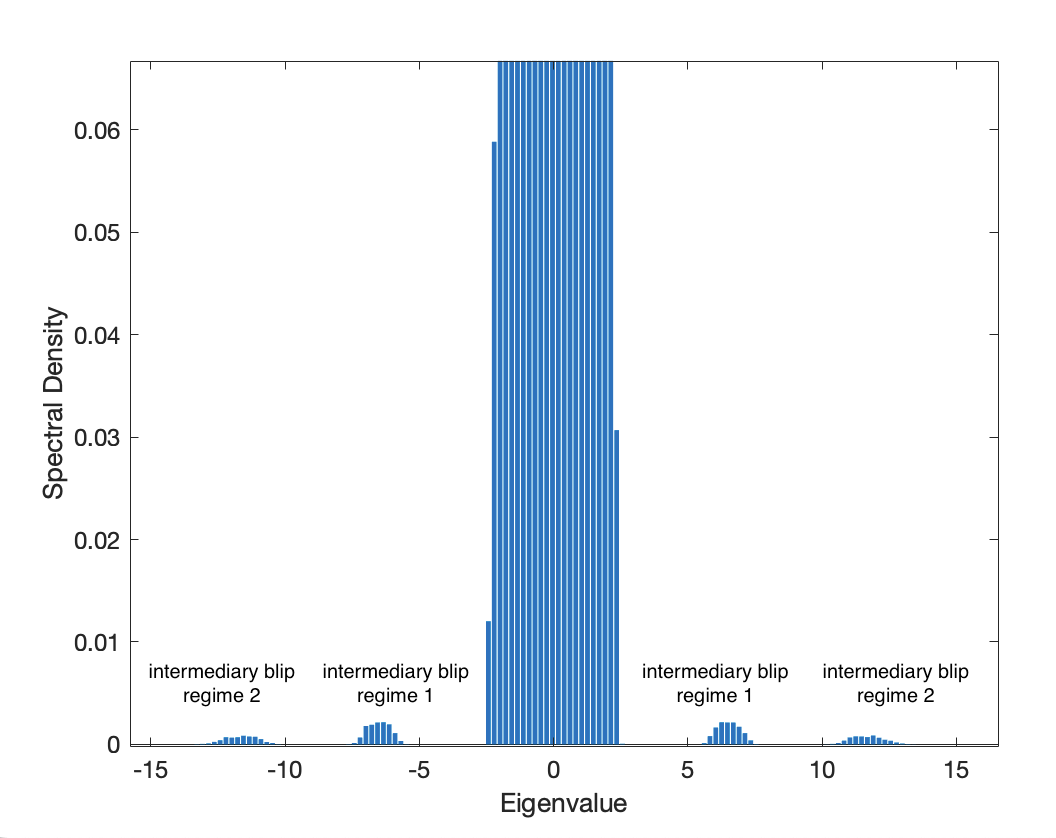}
\caption{Normalized intermediary blip regimes for five hundred 1500$\times$1500 matrices from $\{\textup{$3$-checkerboard, $5$-checkerboard}\}$.}
\label{Intermediary Blip figure}
\end{figure}
\begin{figure}
\includegraphics[width=0.7\textwidth]{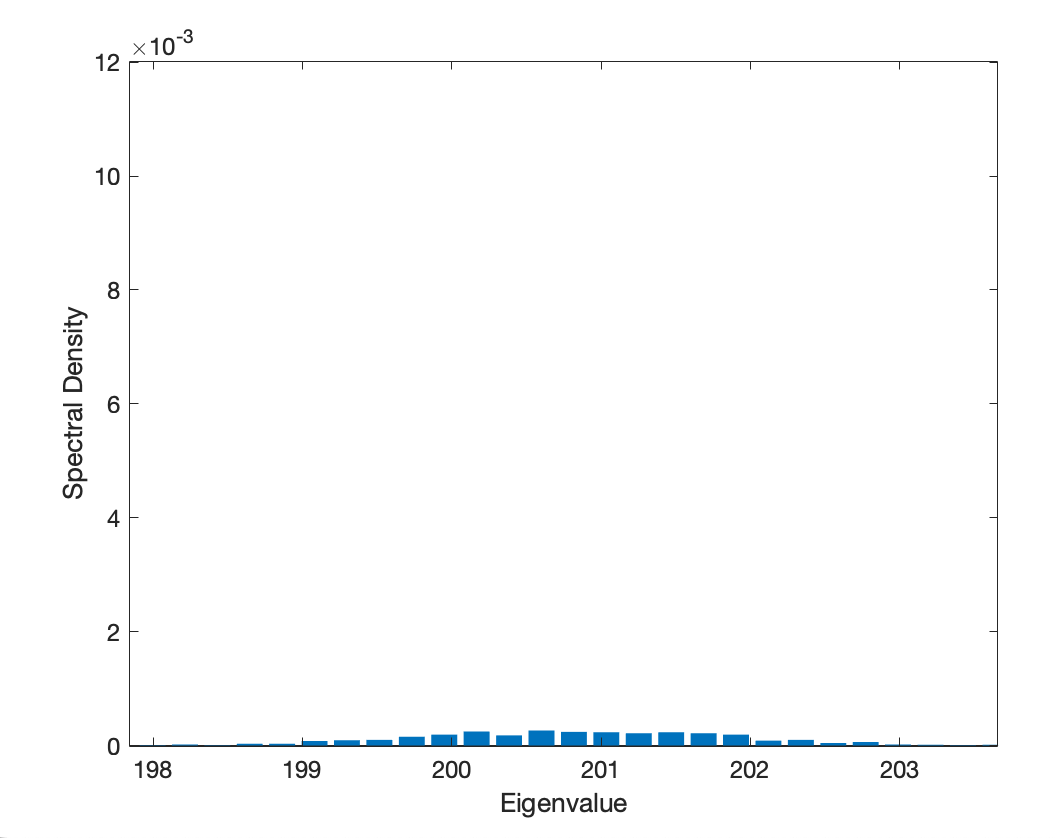}
\caption{Normalized largest blip regime for five hundred 1500$\times$1500 matrices from $\{\textup{$3$-checkerboard, $5$-checkerboard}\}$.}
\label{Largest Blip figure}
\end{figure}

The presence of multiple regimes poses a challenge to analyzing each blip regime. To this end, we introduce suitable weight functions and empirical blip spectral measures to localize at each blip. Then we exploit combinatorics and cancellation techniques to obtain the limiting expected moments of each empirical blip spectral distribution. This method enables us to find the limiting expected moments of the blip regime of $\{\textup{GOE, $k$-checkerboard}\}$ and of the largest blip regime of $\{\textup{$k$-checkerboard, $j$-checkerboard}\}$. However, it fails to work for the intermediary blip regimes as the cancellation techniques cannot differentiate between the contribution from different regimes. We shall discuss the challenge with finding the limiting expected moments of the intermediary blip regimes in more details in Section \ref{sec: anticommutator Combinatorics}. 

\begin{definition}
Let $n=\log\log(N)$ and choose the weight function to be $f^{(2n)}=x^{2n}(2-x)^{2n}$. Suppose that $A_N$ is an $N\times N$ matrix sampled from GOE and $B_N$ is an $N\times N$ matrix sampled from the $k$-checkerboard ensemble, where the samplings occur independently. Then the $f^{(2n)}$-\textbf{weighted empirical blip spectral measure} associated to $\{A_N,B_N\}$ around $\pm N^{3/2}/k$ is
\begin{align}
\mu_{\{A_N,B_N\},f^{(2n)}}(x)dx \ = \ \frac{1}{2k}\sum_{\lambda\textup{ eigenvalues}}f^{(2n)}\left(\frac{k^2\lambda^2}{N^{3}}\right)\delta\left(x-\frac{\lambda^2-N^{3}/k^2}{N^{5/2}}\right)dx.
\end{align}
\end{definition}
\begin{definition}
Let $n=\log\log(N)$ and choose the weight function to be $f^{(2n)}=x^{2n}(2-x)^{2n}$. Suppose that $A_N$ is an $N\times N$ matrix sampled from the $k$-checkerboard ensemble and $B_N$ is an $N\times N$ matrix sampled from the $j$-checkerboard ensemble, where the samplings occur independently. Then the $f^{(2n)}$-\textbf{weighted empirical blip spectral measure} associated to $\{A_N, B_N\}$ around $2N^2/kj$ is
\begin{align}
\mu_{\{A_N,B_N\},f^{(2n)}}(x)dx=\sum_{\lambda\textup{ eigenvalues}}f^{(2n)}\left(\frac{jk\lambda}{2N^2}\right)\delta\left(x-\left(\frac{\lambda-\frac{2}{jk}N^2}{N}\right)\right)dx.
\end{align}
\end{definition}
\begin{theorem} Let $A_N$ be an $N\times N$ matrix sampled from GOE and $B_N$ be an $N\times N$ matrix sampled from the $k$-checkerboard ensemble, where the samplings occur independently. Then the limiting expected $m$\textup{\textsuperscript{th}} moment associated to the empirical blip spectral measure of $\{A_N, B_N\}$ around $\pm N^{3/2}/k$ is
\begin{align}
\lim_{N\rightarrow\infty}\mathbb{E}\left[\mu_{\{A_N,B_N\}}^{(m)}\right] \ = \ \frac{1}{k}\left(\frac{2}{k^2}\right)^{m}\mathbb{E}_k[\textup{Tr }C_k^m],
\end{align}
where $C_k$ is a $k\times k$ GOE and $\mathbb{E}_k[\textup{Tr }C_k^m]$ is taken over all $k\times k$ hollow GOE.
\end{theorem}
\begin{theorem} Let $A_N$ be an $N\times N$ matrix sampled from the $k$-checkerboard ensemble and $B_N$ be an $N\times N$ matrix sampled from the $j$-checkerboard ensemble, where the samplings occur independently. Then the limiting expected $m$\textup{\textsuperscript{th}} moment of the empirical largest blip spectral measure of $\{A_N, B_N\}$ around $2N^2/kj$ is
\begin{align}
\lim_{N\rightarrow\infty}\mathbb{E}\left[\mu_{\{A_N,B_N\}}^{(m)}\right] \ = \ \sum_{\substack{m_{1a}+m_{1b}+m_{2a}+m_{2b}=m; \\ m_{1a},m_{1b}\textup{ even}}}C(m, m_{1a}, m_{2a}, m_{1b}, m_{2b}),
\end{align}
where 
\begin{align}
&C(m, m_{1a}, m_{2a}, m_{1b}, m_{2b}) \ = \ \nonumber \\
&m!\left(\frac{2}{jk}\right)^m\frac{ 2^{\frac{m_{1a}+m_{1b}}{2}-2(m_{2a}+m_{2b})}m_{1a}!!m_{1b}!!}{m_{1a}!m_{1b}!m_{2a}!m_{2b}!}\left(k\sqrt{1-\frac{1}{k}}\right)^{m_{1a}+2m_{2a}}\left(j\sqrt{1-\frac{1}{j}}\right)^{m_{1b}+2m_{2b}}.
\end{align}
\end{theorem}

%%%%%%%%%%%%%%%%%%%%%%%%%%%%%%%%%%%%%%%%%%%%%%%%%%%%%%%%%%%%%%%%%%%%%%%%%%%%%%%%%%%%%%%%%%%%%%%%%%%%%%%%%%%%%%%%%%%%%%%%%%%%%%%%%%%
%%%%%%%%%%%%%%%%%%%%%%%%%%%%%%%%%%%%%%%%%%%%%%%%%%%%%%%%%%%%%%%%%%%%%%%%%%%%%%%%%%%%%%%%%%%%%%%%%%%%%%%%%%%%%%%%%%%%%%%%%%%%%%%%%%%%
%%%%%%%%%%%%%%%%%%%%%%%%%%%%%%%%%%%%%%%%%%%%%%%%%%%%%%%%%%%%%%%%%%%%%%%%%%%%%%%%%%%%%%%%%%%%%%%%%%%%%%%%%%%%%%%%%%%%%%%%%%%%%%%%%%%%
\section{Limiting Expected Moments of Some Anticommutator Ensembles}\label{sec: anticommutator Combinatorics}

In this section, we provide formulae for the limiting expected moments of the following anticommutator ensembles: $\{\textup{GOE, } \textup{GOE}\}$, $\{\textup{PTE, PTE}\}$, $\{\textup{GOE, PTE}\}$, $\{\textup{GOE}, k\textup{-BCE}\}$, and $\{k\textup{-BCE}, k\textup{-BCE}\}$. Along the way, we develop genus expansion formulae for these matrices based on the matching properties of their cyclic products. Since we are always anticommuting two matrices independently drawn from the aforementioned ensembles, every cyclic product in the expected moment expansion contains terms (or entries) from different matrices. To distinguish entries from different matrices, we denote entries from one matrix by $a$ and entries from the other matrix by $b$. We also denote entries from either matrices by $c$ when it is not necessary to specify them or when we have no information about them. We first start with some definitions and results to facilitate our discussion.

\begin{definition}
For a positive integer $n$, let $[2n]:=\{1, 2, \dots, 2n\}$ and $C=c_1c_2\cdots c_{2n}$ be a configuration satisfying $(c_{2s-1},c_{2s})\in\{(a_{2s-1},b_{2s}),(b_{2s-1},a_{2s})\}$ for all $1\leq s\leq n$. Let $\mathcal{C}_{2,2n}$ be the set of all such configurations. Then, a \textbf{partition of $[2n]$ with respect to $C$}, $\pi_C=(V_1,\dots, V_t)$, is a tuple of subsets of $[2n]$ such that the following holds: 
\begin{enumerate}
\item $V_i\neq\emptyset$ for all $1\leq i\leq t$,
\item $V_1\cup\cdots\cup V_t=[2n]$,
\item $V_i\cap V_j=\emptyset$ for $i\neq j$,
\item for all $1\leq i\leq t$ and $i_1,i_2\in V_i$, $\{c_{i_1},c_{i_2}\}\in\{\{a_{i_1},a_{i_2}\},\{b_{i_1},b_{i_2}\}\}$.
\end{enumerate}
Note that (4) ensures that terms from two matrices are matched within themselves (i.e., there is no matching where a term from one matrix is matched with a term from the other matrix). We call $V_1, V_2, \dots, V_t$ \textbf{blocks} of $\pi_C$. Let $\mathcal{P}_C(2n)$ denote the set of all partitions with respect to $C$
of $[2n]$. A partition is called a \textbf{pairing} (or matchings) if each block is of size 2. We denote all the pairings with respect to $C$ of $[2n]$ as $\mathcal{P}_{2,C}(2n)$.
\end{definition}

Since there is a bijective correspondence between a configuration $C=c_1c_2\cdots c_{2n}$ and the set $[2n]$, namely $c_i\leftrightarrow i$, then we can identify a partition of $[2n]$ with respect to $C$ as a partition of the configuration $C$ itself. Note that this definition can be easily extended to any arbitrary subset $S\subseteq [n]$ and configuration $C_S$ (indexed by $S$), where $S$ is not necessarily equal to $[k]$ for any $k\in \mathbb{N}$.

\begin{definition} A partition with respect to $C$, $\pi_C=(V_1,\dots, V_t)$, of $[2n]$ is \textbf{crossing} if there exists blocks $V$ and $W$ with $i, k\in V$ and $j, l\in W$ such that $i<j<k<l$. We denote the set of non-crossing partitions with respect to $C$ of $[2n]$ by $NC_C(2n)$ and the set of non-crossing pairings with respect to $C$ of $[2n]$ by $NC_{2,C}(2n)$.
\end{definition}

We can represent a pairing of the set $[2n]$ with respect to $C$ by drawing lines that connect pairs of numbers from $[2n]$. Then a non-crossing pairing matches $a$'s and $b$'s within themselves and no two lines cross in the diagram. For example, suppose that $C=a_1b_2a_3b_4b_5a_6b_7a_8$. Then $(\{1, 8\}, \{2, 7\}, \{3, 6\}, \{4, 5\})$ is a non-crossing pairing with respect to $C$ and $(\{1, 3\}, \{2, 4\}, \{5, 7\}, $\{6, 8\}$)$ is a crossing pairing with respect to $C$, as shown in Figure \ref{fig:non-crossing-pairing} and \ref{fig:crossing-pairing}.

\begin{figure}[ht]
\centering
\begin{tikzpicture}

% Nodes for numbers
\node (1) at (0,0) {$a_1$};
\node (2) at (1,0) {$b_2$};
\node (3) at (2,0) {$a_3$};
\node (4) at (3,0) {$b_4$};
\node (5) at (4,0) {$b_5$};
\node (6) at (5,0) {$a_6$};
\node (7) at (6,0) {$b_7$};
\node (8) at (7,0) {$a_8$};

% Draw non-crossing arcs
\draw (0,0.2) to[bend left=60] (7,0.2); % 1-8
\draw[red] (1,0.2) to[bend left=60] (6,0.2); % 2-7
\draw (2,0.2) to[bend left=60] (5,0.2); % 3-6
\draw[red] (3,0.2) to[bend left=60] (4,0.2); % 4-5

\end{tikzpicture}
\caption{A non-crossing pairing: $(\{1, 8\}, \{2, 7\}, \{3, 6\}, \{4, 5\})$.}
\label{fig:non-crossing-pairing}
\end{figure}
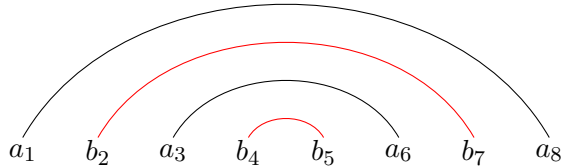

\begin{figure}[ht]
\centering
\begin{tikzpicture}

% Nodes for numbers
\node (1) at (0,0) {$a_1$};
\node (2) at (1,0) {$b_2$};
\node (3) at (2,0) {$a_3$};
\node (4) at (3,0) {$b_4$};
\node (5) at (4,0) {$b_5$};
\node (6) at (5,0) {$a_6$};
\node (7) at (6,0) {$b_7$};
\node (8) at (7,0) {$a_8$};

% Draw non-crossing arcs
\draw (0,0.2) to[bend left=50] (2,0.2); % 1-3
\draw[red] (1,0.2) to[bend left=50] (3,0.2); % 2-4 (red)
\draw (5,0.2) to[bend left=50] (7,0.2); % 6-8
\draw[red] (4,0.2) to[bend left=50] (6,0.2); % 5-7 (red)

\end{tikzpicture}
\caption{A crossing pairing: $(\{1, 3\}, \{2, 4\}, \{5, 7\}, $\{6, 8\}$)$.}
\label{fig:crossing-pairing}
\end{figure}
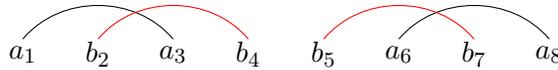

The following formula, known as Wick's formula, provides a connection between expected moments and non-crossing partitions useful to us in the next section:

\begin{proposition}[\cite{MS}] 
Denote by $\mathcal{P}_2(2m)$ the set of pairings of $[2m]$. For each $\pi\in\mathcal{P}_2(2m)$, let $\mathbb{E}_\pi[X_1,\dots, X_{2m}]=\prod_{(r,s)\in\pi}\mathbb{E}[X_rX_s]$. Suppose that $(X_1, \dots, X_n)$ be a real Gaussian random vector. Then
\begin{align}
\mathbb{E}[X_{i_1}, \dots, X_{i_{2m}}] \ = \ \sum_{\pi\in \mathcal{P}_2(2m)}\mathbb{E_\pi}[X_{i_1},\dots, X_{i_{2m}}],
\end{align}
for any $i_1,\dots, i_{2m}\in [n]$.
\end{proposition}

For the remainder of the section, we investigate limiting expected moments of $\{\textup{GOE,GOE}\}$, $\{\textup{PTE, PTE}\}$, $\{\textup{GOE, PTE}\}$, $\{\textup{GOE, $k$-BCE}\}$, $\{k\textup{-BCE},k\textup{-BCE}\}$ one by one. Note that in any expected moment calculation, it is essential to characterize the pairings that do contribute in the limit, as they have proven to considerably simplify the calculation. This requires us to extend the method of genus expansion originally used in the moment calculation of the GUE (see \cite{MS} section 1.7 for more details) to our ensembles.

\subsection{Limiting Expected Moments of the Anticommutator of GOE and GOE}\label{momentGOEGOE}
Let $A_N=(a_{ij})$ and $B_N=(b_{ij})$ be $N\times N$ matrices sampled independent from GOE. We consider the $m$\textsuperscript{th} moment of $\{A_N, B_N\}:=A_NB_N+B_NA_N$:
\begin{align}
M_m(N) \ = \ \frac{1}{N^{m+1}}\mathbb{E}[\textup{Tr}(A_NB_N+B_NA_N)^m] \ = \ \frac{1}{N^{m+1}}\sum_{C\in\mathcal{C}_{2,2m}}\sum_{1\leq i_1,\cdots, i_{2m}\leq N}\mathbb{E}[c_{i_1i_2}c_{i_2i_3}\dots c_{i_{2m}i_1}]\label{GOExGOE eigenvalue trace},
\end{align}
where for the second equality \eqref{GOExGOE eigenvalue trace}, we first expand $\textup{Tr}(A_NB_N+B_NA_N)^m$ using additivity of trace and identify each summand as a configuration in $\mathcal{C}_{2,2m}$ (where we identify $A_N$ as $a$ and $B_N$ as $b$). Then we apply eigenvalue trace lemma to each summand, which gives us the RHS. Now, we apply genus expansion to each $\mathbb{E}[c_{i_1i_2}c_{i_2i_3}\dots c_{i_{2m}i_1}]$. The argument that follows is essentially the same argument as the genus expansion of the $2m$\textsuperscript{th} moment of the GUE, since treating $a$'s and $b$'s both as $c$'s while ensuring that they are matched within themselves preserves the ``non-crossing'' property of pairings that contribute in the limit.

Now, as $N\rightarrow\infty$, we have $M_m(N)\rightarrow 0$ when $m$ is odd, since by standard argument the contribution from each type of configuration is $O(N^{m})$, but the number of types of configurations depends only on $m$. When $m$ is even, if a term in a cyclic product is unmatched, or more than two terms are matched together, then the contribution from such matchings is $O(N^m)$, which is negligible as $N\rightarrow\infty$. So the matchings that contribute in the limit come solely from the pairings of terms in the cyclic products, and we can safely ignore other matchings in the limiting expected moment calculation. By Wick's formula
\begin{align}
\mathbb{E}[c_{i_1i_2}c_{i_2i_3}\cdots c_{i_{2m}i_1}] \ = \ \sum_{
\pi_C\in\mathcal{P}_{2,C}(2m)}\mathbb{E}_{\pi_C}[c_{i_1i_2},c_{i_2i_3},\dots, c_{i_{2m}i_1}].
\end{align}
Since $\mathbb{E}[c_{i_ri_{r+1}}c_{i_si_{s+1}}]=1$ when $i_r=i_{s+1}$ and $i_{r+1}=i_{s}$ and is 0 otherwise given that $i_{r}\neq i_s$ (as we shall see, pairings with $i_r=i_s$ vanish in the limit because they are crossing pairings), then $\mathbb{E}[c_{i_1i_2}c_{i_2i_3}\cdots c_{i_{2m}i_1}]$ is the number of pairings $\pi_C$ with respect to $C$ of $[2m]$ such that $i_r=i_{s+1}$, $i_{r+1}=i_s$, and $a$'s and $b$'s are matched within themselves (i.e., an $a$ is not matched with a $b$). Now, we think of a tuple of indices $(i_1,\dots, i_{2m})$ as a function $i:[2m]\rightarrow [N]$ and write the pairing $\pi_C=\{(r_1,s_1), (r_2, s_2), \dots, (r_{k},s_{m})\}$, as the product of transpositions $(r_1,s_1)(r_2,s_2)\cdots(r_k,s_m)$. We also take $\gamma_{2m}$ to be the cycle $(1, 2, 3, \dots, 2m)$. If $\pi_C$ is a pairing of $[2m]$ and $(r,s)$ is a pair of $\pi_C$, then we express our conditions $i_r=i_{s+1}$ and $i_s=i_{r+1}$ as $i(r)=i(\gamma_{2m}(\pi_C(r)))$ and $i(s)=i(\gamma_{2k}(\pi_C(s)))$ respectively. Hence, $\mathbb{E}_{\pi_C}[c_{i_1i_2}c_{i_2i_3}\cdots c_{i_{2m}i_1}]=1$ if $i$ is constant on the orbits of $\gamma_{2m}\pi_C$ (e.g. $i(r)=i(s+1)$) and $0$ otherwise. Let $\#(\sigma)$ denote the number of cycles of a permutation $\sigma$, then
\begin{align}\label{genus expansion GOExGOE}
M_m(N) \ = \ \frac{1}{N^{m+1}}\mathbb{E}[\textup{Tr}(A_NB_N+B_NA_N)^m] &\ = \ \frac{1}{N^{m+1}}\sum_{C\in \mathcal{C}_{2,2m}}\sum_{\pi_C\in \mathcal{P}_{2,C}(2m)}N^{\#(\gamma_{2m}\pi_C)}.
\end{align}

The following proposition provides a powerful characterization of $\#(\gamma_{2m}\pi)$ based on whether $\pi$ is non-crossing or not. A proof of this proposition can be found in Section 1.8 of \cite{MS}.

\begin{proposition}\label{non-crossing}
If $\pi$ is a pairing of $[2m]$ then $\#(\gamma_{2m}\pi)\leq m-1$ unless $\pi$ is non-crossing, in which case $\#(\gamma_{2m}\pi)=m+1$.
\end{proposition}

Combined with \eqref{genus expansion GOExGOE}, Proposition \ref{non-crossing} intuitively says that for large $N$, $M_m(N)$ is the number of non-crossing pairings with respect to each configuration $C$ summed over all configurations in $\mathcal{C}_{2,2m}$. This leads us to the following lemma.

\begin{lemma}
As $N\rightarrow \infty$,
\begin{align}\label{genusGOEGOE}
M_m \ = \ \lim_{N\rightarrow\infty} M_m(N) \ = \ \sum_{C\in \mathcal{C}_{2,2m}}\sum_{\pi_C\in NC_{2,C}(2m)}1.
\end{align}
\end{lemma}

Given genus expansion formula \eqref{genusGOEGOE}, we are then able to obtain a recurrence relation for the even moment of $\{\textup{GOE, GOE}\}$.

\begin{theorem}\label{GOE-GOE moment recurrence}
The limiting expected $2m$\textup{\textsuperscript{th}} moment $M_{2m}$ of $\{\textup{GOE, GOE}\}$ is given by $M_{2m}=2f(m)$, where $f(0)=f(1)=1$, $g(1)=1$, and for $m\geq 2$,
\begin{align}
f(m) &\ = \ g(m)+2\sum_{j=1}^{m-1}g(j)f(m-j), \\
g(m) &\ = \ 2f(m-1) + \sum_{\substack{0\leq x_1,x_2\leq m-2\\ x_1+x_2\leq m-2}}(1+\mathbbm{1}_{x_1>0})(1+\mathbbm{1}_{x_2>0})f(x_1)f(x_2)g(m-1-x_1-x_2).
\end{align}
\end{theorem}

\begin{proof} Let $f(m)$ be the number of non-crossing pairings with respect to all configurations in $\mathcal{C}_{2,4m}$ starting with an $a$, and $g(m)$ be the number of non-crossing pairings with respect to all configurations in $\mathcal{C}_{2,4m}$ starting and ending with an $a$ such that these two $a$'s are matched together (i.e., a configuration $a_{i_1i_2}b_{i_2i_3}\cdots \\b_{i_{4m-1}i_{4m}}a_{i_{4m}i_1}$ with $i_{4m}=i_2$).

We first find the recurrence relation for $f(m)$. We know that $a_{i_1i_2}$ is matched with some $a_{i_{4j}i_{4j+1}}$ with $j\leq m$ (in the case when $j=m$, we identify $4m+1$ as $1$) since there should be an even number of both $a$ and $b$ terms between $a_{i_1i_2}$ and $a_{i_{4j}i_{4j+1}}$ to ensure non-crossing pairings. When $j=m$, the number of non-crossing pairings is just $g(m)$ by definition. When $j<m$, the number of non-crossing pairings within $a_{i_1,i_2}b_{i_2,i_3}\cdots b_{i_{4j-1}i_{4j}}a_{i_{4j}i_{4j+1}}$ is $g(j)$. We multiply this by the number of non-crossing pairings within the rest of the cyclic product which have no restrictions and is therefore simply $2f(m-j)$, with the 2 accounting for starting with either an $a$ or $b$. Thus, summing over all possible $j$'s, we have
\begin{align}
f(m) \ = \ g(m)+2\sum_{j=1}^{m-1}g(j)f(m-j).
\end{align}

Similarly, we know that either $b_{i_2 i_3}$ is matched with $b_{i_{4m-1}i_{4m}}$, or $b_{i_2 i_3}$ is matched with $b_{i_{4x_1+3}i_{4x_1+4}}$ and $b_{i_{4m-1}i_{4m}}$ is matched with $b_{i_{4m-4x_2-2}i_{4m-4x_2-1}}$, with $4x_1+4<4m-4x_2-2$, or $x_1+x_2\leq m-2$. In the first case, since there are no restrictions on the $4m-4$ terms between $b_{i_2 i_3}$ and $b_{i_{4m-1}i_{4m}}$, the number of non-crossing pairings is $2f(m-1)$. In the second case, the number of non-crossing pairings of terms between $b_{i_2 i_3}$ and $b_{i_{4x_1+3}i_{4x_1+4}}$ is $(1+\mathbbm{1}_{x_1>0})f(x_1)$, the number of non-crossing pairings of terms between $b_{i_{4m-4x_2-2}i_{4m-4x_2-1}}$ and $b_{i_{4m-1}i_{4m}}$ is $(1+\mathbbm{1}_{x_2>0})f(x_2)$, with the indicator functions accounting for the intermediary terms starting with either an $a$ or a $b$, and lastly the number of non-crossing pairings of terms between $b_{i_{4x_1+3}i_{4x_1+4}}$ and $b_{i_{4m-4x_2-2}i_{4m-4x_2-1}}$ is $g(m-1-x_1-x_2)$. Thus, 
\begin{align}
    g(m) \ = \ 2f(m-1) + \sum_{\substack{0\leq x_1,x_2\leq m-2\\ x_1+x_2\leq m-2}}(1+\mathbbm{1}_{x_1>0})(1+\mathbbm{1}_{x_2>0})f(x_1)f(x_2)g(m-1-x_1-x_2).
\end{align}

We have now defined our recurrence for $f(m)$, which counts the number of non-crossing pairings with respect to configurations in $\mathcal{C}_{2,4m}$ starting with an $a$. Since general configurations in $\mathcal{C}_{2,4m}$ can start with either an $a$ or a $b$, we multiply $f(m)$ by $2$ to get all possible non-crossing pairings of configurations in $\mathcal{C}_{2,4m}$, and we arrive at the even moments being $M_{2m}=2f(m)$.
\end{proof} 

A natural extension of the anticommutator $\{A_N,B_N\}$ is the \textbf{$\ell$-anticommutator} of $\ell$ matrix ensembles $A^{(1)}_N,A^{(2)}_{N},\dotsc, A^{(\ell)}_{N}$, defined as
\begin{align}
\{A^{(1)}_N, A^{(2)}_N, \dotsc, A^{(\ell)}_N\} \ := \ \sum_{\sigma\in S_\ell} A^{(\sigma(1))}_N A^{(\sigma(2))}_N \cdots A^{(\sigma(\ell))}_N,
\end{align}
where $S_\ell$ is the symmetric group of order $\ell$. An example of $3$-anticommutator and $4$-anticommutator is illustrated in Figure \ref{3Anticommutator} and \ref{4Anticommutator}.

\begin{figure}[ht]
\centering
\begin{minipage}[b]{0.4\textwidth}
\includegraphics[width=\textwidth]{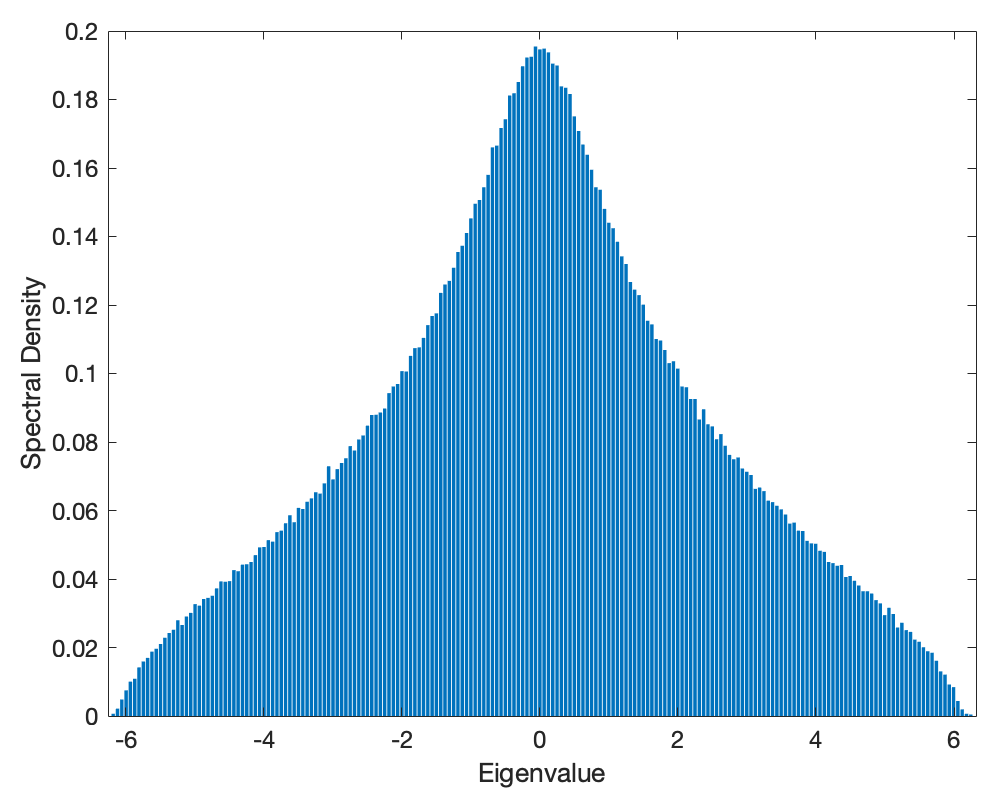}
\caption{Normalized empirical spectral density for one hundred 1500$\times$1500 matrices from $3$-anticommutator.}
\label{3Anticommutator}
\end{minipage}
\begin{minipage}[b]{0.4\textwidth}
\includegraphics[width=\textwidth]{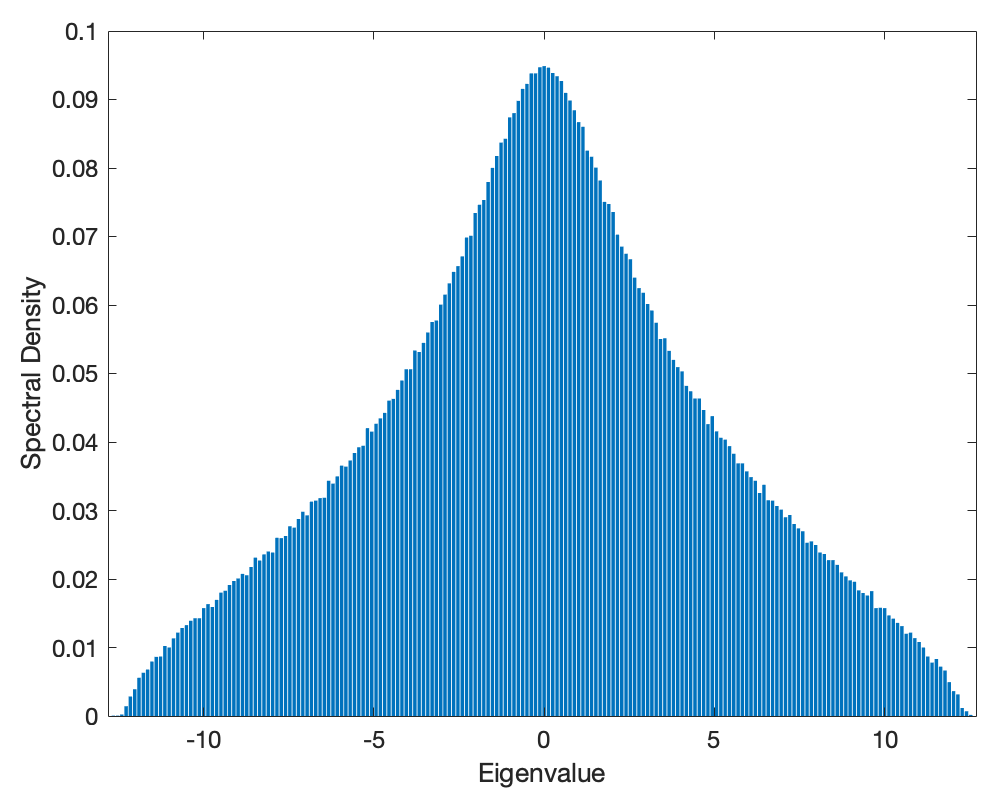}
\caption{Normalized empirical spectral density for one hundred 1500$\times$1500 matrices from $4$-anticommutator.}
\label{4Anticommutator}
\end{minipage}
\end{figure}
Based on numerical simulation, the spectrum of $\ell$-anticommutator is $\Theta(N^{\ell/2})$. Even after normalization, the support of the spectrum still increases as $\ell$ increases, and it seems that the support tends to infinity as $\ell$ tends to infinity. By employing the same method as in the proof of Lemma \ref{GOE-GOE moment recurrence}, we are able to obtain a recurrence relation for the limiting expected moments of the $\ell$-anticommutator of GOE. Now, however, instead of two interdependent recurrence relations, we have $\ell$ interdependent recurrence relations. We leave the result and proof in Appendix
\ref{AppendixMomentsAnti}.

Next, we use generating functions to find the limiting spectral distribution of $\{\textup{GOE, GOE}\}$ from the recurrence relations given in Lemma \ref{GOE-GOE moment recurrence}. We start by introducing some definitions and facts about generating functions, which is also useful for finding the spectral distribution of $\{\textup{GOE, PTE}\}$ later on.

\begin{definition}
For a sequence $\{a_n\}_{n=0}^\infty$ of real numbers, the \textbf{ordinary generating function} $A(z)$ of $\{a_n\}_{n=0}^\infty$ is defined as
\begin{align}
A(z) \ = \ \sum_{n=0}^\infty a_nz^n.
\end{align}
\end{definition}

\begin{definition}
For a doubly-indexed sequence $\{a_{n,m}\}_{n,m\geq 0}$ of real numbers, the \textbf{exponential bivariate generating function} $A(z,w)$ of $\{a_{n,m}\}_{n,m\geq 0}$ is defined as
\begin{align}
A(z,w) \ = \ \sum_{n,m\geq 0}a_{n,m}\frac{z^{n}}{n!}w^m.
\end{align}
\end{definition}

\begin{definition}
The \textbf{convolution} or the \textbf{Cauchy product} of two ordinary generating functions $A(z)=\sum_{n=0}^\infty a_nz^n$ and $B(z)=\sum_{n=0}^\infty b_nz^n$ is defined as $C(z)=A(z)B(z)$. The coefficient of $z^n$ of $C(z)$, denoted as $[z^n]C(z)$, is given by
\begin{align}
[z^n]C(z) \ = \ \sum_{k=0}^na_kb_{n-k}.
\end{align}
\end{definition}

\begin{theorem}
The limiting expected $2m$\textup{\textsuperscript{th}} moment $M_{2m}$ of $\{\textup{GOE, GOE}\}$ is given by the explicit formula
\begin{align}
M_{2m} \ = \ \frac{1}{m}\sum_{k=1}^m2^k\binom{2m}{k-1}\binom{m}{k}.
\end{align}
\end{theorem}

\begin{proof}
By Theorem \ref{GOE-GOE moment recurrence}, $M_{2m}=2f(m)$, where $g(0)=f(0)=1$ and for $m\geq 2$,
\begin{align}
f(m) &\ = \  g(m-1)+2\sum_{j=1}^{m-1}g(j-1)f(m-j), \label{f_recurrence_GOE_GOE}\\
g(m-1) &\ = \ 2f(m-1) + \sum_{\substack{0\leq x_1,x_2\leq m-2\\ x_1+x_2\leq m-2}}(1+\mathbbm{1}_{x_1>0})(1+\mathbbm{1}_{x_2>0})f(x_1)f(x_2)g(m-2-x_1-x_2).\label{g_recurrence_GOE_GOE}
\end{align}
Note that we have shifted $g$ relative to Theorem \ref{GOE-GOE moment recurrence} to make $g$ zero-indexed. Let $F(z),G(z),$ and $\tilde{F}(z)$ be the ordinary generating functions of $\{f(m)\}_{m=0}^\infty,\{g(m)\}_{m=0}^\infty,$ and $\{f(m)(1+\mathbbm{1}_{m>0})\}_{m=0}^\infty$ respectively. Note that here, $\tilde{F}(0)$ is set to be 1, and for $m\geq 1$, $M_{2m}$ is the $m$\textup{\textsuperscript{th}} coefficient of $\tilde{F}$. From \eqref{f_recurrence_GOE_GOE}, we have
\begin{align}
f(m+1) \ = \ -g(m)+2\sum_{j=0}^{m}g(j)f(m-j).
\end{align}
Multiplying the above expression by $z^m$ on both sides and summing over all $m\geq 0$ gives us the relation
\begin{align}
\frac{F(z)-1}{z} \ = \ -G(z)+2G(z)F(z).
\end{align}
Solving $G(z)$ in terms of $F(z)$, we have
\begin{align}
G(z) \ = \ \frac{F(z)-1}{z(2F(z)-1)}.\label{LHS_F_functional_equation}
\end{align}
On the other hand, from \eqref{g_recurrence_GOE_GOE}, we have
\begin{align}
g(m+1) &\ = \ 2f(m+1)+\sum_{\substack{0\leq x_1,x_2\leq m\\ x_1+x_2\leq m}}(1+\mathbbm{1}_{x_1>0})(1+\mathbbm{1}_{x_2>0})f(x_1)f(x_2)g(m-x_1-x_2)\nonumber\\
&\ = \ 2f(m+1)+\sum_{x_1=0}^{m}(1+\mathbbm{1}_{x_1>0})f(x_1)\sum_{x_2=0}^{m-x_1}f(x_2)(1+\mathbbm{1}_{x_2>0})g(m-x_1-x_2).\label{shifted_g_recurrence_GOE_GOE}
\end{align}
Multiplying the above expression by $z^m$ on both sides and summing over $m\geq 0$, we recognize that the double sum turns into two convolutions that yields $(2F(z)-1)^2G(z)$. Hence, \eqref{shifted_g_recurrence_GOE_GOE} gives the relation
\begin{align}
\frac{G(z)-1}{z} \ = \ 2\left(\frac{F(z)-1}{z}\right)+(2F(z)-1)^2G(z).
\end{align}
Solving $G(z)$ in terms of $F(z)$, we have
\begin{align}
G(z) \ = \ \frac{2F(z)-1}{1-z(2F(z)-1)^2}.\label{RHS_F_functional_equation}
\end{align}
So \eqref{LHS_F_functional_equation} and \eqref{RHS_F_functional_equation} together give the functional equation for $F(z)$
\begin{align}
\frac{F(z)-1}{z(2F(z)-1)} \ = \ \frac{2F(z)-1}{1-z(2F(z)-1)^2}.
\end{align}
Using the relation $\tilde{F}(z)=2F(z)-1$, we obtain
\begin{align}
\tilde{F}(z) \ = \ 1+2z\left(\frac{\tilde{F}(z)^2}{1-z\tilde{F}(z)^2}\right).\label{final_functional_equation_GOE_GOE}
\end{align}
Multiplying \eqref{final_functional_equation_GOE_GOE} by $1-z\tilde{F}(z)^2$, we see that $\tilde{F}$ also satisfies the functional equation
\begin{align}
\tilde{F}(z) \ = \ 1+z\left(\tilde{F}(z)^2+\tilde{F}(z)^3\right).
\end{align}
The recurrence relation induced by this functional equation is uniquely determined by the choice of the first term \cite{BS}. Since $\Tilde{F}(0)=1$, uniqueness together with \cite{BS} gives that the sequence of coefficients of $\tilde{F}$ is OEIS sequence A027307, the number of walks from $(0,0)$ to $(3m,0)$ on or above the $x$-axis with steps $(2,1)$, $(1,2)$ and $(1,-1)$. By \cite{YJ}, these are the $3$-Schr\"oder numbers $r_m^{(3)}$, which admit the explicit formula
\begin{align}
r_m^{(3)} \ = \ \frac{1}{m}\sum_{k=1}^m2^k\binom{2m}{k-1}\binom{m}{k}\label{3_schroder_number_formula}
\end{align}
via Theorem 2.4 of \cite{YJ}. Thus, $M_{2m}=[z^m]\tilde{F}(z)=r_m^{(3)}$, so \eqref{3_schroder_number_formula} gives the explicit formula for $M_{2m}$ as desired.
\end{proof}

\begin{cor}
The limiting spectral density of $\{\textup{GOE, GOE}\}$ is
\begin{align}
\mu(x) \ = \ -\frac{\sqrt{3}}{2\pi|x|}\left(\frac{3x^2+1}{9h(x)}-h(x)\right),\quad |x| \ \leq \ \sqrt{\frac{11+5\sqrt{5}}{2}},
\end{align}
where
\begin{align}
h(x) \ = \ \sqrt[3]{\frac{18x^2+1}{27}+\sqrt{\frac{x^2(1+11x^2-x^4)}{27}}}.
\end{align}
\end{cor}
\begin{proof}
By \cite{commutator}, the ordinary generating function of the limiting expected moments of $\{\textup{GOE, GOE}\}$ is given by
\begin{align}
M_0(z) \ = \ \left(\frac{y}{\left(2+y\right)(1+y)^2}\right)^{-1}\left(z^2\right),
\end{align}
where $(\cdot)^{-1}$ denotes inverse under composition of a formal power series. Here, we follow the convention used in \cite{commutator} that the constant term of an ordinary generating function is set to be $0$. This gives us
\begin{align*}
\frac{M_0(z)}{\left(2+M_0(z)\right)(1+M_0(z))^2} \ = \ z^2.
\end{align*}
Solving for $M_0(z)$, we find that $M_0(z)$ satisfies the functional equation
\begin{align*}
z^2M_0(z)^3+4z^2M_0(z)^2+(5z^2-1)M_0(z)+2z^2 \ = \ 0.
\end{align*}
On the other hand, we know that $\tilde{F}(z)=1+z(\tilde{F}(z)^2+\tilde{F}(z)^3)$ and the functional equation admits an unique solution. Letting $M(z):=\tilde{F}(z^2)-1$, we find that $M(z)$ satisfies the functional equation
\begin{align}\label{MGF functional eq}
M(z)+1 \ = \ 1+z^2\left((M(z)+1)^2+(M(z)+1)^3\right),
\end{align}
which also admits an unique solution. After simplifying \eqref{MGF functional eq}, we arrive the same functional equation as \eqref{MGF functional eq}
\begin{align}
z^2M(z)^3+4z^2M(z)^2+(5z^2-1)M(z)+2z^2 \ = \ 0.
\end{align}
Since $M_0,M$ satisfy the same functional equation and have null constant terms, it follows that $M(z)=M_0(z)$. By \cite{commutator}, $M(z)$ corresponds to the density function
\begin{align}
\mu(x)dx \ = \ -\frac{\sqrt{3}}{2\pi|x|}\left(\frac{3x^2+1}{9h(x)}-h(x)\right)dx,\quad |x| \ \leq \ \sqrt{\frac{11+5\sqrt{5}}{2}},
\end{align}
where
\begin{align}
h(x) \ = \ \sqrt[3]{\frac{18x^2+1}{27}+\sqrt{\frac{x^2(1+11x^2-x^4)}{27}}}.
\end{align}
\end{proof}

\subsection{Limiting Expected Moments of the Anticommutator of PTE and PTE}\label{momentsPTPT} The palindromic Toeplitz ensemble (PTE) was introduced by Massey-Miller-Sinsheimer in \cite{palindromicToeplitz} to remove the Diophantine obstruction encountered in the moment calculation of Toeplitz ensemble in \cite{Toeplitz}, where the limiting spectral distribution is almost an Gaussian. Essentially, the additional symmetry in the structure of PTE allows almost all pairings to have consistent choices of indexing and contribute in the limit (the pairings that don't have consistent choice of indexing are negligible in the limit). Through this fact, they showed that the limiting expected $2m$\textsuperscript{th} moment of the PTE is $(2m-1)!!$, which is exactly the $2m$\textsuperscript{th} moment of standard Gaussian, and hence the spectral distribution of PTE converges almost surely to Gaussian. The limiting expected moment calculation of PTE can be naturally extended to that of $\{\textup{PTE, PTE}\}$. By standard argument, the expected odd moments of $\{\textup{PTE, PTE}\}$ vanish in the limit. For expected even moments, we can view each cyclic product in the expected $2m$\textsuperscript{th} moment of $\{\textup{PTE, PTE}\}$ as a cyclic product in the expected $4m$\textsuperscript{th} moment of PTE. Thus, the matching in each cyclic product is again free, that is, almost all pairings have consistent choices of indexing and those pairings that don't have consistent choices of indexing are negligible in the limit. Thus, this gives us the genus expansion formula:
\begin{align}\label{genusPalindromicPalindromic}
M_m \ = \ \lim_{N\rightarrow \infty}M_m(N) \ = \ \sum_{C\in\mathcal{C}_{2,2m}}\sum_{\pi_C\in P_{2,C}(2m)}1.
\end{align}
Intuitively, this is the statement that $M_m$ is equal to the number of pairings with respect to configuration $C$ summed over all configurations in $\mathcal{C}_{2,2m}$.

\begin{theorem}\label{PTEPTEmomenT}
The limiting expected $2m$\textup{\textsuperscript{th}} moment $M_{2m}$ of $\{\textup{PTE, PTE}\}$ is $2^{2m}((2m-1)!!)^2$.
\end{theorem}
\begin{proof}
Recall that a valid configuration $C=c_{1}c_2\cdots c_{4m}$ satisfy $(c_{2s-1},c_{2s})\in \{(a_{2s-1},b_{2s}),(b_{2s-1},a_{2s})\}$, then the number of valid configurations is $2^{2m}$. Moreover, each configuration has $(2m-1)!!$ ways of matching up the $a$'s and $(2m-1)!!$ ways of matching up the $b$'s, then by \eqref{genusPalindromicPalindromic} we have
\begin{align}\label{PTEPTEmomentmoment}
M_{2m} \ = \ 2^{2m}((2m-1)!!)^2.
\end{align}
\end{proof}
In what follows, we provide two proofs of Corollary \ref{convolutiondensity}. The first one is based on the relationship between chi-square distribution and normal distribution, and the second one makes use of moment generating functions and novel integration techniques. Note that from now on, we write $X\sim \mu$ if random variable $X$ follows probability distribution $\mu$. Alternatively, if $X\sim \mu$ and $Y\sim \mu$, then we also write $X\sim Y$. We start by defining $\chi_k^2$ distribution and stating some of its properties.
\begin{definition}
\textbf{Chi-squared distribution with $k$ degree of freedom}, denoted as $\chi_k^2$, is characterized by the probability density function
\begin{align}
f(x) \ = \ \frac{1}{2^{k/2}\Gamma(k/2)}x^{k/2-1}e^{-x/2},
\end{align}
where $\Gamma(x)$ is the Gamma function.
\end{definition}
There is an interesting link between $\chi_1^2$ and $N(0,1)$: suppose that $X\sim N(0,1)$, $Y\sim \chi_1^2$, and $S$ is a random sign. Then $X\sim S\sqrt{Y}$. This is crucial to the first proof of Corollary \ref{convolutiondensity}.
\begin{cor}\label{convolutiondensity}
The limiting spectral density of $\{\textup{PTE, PTE}\}$ is given by the convolution
\begin{align}
\mu(x) \ = \ f(x) * g(x),
\end{align}
where $f(x)$ is the probability density function of $\chi_1^2$ and $g(x)$ is the probability density function of $-\chi_1^2$, the negative of a chi-squared random variable with 1 degree of freedom.
\end{cor}
\begin{proof}[Proof 1 of Corollary \ref{convolutiondensity}]
We know that odd moments of normal $N(0,2)$ distribution are 0 and the $2m$\textsuperscript{th} moment of normal $N(0,2)$ distribution is $2^m(2m-1)!!$. Suppose that $X,Y$ are i.i.d. random variables $\sim N(0,2)$. Then we have
\begin{align}
\mathbb{E}[XY] \ = \ \mathbb{E}[X]\mathbb{E}[Y] \ = \ (2^m(2m-1)!!)^2 \ = \ 2^{2m}((2m-1)!!)^2. 
\end{align}
These match with the limiting expected moments of $\{\textup{PTE,PTE}\}$. We will establish in Theorem \ref{section2Bulk} that the empirical spectral distribution of $\{\textup{PTE, PTE}\}$ converges almost surely to that of the product of two i.i.d. random variables $\sim N(0,2)$. On the other hand, suppose that $Y_1,Y_2$ are i.i.d. random variables $\sim \chi_1^2$ and $S_1,S_2$ are i.i.d. random signs. We know that $S_1\sqrt{Y_1}\sim N(0,1)$, $S_2\sqrt{Y_2}\sim N(0,1)$. Moreover, $S_1\sqrt{Y_1}$ and $S_2\sqrt{Y_2}$ are independent, then $Y_1-Y_2=(S_1\sqrt{Y_1}+S_2\sqrt{Y_2})\cdot (S_1\sqrt{Y_1}-S_2\sqrt{Y_2})$ is the product of two i.i.d. random variables $\sim N(0,2)$. Hence, the limiting spectral distribution of $\{\textup{PTE, PTE}\}$ agrees with that of $Y_1-Y_2$. Furthermore, we know that the probability density function of the sum of two independent random variables is the convolution of the probability density function of the two random variables. This proves the claim.
\end{proof}
On the other hand, proof 2 of Corollary \ref{convolutiondensity} makes use of the fact that random variables whose moment generating functions agree on a small neighborhood of 0 has the same distribution. Then, it suffices to prove that the moment generating function of the limiting spectral distribution of $\{\textup{PTE, PTE}\}$, which is $M(z):=1+\sum_{m=1}^\infty 2^{2m}((2m-1)!!)^2\frac{z^{2m}}{(2m)!}$, agrees with the moment generating function $Y_1-Y_2$, where $Y_1, Y_2$ are i.i.d. $\sim \chi_1^2$, on a small neighborhood of 0.

\begin{lemma}
The moment generating function  $M(z):=1+\sum_{m=1}^\infty 2^{2m}((2m-1)!!)^2\frac{z^{2m}}{(2m)!}$ of the limiting spectral distribution of $\{\textup{PTE, PTE}\}$ converges on the interval $(-\frac{1}{2},\frac{1}{2})$.
\end{lemma}
\begin{proof}
We use the transformation $z^2\mapsto x$. It suffices to prove that $F(\sqrt{x})=1+\sum_{m=1}^\infty 2^{2m}((2m-1)!!)^2\frac{x^{m}}{(2m)!}$ converges on $[0,\frac{1}{4})$. After the transformation, the $m$\textsuperscript{th} term of the series is $a_m:=2^{2m}((2m-1)!!)^2\frac{x^m}{(2m)!}$. Then,
\begin{align}
\lim_{m\rightarrow\infty}\left|\frac{a_{m+1}}{a_m}\right|&\ = \ \lim_{m\rightarrow\infty} \frac{2^{2(m+1)}((2(m+1)-1)!!)^2}{2^{2m}((2m-1)!!)^2} \times \frac{(2m)!}{(2(m+1))!}|x| \nonumber\\
&\ = \ \lim_{m\rightarrow\infty} 4\times \frac{(2(m+1)-1)^2}{2(m+1)(2(m+1)-1)}|x| \nonumber\\
&\ = \ 4|x|.
\end{align}
By the ratio test, when $4|x|<1$, the series converges absolutely. Thus, $F(\sqrt{x})$ converges on $[0,\frac{1}{4})$, as desired.
\end{proof}

To find a closed-form expression of $M(z)$, we employ the following calculus fact.
\begin{prop}\cite{MT-B}\label{trignometric integral}
For even $n> 0$
\begin{align}
\int_{0}^{\frac{\pi}{2}}\sin^n(x)dx \ = \ \frac{\pi}{2}\times \frac{(n-1)!!}{n!!}.
\end{align}
\end{prop}

\begin{lemma}\label{MGFcomp}
For $z\in (-\frac{1}{2},\frac{1}{2})$,
\begin{align}
M(z) \ = \ \frac{1}{\sqrt{1-4z^2}} \ = \ \frac{1}{\sqrt{1+2z}}\cdot\frac{1}{\sqrt{1-2z}}.
\end{align}
\end{lemma}
\begin{proof}
First, we observe that
\begin{align}
M(z) \ = \ 1+\sum_{m=1}^\infty 2^{2m}((2m-1)!!)^2\frac{z^{2m}}{(2m)!} \ = \ 1+\sum_{m=1}^\infty \frac{(2m-1)!!}{(2m)!!}(2z)^{2m}.
\end{align}
By Proposition \ref{trignometric integral},
\begin{align}
M(z) \ = \ 1+\frac{2}{\pi}\sum_{m=1}^\infty \int_{0}^{\frac{\pi}{2}}(2z\sin(x))^{2m}dx \ = \ \frac{2}{\pi}\sum_{m=0}^\infty \int_{0}^{\frac{\pi}{2}}(2z\sin(x))^{2m}dx.
\end{align}
Note that since $z\in (-\frac{1}{2},\frac{1}{2})$, $|2z\sin(x)|<1$, then $\sum_{m=0}^\infty (2z\sin(x))^{2m}$ as a geometric series converges uniformly to $\frac{1}{1-4z^2\sin^2(x)}$. Thus we can exchange the order of the limit and the integration sign, and we have
\begin{align}
M(z) &\ = \ \frac{2}{\pi}\int_0^{\frac{\pi}{2}}\frac{1}{1-4z^2\sin^2(x)}dx \nonumber\\
&\ = \ \frac{2}{\pi}\int_{0}^{\frac{\pi}{2}}\frac{\sec^2(x)}{\sec^2(x)-4z^2\tan^2(x)}dx.
\end{align}
Using $\sec^2(x)-\tan^2(x)=1$, we have
\begin{align}
M(z) &\ = \ \frac{2}{\pi}\int_{0}^{\frac{\pi}{2}}\frac{\sec^2(x)}{1+(1-4z^2)\tan^2(x)}dx.
\end{align}
Let $u=\tan(x)$. Then $du=\sec^2(x)dx$. Thus,
\begin{align}
M(z) &\ = \ \frac{2}{\pi}\int_{0}^\infty \frac{du}{1+(1-4z^2)u^2} \nonumber \\
&\ = \ \frac{2}{\pi}\frac{1}{\sqrt{1-4z^2}}\arctan((\sqrt{1-4z^2})u)\Bigg|_{0}^{\infty} \nonumber \\
&\ = \ \frac{1}{\sqrt{1-4z^2}} \ = \ \frac{1}{\sqrt{1+2z}}\cdot\frac{1}{\sqrt{1-2z}}.
\end{align}
\end{proof}
\begin{proposition}
The moment generating function $F(z)$ of $\chi_1^2$ is $\frac{1}{\sqrt{1-2z}}$.
\end{proposition}
\begin{proof}
This follows from straightforward computation of $F(z)=\mathbb{E}\left[e^{zY}\right]$, where $Y\sim \chi_1^2$.
\end{proof}
Suppose that $Y_1, Y_2$ are i.i.d. $\sim \chi_1^2$. Then the moment generating function of $Y_1-Y_2$ is given by
\begin{align}
\mathbb{E}\left[e^{z(Y_1-Y_2)}\right] \ = \ \mathbb{E}\left[e^{(-z)Y_2}\right]\mathbb{E}\left[e^{zY_1}\right] \ = \ \frac{1}{\sqrt{1+2z}}\cdot \frac{1}{\sqrt{1-2z}}.
\end{align}
This agrees with the moment generating function of the limiting spectral distribution of $\{\textup{PTE, PTE}\}$ by Lemma \ref{MGFcomp}. We know that the probability density function of $Y_1-Y_2$ is given by $\mu(x)dx=f(x)*g(x)dx$ as stated in Corollary \ref{convolutiondensity}. Then proof 2 of Corollary \ref{convolutiondensity} is completed by the following fact about moment generating functions.
\begin{proposition}

[\cite{Bil}]\label{uniqueness of moment generating function}
Let $Y_1$ and $Y_2$ be random variables with the same moment generating function $M(z)$ in some neighborhood of $0$, then $Y_1$ and $Y_2$ have the same distribution.
\end{proposition}

\subsection{Limiting Expected Moments of the Anticommutator of GOE and PTE}\label{subsectionGOEpalindromicToeplitz}
So far, we've only been looking at \textbf{homogeneous} anticommutator ensembles $\{A_N, B_N\}$, i.e., $A_N$ and $B_N$ are matrices from the same ensemble. Genus expansions of $\{\textup{GOE, GOE}\}$ and $\{\textup{PTE, PTE}\}$ suggest that in general, genus expansion of a homogeneous anticommutator ensemble $\{A_N, B_N\}$ is a straightforward generalization of the genus expansion of $A_N$ (or $B_N$). A natural question to ask is: what does genus expansion of an \textbf{inhomogeneous} anticommutator ensembles $\{A_N, B_N\}$ (i.e., $A_N$ and $B_N$ are matrices from different ensembles) look like? In this section, we turn our attention to an inhomogeneous anticommutator ensemble, namely $\{\textup{GOE, PTE}\}$. Interestingly, we see that the matching properties of GOE and PTE are well preserved under the anticommutator operator, that is, the contributions to the expected moments of $\{\textup{GOE, PTE}\}$ in the limit come solely from non-crossing matchings of the GOE terms and free matchings of the PTE terms that don't cross the matchings of the GOE terms.

Let $A_N$ be a matrix sampled from GOE and $B_N$ be a matrix sampled from PTE. Similar to the previous examples, we have that the $m$\textsuperscript{th} moment of $\{A_N, B_N\}$ is given by
\begin{align}\label{expansionGOEPT}
M_m(N) \ = \ \frac{1}{N^{m+1}}\sum_{C\in\mathcal{C}_{2, 2m}}\sum_{1\leq i_1,\dots, i_{2m}\leq N}\mathbb{E}[c_{i_1i_2}c_{i_2i_3}\cdots c_{i_{2m}i_1}].
\end{align}
Moreover, the limiting expected odd moments of $\{\textup{GOE, PTE}\}$ are also equal to 0.

\begin{definition}
For positive integers $n$ and $m$, let $C=c_1c_2\cdots c_{2m}$ be a $(2, m)$-configuration, where $c_i\in\{a_i, b_i\}$ for all $1\leq i\leq 2m$ under the restriction that $(c_{2s-1},c_{2s})\in \{(a_{2s-1}, b_{2s}), (b_{2s-1}, a_{2s})\}$ and $S\subseteq [2n]$ be the set of all the indices of the $a$'s. Let $\pi_S$ be a partition of $S$. Then a \textbf{layer of $[2n]$ with respect to $S$} is a maximal subset $B^{(i)}_S\subseteq [2n]\setminus S$ such that for any $j, k\in B^{(i)}_S$, there doesn't exist $(p, q)\in \pi_S$ such that $j<p<k<q$ or $p<j<q<k$. It's clear from definition that distinct layers must be disjoint. Then we denote the union of all the layers with respect to $S$ by $B_S:=\cup_{i=1}^t B^{(i)}_S$, where $t$ is the total number of layers.
\end{definition}

\begin{lemma}\label{PT-GOE lemma for bs}
Consider a cyclic product in \eqref{expansionGOEPT}. Let $S$ be the set of all the indices of the $a$'s, $\pi_S$ be a matching of the $a$'s and $\pi_{[2n]\setminus S}$ be a matching of the $b$'s. If $\pi_S$ is non-crossing, then the matching $\pi_S\circ\pi_{[2n]\setminus S}$ contribute to \eqref{expansionGOEPT} in the limit if and only if there exists $i$ such that $j, k\in B^{(i)}_S$ for each $(j, k)\in \pi_{[2n]\setminus S}$, i.e., every layer is matched within itself. For the $2m$\textup{\textsuperscript{th}} moment, the number of ways to assign indices for all the $t$ layers is $N^{m+t}+O(N^{m+t-1})$.
\end{lemma}

\begin{proof}
First, observe that each layer $B^{(i)}_{S}$ can be thought of as a cyclic product. For example, consider the following layer $B^{(i)}_S$ consisting of $2\ell$ $b$'s. For clarity, we include some of the $a$'s to highlight how the matching the $a$'s give rise to the cyclic product:
\begin{align}
& a_{i_{j_1-1}i_{j_1}}b_{i_{j_1}i_{j_1+1}}a_{i_{j_1+1}i_{j_1+2}}\cdots a_{i_{j_2-1}i_{j_2}}b_{i_{j_2}i_{j_2+1}}a_{i_{j_2+1}i_{j_2+2}}\cdots \nonumber \\&a_{i_{j_3-1}i_{j_3}}b_{i_{j_3}i_{j_3+1}}a_{i_{j_3+1}i_{j_3+2}}\cdots b_{i_{j_{2\ell}}i_{j_{2\ell}+1}}a_{i_{j_{2\ell}+1}i_{j_{2\ell}+2}}\cdots.
\end{align}

Since the $b$'s form a layer, then for every neighboring two $b$'s, the inner adjacent two $a$'s must be paired together. For example, $a_{i_{j_1+1}i_{j_1+2}}$ and $a_{i_{j_2-1}i_{j_2}}$, which are adjacent $b_{i_{j_1}i_{j_1+1}}$ and $b_{i_{j_2}i_{j_2+1}}$, must be paired together to ensure that all the $b$'s form a layer. Hence, the indices must satisfy the relations $i_{j_1}=i_{j_{2\ell+1}}$, $i_{j_1+1}=i_{j_2}$, $i_{j_2+1}=i_{j_3}$, $\dots$, $i_{j_{2\ell-1}+1}=i_{j_{2\ell}}$, which allows us to think of $B^{(i)}_S$ as $b_{i_1i_2}b_{i_2i_3}\cdots b_{i_{2\ell}i_1}$. Let $\#(B^{(i)}_S)$ be the number of $b$'s in the layer $B^{(i)}$. For each cyclic product, if the matching of all the $b$'s is within each layer, then the number of ways to choose indices for all the $b$'s is $\prod_{i=1}^t (N^{\#(B^{(i)}_{S})/2+1}+O(N^{\#(B^{(i)}_S)/2}))=N^{m+t}+O(N^{m+t-1})$ by \cite{palindromicToeplitz}.

We move on to the case where the matchings of the $b$'s are across different layers. Now, for two arbitrary layers $B^{(i_1)}_S$ and $B^{(i_2)}_S$, suppose that all the $b$'s are paired within these two layers except for at least two $b$'s that are paired across these two layers. Due to the special structure of palindromic Toeplitz, if $b_{i_ji_{j+1}}$ and $b_{i_ki_{k+1}}$ are paired together, then the indices must satisfy the equation $i_{j+1}-i_j+i_{k+1}-i_k=C_j$ for some $C_j\in\{0, \pm (N-1)\}$. Hence, similar to \cite{palindromicToeplitz}, we can think of the matching of all the indices as a system of $M:=(m(B^{(i_1)}_S)+m(B^{(i_2)}_S)/2$ equations, where each index appears exactly twice. After labeling these equations, we pick any equation as $\textup{eq}(M)$, and choose an index that has occured only once. Then, we select the equation in which this index first appeared and label this equation as $\textup{eq}(M-1)$. This index is one of our dependent indices and guarantees consistency choice of indices for the other indices in $\textup{eq}(M-1)$. We can continue this process, and at stage $s$, if at least one index of $\textup{eq}(M-s)$ has occured only once among $\textup{eq}(M-s), \textup{eq}(M-s+1), \dots, \textup{eq}(M)$, then we can choose such an index as one of our dependent indices and continue this process. The only way to terminate this process at stage $s<M-1$ is for all the indices among $\textup{eq}(M-s), \textup{eq}(M-s+1), \dots, \textup{eq}(M)$ to occur twice, which implies that each layer is paired within itself, a contradiction. Hence, if at least two $b$'s are paired across these two layers, then the number of dependent indices is $M/2-1$ and the number of ways to choose indices for all the $b$'s is $N^{M/2+1}$. This is a lower order term compared to the case where each layer is paired within itself, which gives that the number of ways to choose indices for all the $b$'s is $N^{M/2+2}$.

Finally, we consider the case where the matchings of the $a$'s cross each other. If a matching of two $a$'s cross another matching of two $a$'s, then we automatically have three layers $B^{(i_1)}_S, B^{(i_2)}_S$, and $B^{(i_3)}_S$. Due to the mismatch, we can no longer view different layers as independent cyclic products, but all three layers as a single cyclic product. The total number of ways to assign the indices for the three layers is $N^{(m(B^{(i_1)}_S)+m(B^{(i_2)}_S+m(B^{(i_3)}))/2)+1}$, which is a lower order term compared to the case where each of the three layers is matched within itself. Thus, the number of ways to assign indices for all the layers is $O(N^{m+t-1})$, which is again a lower order term. 
\end{proof}

\begin{lemma}\label{PT-GOE lemma for as}
With the same notation as in Lemma \ref{PT-GOE lemma for bs}, regardless of whether $\pi_S$ is non-crossing or not, the number of ways to assign the remaining indices for the $2m$\textup{\textsuperscript{th}} moment is $N^{m+1-t}+O(N^{m-t})$.
\end{lemma}

\begin{proof}
For a fixed $m$, when a cyclic product has only one layer, the only possible configurations for the cyclic product are $abba\cdots abba$ or $baab\cdots baab$; moreover, all the $a$'s must be matched in adjacent pairs. Since there is one free index for each adjacent pair of $a$, then the number of ways to assign the remaining indices is $N^{m}=N^{(m+1)-1}$. This proves the base case.

When a cyclic product $C$ has two layers $B^{(1)}_S$ and $B^{(2)}_S$, suppose that layer $B_S^{(1)}$ is contained in the cyclic product $C_1$ with $2k_1$ total $b$'s and layer $B_S^{(2)}$ is contained in the cyclic product $C_2$ with $2k_2$ total $b$'s, where $C_1\cap C_2=\emptyset$. We can think of $C$ as inserting $C_2$ into $C_1$. From the base case, $C_1$ and $C_2$ are either $abba\cdots abba$ or $baab\cdots baab$. Then, without loss of generalilty, suppose that $C_1$ has the configuration $abba\cdots abba$. If $C_2$ is inserted between two $a$'s in $C_1$, then it must have the configuration $abba\cdots abba$, otherwise $C$ has only one layer instead of two layers. Let $C_2$ be $a_{i'_{1}i'_{2}}b_{i'_{2}i'_{3}}\cdots b_{i'_{4k-1}i'_{4k}}a_{i'_{4k}i'_{1}}$ and surrounded by $a_{i_\ell i_{\ell+1}}$ and $a_{i_{\ell+1}i_{\ell+2}}$ in $C_1$. Since $a_{i'_{1}i'_{2}}$ and $a_{i'_{4k}i'_{1}}$ as well as $a_{i_\ell i_{\ell+1}}$ and $a_{i_{\ell+1}i_{\ell+2}}$ are no longer adjacent, then we lose one additional degrees of freedom and the number of ways to assign the remaining indices is $k_1+k_2-1$. If $C_2$ is inserted between two $b$'s in $C_1$, then it can either be $abba\cdots abba$ or $baab\cdots baab$. Similarly, we can see that the number of ways to assign the remaining indices is $k_1+k_2-1$. Similar constructions follow when we have an arbitrary number of layers in the cyclic product, and whenever we get another layer we lose one extra degree of freedom, giving us $N^{(m+1)-t}+O(N^{m-t})$ ways of assigning the remaining indices for $t$ layers.
\end{proof}

By \eqref{PT-GOE lemma for bs} and \eqref{PT-GOE lemma for as}, for the $2m$\textsuperscript{th} moment, the number of ways to assign all the indices is $N^{2m+1}+O(N^{2m})$ when $\pi_S$ is non-crossing and every layer is matched within itself, and $O(N^{2m})$ otherwise. In other words, in the limit, the only contributions come from non-crossing matchings of the $a$'s and matchings of the $b$'s within the same layer. This leads us to Theorem \ref{sigmarecurrence}, as follows.

\begin{theorem}\label{sigmarecurrence}
The limiting expected $2m$\textup{\textsuperscript{th}} moment $M_{2m}$ of $\{\textup{GOE, PTE}\}$ is given by $\sigma_{m,0}$, where $\sigma_{n,s}$ satisfies the initial conditions $\sigma_{0,s}=(2s-1)!!$ for $s\geq 1$ and $\sigma_{0,0}=1$, and the recurrence relation
\begin{align}
\sigma_{n,s} \ = \ \sum_{k=1}^n\sigma_{k-1,1}\cdot \sigma_{n-k,s}+\sigma_{k-1,0}\cdot \sigma_{n-k,s+1}.
\end{align}
\end{theorem}

\begin{proof}
Let $\sigma_{n,s}$ be the total number of matchings of any configurations of $a$'s and $b$'s that start with $s$ adjacent pairs of $bb$, followed by a valid configuration with a total of $2n$ adjacent pairs of $ab$ or $ba$, where the $a$'s are matched without crossing each other and the $b$'s are matched freely within the same layer. It is clear that the $2m$\textsuperscript{th} moment of $\{\textup{GOE, PTE}\}$ is given by $\sigma_{m, 0}$, and that $\sigma_{0,s}=(2s-1)!!$ for $s\geq 1$, since the number of matchings of $s$ adjacent pairs of $bb$ where the matchings of the $b$'s are free is equal to $(2s-1)!!$. We also set $\sigma_{0,0}=1$ to keep the right count of the total number of matchings.

Now, we move on to prove the recurrence relation for $\sigma_{n, s}$. Suppose that the $a$ in the first adjacent pair of $ab$ or $ba$ is paired with another $a$ in the $t$\textsuperscript{th} adjacent pair of $ab$ or $ba$. Note that $t$ must be even, since otherwise there is only an odd number of $a$'s and $b$'s within the layer created by the matching of the two $a$'s, in which case the contribution is a lower order term by Lemma \ref{PT-GOE lemma for bs} and \ref{PT-GOE lemma for as}. Suppose that $t=2k$ for some $1\leq k\leq n$. If the first adjacent pair of $ab$ or $ba$ is an $ab$, then the $2k$\textsuperscript{th} adjacent pair of $ab$ or $ba$ must be a $ba$, and vice versa. In both cases, the matching of the two $a$'s split the cyclic product into two smaller cyclic product, as illustrated in the following example.

\begin{example}
If $(n, s)=(2, 2)$, then an example of a cyclic product with $k=2$ is 
\begin{align}
b_{i_1i_2}b_{i_2i_3}b_{i_3i_4}b_{i_4i_5}b_{i_5i_6}a_{i_6i_7}b_{i_7i_8}a_{i_8i_9}a_{i_9i_{10}}b_{i_{10}i_{11}}a_{i_{11}i_{12}}b_{i_{12}i_{1}}.
\end{align}
The matching of the $a$'s partitions the cyclic product into two smaller cyclic products $b_{i_7i_8}a_{i_8i_9}a_{i_9i_{10}}b_{i_{10}i_{11}}$ (inner cyclic product) and $b_{i_{12}i_1}b_{i_1i_2}b_{i_2i_3}b_{i_3i_4}b_{i_4i_5}b_{i_5i_6}$ (outer cyclic product), where a term from either smaller cyclic product is paired with another term in the same smaller cyclic product.
\end{example} 

If the first adjacent pair of $ab$ or $ba$ is a $ba$, then the matching of the $a$'s partitions the cyclic product into two smaller cyclic products $C'_{1}$ (inner cyclic product)  and $C'_{2}$ (outer cyclic product), where $C'_{1}$ doesn't start with any adjacent pairs of $bb$ and has a valid configuration with a total of $2(k-1)$ adjacent pairs of $ab$ or $ba$, and $C'_2$ starts with $s+1$ adjacent pairs of $bb$ and has a valid configuration with a total of $2(n-k)$ adjacent pairs of $ab$ or $ba$. Then the total number of matchings is $\sigma_{k-1,0}$ for $C'_{1}$ and $\sigma_{n-k,s+1}$ for $C'_{2}$. Similarly, if the first adjacent pair of $ab$ or $ba$ is a $ab$, then the total number of matchings is $\sigma_{k-1, 1}$ for $C'_{1}$ and $\sigma_{n-k,s}$ for $C'_{2}$.

Summing over all possible $k$'s, we have 
\begin{align}
\sigma_{n,s} \ = \ \sum_{k=1}^n\sigma_{k-1,1}\cdot \sigma_{n-k,s}+\sigma_{k-1,0}\cdot \sigma_{n-k,s+1},
\end{align}
as desired.
\end{proof}

\begin{cor}\label{horriblePDE}
The generating function $F(z,w):=\sum_{n,s\geq0}\sigma_{n,s}z^n w^s/s!$ satisfies the partial differential equation
\begin{align}
F(z,w) \ = \ (1-2w)^{-1/2}+z\left(\frac{\partial F(z,0)}{\partial w}F(z,w)+F(z,0)\frac{\partial F(z,w)}{\partial w}\right).
\end{align}
\end{cor}

\begin{proof}
    This follows from the recurrence relation given in Theorem \ref{sigmarecurrence}. Let $F(z,w):=\sum_{n,s\geq0}\sigma_{n,s}z^n w^s/s!$. Then
    \begin{align}
        \sum_{\substack{n\geq 1\\ s\geq0}}\sigma_{n,s}z^n\frac{w^s}{s!} \ = \ \sum_{\substack{n\geq 1\\ s\geq0}}\left(\sum_{k=1}^n\sigma_{k-1,1} \sigma_{n-k,s}\right)z^n\frac{w^s}{s!}+\sum_{\substack{n\geq 1\\ s\geq0}}\left(\sum_{k=1}^n \sigma_{k-1,0} \sigma_{n-k,s+1}\right)z^n\frac{w^s}{s!}.
    \end{align}
    We have
    \begin{align}
        \sum_{\substack{n\geq 1\\ s\geq0}}\sigma_{n,s}z^n\frac{w^s}{s!} \ = \ F(z,w)-\sum_{s\geq0}\sigma_{0,s}\frac{w^s}{s!} \ = \ F(z,w)-(1-2w)^{-1/2}
    \end{align}
    by noting the identity $\sum_{s\geq0}\sigma_{0,s}w^s/s!=(1-2w)^{-1/2}$, which follows from the generalized binomial theorem.
    
    The convolution of $\sigma_{k-1,1}$ and $\sigma_{k,s}$ corresponds to the product of $z\frac{\partial F(z,0)}{\partial w}$ and $F(z,w)$ respectively. Further, the convolution of $\sigma_{k-1,0}$ and $\sigma_{k,s+1}$ corresponds to the product of $zF(z,0)$ and $\frac{\partial F(z,w)}{\partial w}$ respectively. Thus,
    \begin{align}
        F(z,w)-(1-2w)^{-1/2} \ = \ z\frac{\partial F(z,0)}{\partial w}F(z,w)+zF(z,0)\frac{\partial F(z,w)}{\partial w},
    \end{align}
    giving the equation as desired.
\end{proof}

From Corollary \ref{horriblePDE}, we ask the following open question:

\begin{que}
    Is there a closed-form generating function solution of the form $G(z,w)=\sum_{n,s\geq0}\sigma_{n,s}z^n \frac{w^s}{s!}$ to the partial differential equation
    \begin{align}
        G(z,w) \ = \ (1-2w)^{-1/2}+z\left(\frac{\partial G(z,0)}{\partial w}G(z,w)+\frac{\partial G(z,w)}{\partial w}G(z,0)\right)
    \end{align}
    with initial conditions $\sigma_{0,0}=1$ and $\sigma_{0,s}=(2s-1)!!$ for $s\geq1$? An answer in the affirmative would give an explicit formula for $M_{2m}$.
\end{que}

Due to the complexity of the recurrence in Theorem \ref{sigmarecurrence} and partial differential equation in Proposition \ref{horriblePDE}, it is difficult to obtain an explicit formula for $\sigma_{n,0}$ and thus $M_{2m}$. Nevertheless, we provide some crude upper and lower bounds on $M_{2m}$ in the following corollary.
\begin{cor} 
We have the bounds
\begin{align}
\sum_{i=0}^m\binom{m}{i}(2i-1)!! \ \leq \ 
M_{2 m} \ \leq \ 4^m((2m-1)!!) C_{m},
\end{align}
where $C_m$ is the $m$\textup{\textsuperscript{th}} Catalan number. 
\end{cor}

\begin{proof}
We know that the limiting expected $2m$\textsuperscript{th} moment of $\{\textup{GOE, PTE}\}$ is the total number of pairings of all cyclic products of length $4m$, where the matching rules are specified by Lemma \ref{PT-GOE lemma for bs} and \ref{PT-GOE lemma for as}. The total number of configurations is $4^m$, the total number of ways to pair the $a$'s is $C_m$, and the total number of ways to pair b's is $(2m - 1)!!$. Ignoring the potential crossings between the pairings, we can overcount by multiplying the three qualities, which establishes the upper bound
\begin{align}
M_{2 m} \ \leq \ 4^m((2m-1)!!) C_{m}. 
\end{align}
For the lower bound, we define a sequence $\nu_{4m, 2k}$ that provides a lower bound for the total number pairings of all cyclic products of length $4m+2k$ where the first 
$4m$ terms is a valid configuration and the last $2k$ terms are $k$ blocks of $bb$. The base case is given by $\nu_{0, 2k} \ = \ (2k - 1)!!$ due to the free matching property of the $b$'s. Then, we assume the last two blocks are either $abba$ or $baab$ and the $a$'s are matched together. This gives us the recurrence relation
\begin{align}
\nu_{4m, 2k} \ = \ \nu_{4(m-1), 2k} + \nu_{4(m - 1), 2k + 2}.
\end{align}
By induction, we find that 
\begin{align}
\nu_{4m, 0} \ = \ \sum_{i = 0}^m \binom{m}{i} (2i - 1)!!.
\end{align}
\end{proof}

\subsection{Limiting Expected Moments of the Anticommutator of GOE and BCE and of the Anticommutator of BCE and BCE} The real symmetric $k$-block circulant ensemble is introduced by Kolo$\breve{{\rm g}}$lu-Kopp-Miller in \cite{Block Circulant} and combines the structure of palindromic Toeplitz matrices and block circulant matrices: not only do entries on different diagonals satisfy relations analogous to the palindromic Toeplitz ensemble, but the entries on the same diagonal also appear periodically due to the $k$-block structure. Because of its complicated structure, the $2m\textsuperscript{th}$ moments of the spectral distribution are not given explicitly, but in terms of the number of pairings of the edges of a $2m$-gon which give rise to a genus $g$ surface. Using generating function techniques and contour integral, a closed-form expression for the limiting spectral density is also obtained as the product of a Gaussian and some polynomial of degree $2k-2$. Similar to the palindromic Toeplitz case, we can extend the moment calculation of $k$-block circulant ensemble to that of $\{k\textup{-BCE}, k\textup{-BCE}\}$ and $\{\textup{GOE}, k\textup{-BCE}\}$.

Suppose that $b_{i_si_{s+1}}$ and $b_{i_ti_{t+1}}$ are entries from an $N\times N$ real symmetric $k$-block circulant matrix, then $b_{i_si_{s+1}}$ and $b_{i_ti_{t+1}}$ are matched iff either of the following relations hold:
\begin{enumerate}
\item $i_{s+1}-i_s=i_{t+1}-i_t+C_s$ and $i_s\equiv i_{t}\textup{ (mod $k$)}$, or
\item $i_{s+1}-i_s=-(i_{t+1}-i_t)+C_s$ and $i_s\equiv i_{t+1}\textup{ (mod $k$)}$,
\end{enumerate}
where $C_s\in \{0, \pm N\}$. The difference in sign in the two relations above allows us to think of the matching of $(s,s+1)$ and $(t,t+1)$ as having the same or different orientations. For both $\{k\textup{-BCE}, k\textup{-BCE}\}$ and $\{\textup{GOE, $k$\textup{-BCE}}\}$, we can apply the same argument from \cite{Toeplitz}, \cite{palindromicToeplitz}, and \cite{Block Circulant} to show that the total contribution of all the pairings with at least one matching of the same orientation is $O(1/N)$. Hence, it suffices to consider those pairings with matchings of the same orientation, i.e., $i_{s+1}-i_s=-(i_{t+1}-i_t)+C_s$ and $i_s\equiv i_t\textup{ (mod $k$)}$. By assumption $k=o(N)$, then the modular restrictions do not reduce the total degrees of freedom. Hence, analogous to the palindromic Toeplitz case, we can think of the pairing of terms of in the $2m\textsuperscript{th}$ moment of an $N\times N$ real symmetric $k$-block circulant matrix as a system of $m$ linear equations each of the form $i_{s+1}-i_s=-(i_{t+1}-i_t)+C_s$. This gives us $m+1$ free indices with $m-1$ dependent indices and constants $C_{s}\in \{0,\pm N\}$ uniquely determined, except for a lower order term of choices of free indices.

Using the idea of layers developed in subsection \ref{subsectionGOEpalindromicToeplitz}, we see that if $\pi$ is a pairing of $\{\textup{GOE}, k\textup{-BCE}\}$, then $\pi$ contributes to the moment of $\{\textup{GOE}, k\textup{-BCE}\}$ iff the GOE terms are matched non-crossing and the real symmetric $k$-block circulant terms are matched within each layer. Now, consider a pairing $\pi$ in the $2m\textsuperscript{th}$ moment of $\{\textup{GOE}, k\textup{-BCE}\}$. We identify the matched indices in the same congruence class modulo $k$ by the equivalence relation $\sim$. For example, if $a_{i_si_{s+1}}$ and $a_{i_ti_{t+1}}$ are matched, i.e., $i_s=i_{t+1}$ and $i_{s+1}=i_t$, then $i_s\sim i_{t+1}$ and $i_{s+1}\sim i_t$. If $b_{i_si_{s+1}}$ and $b_{i_ti_{t+1}}$ are matched, i.e., $i_{s+1}-i_s=-(i_{t+1}-i_t)+C_s$ and $i_s\equiv i_{t+1}\textup{ (mod $k$)}$, then we also have $i_{s+1}\equiv i_t\textup{ (mod $k$)}$. Hence, we still have $i_s\sim i_{t+1}$ and $i_{s+1}\sim i_t$. We see that the number of equivalence classes of indices of the pairing $\pi$ is $\#(\gamma_{4m}\pi)$. For each equivalence class, there are $k$ ways to choose congruence classes. So the number of ways to choose congruence class for all the indices is $k^{\#(\gamma_{4m}\pi)}$. Since there are $N/k$ choices for indices for each congruence class, then $2m\textsuperscript{th}$ moment of $\{\textup{GOE}, k\textup{-BCE}\}$ is given by
\begin{align}
M_{2m} \ = \ \sum_{C\in\mathcal{C}_{2,4m}}\sum_{\pi_C\in NCF_{2,C}(4m)} k^{\#(\gamma_{4m}\pi)-(2m+1)},
\end{align}
where $NCF_{2,C}(4m)$ denotes the set of all the pairings with respect to $C$ of $[4m]$ where the GOE terms are matched non-crossingly and the block circulant terms are matched freely without crossing the matchings of the GOE terms. Similarly, the $2m\textsuperscript{th}$ moment of $\{k\textup{-BCE}, k\textup{-BCE}\}$ is given by
\begin{align}
M'_{2m} \ = \ \sum_{C\in \mathcal{C}_{2,4m}}\sum_{\pi_C\in \mathcal{P}_{2,C}(4m)}k^{\#(\gamma_{4m}\pi)-(2m+1)}.
\end{align}
Summing up, we have the following theorem.
\begin{theorem}
The limiting expected $2m$\textup{\textsuperscript{th}} moment $M_{2m}$ of $\{\textup{GOE, $k$-BCE}\}$ is given by the genus expansion formula
\begin{align}
M_{2m} \ = \ \sum_{C\in\mathcal{C}_{2,4m}}\sum_{\pi_C\in NCF_{2,C}(4m)}k^{\#(\gamma_{4m}\pi)-(2m+1)},
\end{align}
and the limiting expected $2m$\textup{\textsuperscript{th}} moment $M'_{2m}$ of $\{\textup{$k$-BCE, $k$-BCE}\}$ is given by the genus expansion formula
\begin{align}
M'_{2m} \ = \ \sum_{C\in\mathcal{C}_{2,4m}}\sum_{\pi_C\in \mathcal{P}_{2,C}(4m)}k^{\#(\gamma_{4m}\pi)-(2m+1)}.
\end{align}
\end{theorem}

\section{Limiting Expected Moments of Blips of Checkerboard Ensemble}
In Section \ref{sec: anticommutator Combinatorics}, we considered random matrix ensembles whose eigenvalues are $\Theta(N)$. In this section, we introduce new random matrix ensembles, $\{\textup{GOE, $k$-checkerboard}\}$ and $\{\textup{$k$-checkerboard, $j$-checkerboard}\}$, where this condition no longer holds. As we shall see, this new ensemble exhibits entirely different behaviors from these previously considered ensembles.  Throughout our discussion of $\{\textup{GOE, $k$-checkerboard}\}$, when the dimension of the ensemble is $N$, we always assume that $k>1$ and $k\mid N$; similarly, for $\{\textup{$k$-checkerboard, $j$-checkerboard}\}$, when the dimension of the ensemble is $N$, we assume that $k,j>1$, $\textup{gcd}(k,j)=1$, and $k,j\mid N$. These are simplications crucial to our calculation later on. According to numerical simulation, for large $N$ the spectrum of $\{\textup{GOE, $k$-checkerboard}\}$ consists of 2 regimes in total: $1$ bulk regime of size $\Theta(N)$ and 1 blip regime that contains $2k$ eigenvalues and are positioned symmetrically around 0 at $\pm N^{3/2}/k$. On the other hand, for large $N$ the spectrum of $\{\textup{$k$-checkerboard, $j$-checkerboard}\}$ consists of 4 regimes in total: $1$ bulk regime of size $\Theta(N)$, one intermediary blip regime that contains $2k-2$ eigenvalues and is positioned symmetrically  around 0 at $\pm \frac{N^{3/2}}{k}\sqrt{1-\frac{1}{j}}+\Theta(N)$, one intermediary blip regime that contains $2j-2$ eigenvalues and is positioned symmetrically around 0 at $\pm \frac{N^{3/2}}{j}\sqrt{1-\frac{1}{k}}+\Theta(N)$, and one largest blip regime at $2N^2/kj+\Theta(N)$ containing of 1 eigenvalue. This is different from \cite{split} and \cite{CLMY}, where only blip regimes of the same order are observed.

In what follows, we first define the empirical blip spectral measure for each blip regime in terms of the appropriate weight function. Then, we adopt the language developed in \cite{split} to identify the types of cyclic products that contribute to the limiting expected $m$\textsuperscript{th} moments of the empirical blip spectral measure. Finally, we use combinatorics and cancellation techniques to obtain the limiting expected $m$\textsuperscript{th} moments of the empirical blip spectral measure for $\{\textup{GOE, $k$-checkerboard}\}$ and for the largest blip regime of $\{\textup{$k$-checkerboard, $j$-checkerboard}\}$. We end this section by highlighting the challenges with finding the moments of the intermediary blip regimes of $\{\textup{$k$-checkerboard, $j$-checkerboard}\}$.

\subsection{Preliminaries}
In the following definitions of the empirical spectral measure for each blip regime, we will use a weight function that puts the weight 1 at the location of the blip regime, decays very rapidly to 0 in the neighborhood of that regime, and puts the weight 0 at all other regimes (if any). As we are only focusing on the largest blip regime in this section, after some normalization, we can use $f^{(2n)}=x^{2n}(2-x)^{2n}$ as our weight function, where $n=\log\log(N)$. Note that we choose $n=\log\log(N)$ is key to ensuring the decay condition and assigning the right amount of weight to each blip regime.

\begin{definition}\label{Blip Spectral Measures}
Let $n=\log\log(N)$ and choose the weight function to be $f^{(2n)}=x^{2n}(2-x)^{2n}$. Suppose that $A_N$ is an $N\times N$ matrix sampled from GOE and $B_N$ is an $N\times N$ matrix sampled from the $k$-checkerboard ensemble, where the samplings occur independently. Then the $f^{(2n)}$-\textbf{weighted empirical blip spectral measure} associated to $\{A_N,B_N\}$ around $\pm N^{3/2}/k$ is
\begin{align}\label{blip regime GOE Checkerboard}
\mu_{\{A_N,B_N\},f^{(2n)}}(x)dx \ = \ \frac{1}{2k}\sum_{\lambda\textup{ eigenvalues}}f^{(2n)}\left(\frac{k^2\lambda^2}{N^{3}}\right)\delta\left(x-\frac{\lambda^2-N^{3}/k^2}{N^{5/2}}\right)dx.
\end{align}
\end{definition}
\begin{definition}
Let $n=\log\log(N)$ and choose the weight function to be $f^{(2n)}=x^{2n}(2-x)^{2n}$. Suppose that $A_N$ is an $N\times N$ matrix sampled from the $k$-checkerboard ensemble and $B_N$ is an $N\times N$ matrix sampled from the $j$-checkerboard ensemble, where the samplings occur independently. Then the $f^{(2n)}$-\textbf{weighted empirical blip spectral measure} associated to $\{A_N, B_N\}$ around $2N^2/kj$ is
\begin{align}
\mu_{\{A_N,B_N\},f^{(2n)}}(x)dx=\sum_{\lambda\textup{ eigenvalues}}f^{(2n)}\left(\frac{jk\lambda}{2N^2}\right)\delta\left(x-\left(\frac{\lambda-\frac{2N^2}{jk}}{N}\right)\right)dx.
\end{align}
\end{definition}

Recall that the blip regime of $\{\textup{GOE, $k$-checkerboard}\}$ contains both positive and negative eigenvalues, and numerical simulation suggests that they are positioned almost symmetrically around 0, so we need to account for this in our definition. This is why we have the $k^2\lambda^2/N^3$ and $(\lambda^2-N^3/k^2)/N^{5/2}$ factors in \eqref{blip regime GOE Checkerboard}-- whenever an eigenvalue is around $\pm N^{3/2}/k+\Theta(N)$, it will be accounted for in the expected moments of this intermediary blip regime. Note that here the normalization inside the delta functional is $N^{5/2}$ instead of $N$ because the error term of $\lambda^2$ is $\Theta(N^{5/2})$ for $\lambda=\pm N^{3/2}/k+\Theta(N)$.

Since the weight function $f^{(2n)}(x)=x^{2n}(2-x)^{2n}$ is a polynomial, then we can instead write $f^{(2n)}(x)=\sum_{\alpha=2n}^{4n}c_\alpha x^{\alpha}$, where $c_\alpha$ is the coefficient of the $x^\alpha$ term of $f^{(2n)}(x)$. We choose $f^{(2n)}:=x^{2n}(2-x)^{2n}$ to be the weight function for the empirical blip spectral measure of the anticommutator of an $N\times N$ GOE and an $N\times N$ $k$-checkerboard because $f^{(2n)}\left(\frac{k^2\lambda^2}{N^3}\right)\approx 0$ for any bulk eigenvalue $\lambda$, and $f^{(2n)}\left(\frac{k^2\lambda^2}{N^3}\right)\approx 1$ for any blip eigenvalue $\lambda$. Similar reasoning applies to choosing $f^{(2n)}$ to be the weight function for the empirical largest blip spectral measure of the anticommutator of an $N\times N$ $k$-checkerboard and an $N\times N$ $j$-checkerboard. By eigenvalue trace lemma, the expected $m$\textsuperscript{th} moment of the $f^{(2n)}$-weighted empirical blip spectral measure of the anticommutator of an $N\times N$ GOE and an $N\times N$ $k$-checkerboard $\{A_N, B_N\}$ around $\pm N^{3/2}/k$ is

\begin{align}
\mathbb{E}\left[\mu_{{\{A_N,B_N\},f^{(2n)}}}^{(m)}\right] &\ = \ \mathbb{E}\left[\frac{1}{2k}\sum_{\lambda\textup{ eigenvalues}}\sum_{\alpha=2n}^{4n}c_\alpha\left(\frac{k^2\lambda^2}{N^{3}}\right)^\alpha\left(\frac{\lambda^2-N^{3}/k^2}{N^{5/2}}\right)^{m}\right] \nonumber \\
&\ = \ \mathbb{E}\left[\frac{1}{2k}\sum_{\alpha=2n}^{4n}c_\alpha\left(\frac{k^2}{N^{3}}\right)^\alpha\frac{1}{N^{5m/2}}\sum_{i=0}^m \binom{m}{i}\left(-\frac{N^{3}}{k^2}\right)^{m-i}\textup{Tr}\left(\{A_N,B_N\}^{2(\alpha+i)}\right)\right] \nonumber \\
&\ = \ \frac{1}{2k}\sum_{\alpha=2n}^{4n}c_\alpha\left(\frac{k^2}{N^{3}}\right)^\alpha\frac{1}{N^{5m/2}}\sum_{i=0}^m \binom{m}{i}\left(-\frac{N^{3}}{k^2}\right)^{m-i}\mathbb{E}\left[\textup{Tr}\left(\{A_N,B_N\}^{2(\alpha+i)}\right)\right]\label{MomentCalc}.
\end{align}

Then by eigenvalue trace lemma, the expected $m$\textsuperscript{th} moment of the $f^{(2n)}$-weighted empirical largest blip spectral measure of the anticommutator of an $N\times N$ $k$-checkerboard and an $N\times N$ $j$-checkerboard $\{A_N, B_N\}$ around $2N^2/kj$ is
\begin{align}\label{momentOfCNDN}
\mathbb{E}\left[\mu^{(m)}_{\{A_N, B_N\},f^{(2n)}}\right] &\ = \ \mathbb{E}\left[\sum_{\lambda\textup{ eigenvalues}} \sum_{\alpha=2n}^{4n}c_\alpha\left(\frac{jk\lambda}{2N^2}\right)^\alpha \left(\frac{\lambda-\frac{2N^2}{jk}}{N}\right)^m\right] \nonumber \\
&\ = \ \mathbb{E}\left[\sum_{\alpha=2n}^{4n}c_\alpha\left(\frac{jk}{2N^2}\right)^\alpha\frac{1}{N^m} \sum_{i=0}^m\binom{m}{i}\left(-\frac{2N^2}{jk}\right)^{m-i}\textup{Tr}\left(\{A_N, B_N\}^{\alpha+i}\right)\right] \nonumber \\
&\ = \ \sum_{\alpha=2n}^{4n}c_\alpha \left(\frac{jk}{2N^2}\right)^\alpha \frac{1}{N^m}\sum_{i=0}^m\binom{m}{i}\left(-\frac{2N^2}{jk}\right)^{m-i}\mathbb{E}\left[\text{Tr}\left(\{A_N,B_N\}^{\alpha+i}\right)\right].
\end{align}

We know that for an $N\times N$ anticommutator ensemble $\{X_N, Y_N\}$, the $(\alpha+i)$\textsuperscript{th} expected moment is
\begin{align}
\mathbb{E}[\textup{Tr}(\{X_N, Y_N\}^{\alpha+i})] \ = \ \sum_{1\leq i_1, \cdots, i_{2(\alpha+2i)}\leq N}\mathbb{E}[c_{i_1i_2}\cdots c_{i_{2(\alpha+2i)}i_1}],
\end{align}
where each $c_{i_t}$ is a matrix entry from either $X_N$ or $Y_N$. So now, the blip moment calculation has become a combinatorial problem of counting and leveraging the expected value of different configurations of cyclic products. To facilitate our discussion, we adopt the language introduced in \cite{split} to describe this combinatorial problem. From now on in this section, we refer to each matrix entry as a \textbf{term}. When $A_N$ is sampled from GOE, there is no weight term, and we use $a$ to denote a term of $A_N$; when $A_N$ is sampled from the $k$-checkerboard ensemble, we use $a$ to denote a non-weight term of $A_N$ and $w$ to denote a weight term of $A_N$. Regardless of whether $B_N$ is sampled from the $k$-checkerboard or the $j$-checkerboard ensemble, we use $b$ to denote a non-weight term of $B_N$ and $v$ to denote a weight term of $B_N$. Finally, we use $c$ to denote any entry term of $A_N$ or $B_N$. 

\begin{defi}
A \textbf{block} is a set of adjacent $a$'s and $b$'s surrounded by $w$'s (and $v$'s, for the anticommutator of $k$-checkerboard and $j$-checkerboard) in a cyclic product, where the last term of a cyclic product is considered to be adjacent to the first. We refer to a block of length $\ell$ as an $\ell$ block or sometimes a block of size $\ell$.
\end{defi}

\begin{defi}
A \textbf{weight block} is a set of adjacent $w$'s (and $v$'s for the anticommutator of $k$-checkerboard and $j$-checkerboard) surrounded by $a$'s and $b$'s in a cyclic product. We similarly refer to a weight block of length $\ell$ as an $\ell$ weight block or sometimes a weight block of size $\ell$.
\end{defi}

\begin{defi}
An \textbf{adjacent pair} is a pair of adjacent entries of the form $c_{i_{2\ell-1}i_{2\ell}}$, where the first term starts with an odd index.
\end{defi}

\begin{defi}
A \textbf{weight pair} is a pair of adjacent weight terms $c_{i_{2\ell -1}i_{2\ell}}c_{i_{2\ell}i_{2\ell+1}}$. The anticommutator of GOE and the $k$-checkerboard ensemble doesn't have a weight pair in its expected moment calculation. For the anticommutator of the $k$-checkerboard ensemble and the $j$-checkerboard ensemble, a weight pair can be 
\begin{align}
\{c_{i_{2\ell -1}i_{2\ell}}, c_{i_{2\ell}i_{2\ell+1}}\}\in \{\{w_{i_{2\ell -1}i_{2\ell}}, v_{i_{2\ell}i_{2\ell+1}}\}, \{v_{i_{2\ell -1}i_{2\ell}},w_{i_{2\ell}i_{2\ell+1}}\}\} 
\end{align}
\end{defi}

\begin{defi}
An \textbf{mixed pair} is a pair of adjacent weight and non-weight terms $c_{i_{2\ell-1}i_{2\ell}}c_{i_{2\ell}i_{2\ell+1}}$. Due to the structure of anticommutator, a mixed pair for the anticommutator of GOE and the $k$-checkerboard ensemble can be
\begin{align}
\{c_{i_{2\ell-1}i_{2\ell}}, c_{i_{2\ell}i_{2\ell+1}}\}\in \{\{a_{i_{2\ell-1}i_{2\ell}}, v_{i_{2\ell}i_{2\ell+1}}\}, \{v_{i_{2\ell-1}i_{2\ell}}, a_{i_{2\ell}i_{2\ell+1}}\}\}.
\end{align}
A mixed pair for the anticommutator of the $k$-checkerboard ensemble and the $j$-checkerboard ensemble can be
\begin{align}
\{c_{i_{2\ell-1}i_{2\ell}}, c_{i_{2\ell}i_{2\ell+1}}\}\in \{\{a_{i_{2\ell-1}i_{2\ell}}, v_{i_{2\ell}i_{2\ell+1}}\}, \{v_{i_{2\ell-1}i_{2\ell}}, a_{i_{2\ell}i_{2\ell+1}}\}, \{b_{i_{2\ell-1}i_{2\ell}}, w_{i_{2\ell}i_{2\ell+1}}\}, \{w_{i_{2\ell-1}i_{2\ell}}, b_{i_{2\ell}i_{2\ell+1}}\}\}.
\end{align}
\end{defi}

\begin{defi}
A \textbf{configuration} is a set of all cyclic products for which it is specified (a) how many blocks there are and what each of them compose of (e.g., a block of $abba$); and (b) in what order these blocks appear (up to cyclic permutation); However, it is not specified how many $w$'s (and $v$'s for the anticommutator of the $k$-checkerboard ensemble and the $j$-checkerboard ensemble) there are between each block.
\end{defi}

\begin{defi}
A \textbf{congruence configuration} is a configuration together with a choice of the congruence class modulo $k$ every index.
\end{defi}

\begin{defi}
Given a configuration, a \textbf{matching} is an equivalence relation $\sim$ on the $a$'s and $b$'s in the cyclic product which constrains the way of indexing: for any $c_{i_\ell i_{\ell+1}}$ and $c_{i_ti_{t+1}}$, if $\{c_{i_\ell i_{\ell+1}}, c_{i_ti_{t+1}}\}\in \{\{a_{i_\ell i_{\ell+1}},a_{i_ti_{t+1}}\}, \{b_{i_\ell i_{\ell+1}}, b_{i_t i_{t+1}}\}\}$, then $\{i_\ell, i_{\ell+1}\}=\{i_t, i_{t+1}\}$ if and only if $c_{i_\ell i_{\ell+1}}\sim c_{i_ti_{t+1}}$.
\end{defi}

\begin{defi}
Given a configuration, matching, and length of the cyclic product, then an \textbf{indexing} is a choice of
\begin{enumerate}
\item the (positive) number of $w$'s and $v$'s between each pair of adjacent blocks (in the cyclic sense), and
\item the integer indices of each $a$, $b$, $w$, $v$ in the cyclic product.
\end{enumerate}
\end{defi}

\begin{defi}
A \textbf{configuration equivalence} $\sim_C$ is an equivalence relation on the set of all configurations such that for any configurations $C_1, C_2$, $C_1\sim_C C_2$ if and only if they have the same blocks (but they may have different number and arrangement of $w$'s and $v$'s between their blocks). Every equivalence class under $\sim_C$ is called an \textbf{$S$-class}, specified by the blocks in all the configurations in the equivalence class. 
\end{defi}

Lemma \ref{blockslemma} below characterizes the $S$-class with the highest degree of freedom and boils down the blip moment calculation to consideration of some nice configurations. Its assertions hold true for both $\{\textup{GOE, $k$-checkerboard}\}$ and $\{\textup{$k$-checkerboard, $j$-checkerboard}\}$.

\begin{lemma}\label{blockslemma}
Fix $m\geq 1$, consider all the $S$-classes with $|S|=m$. Then a $S$-class with a matching $\sim$ yields the highest degrees of freedom iff it satisfies the following conditions.
\begin{enumerate}
\item It consists only of the following blocks: (i) 1-block of $a$ (i.e., $av$ or $va$); (ii) 1-block of $b$ (i.e., $bw$ or $wb$, but only for \{\textup{$k$-checkerboard and $j$-checkerboard}\}); (iii) 2-block of $aa$ (i.e., $vaav$), (iv) 2-block of $bb$ (i.e., $wbbw$, but only for \{\textup{GOE and $k$-checkerboard}\}).
\item Each 1-block is paired up to another 1-block and the two terms in each 2-block are paired up with each other.
\end{enumerate}
\end{lemma}

\begin{proof}
Similar to Lemma 3.14 of \cite{split}, we see that when a 1-block of $a$ (and $b$ for the anticommutator of $k$-checkerboard and $j$-checkerboard) is paired up with another 1-block of $a$ (and $b$ for the anticommutator of $k$-checkerboard and $j$-checkerboard) or when the terms in a 2-block of $a$'s (and $b$'s for the anticommutator of $k$-checkerboard and $j$-checkerboard) are paired up with each other, there is one degree of freedom lost per block. Now, fix a configuration $\mathcal{C}$ with $\alpha$ from the $a$'s and $\beta$ from the $b$'s and a matching $\sim$. Suppose that $\sim$ partitions all the $a$'s into equivalence classes $\mathcal{E}_1, \cdots, \mathcal{E}_{s_a}$ and the $b$'s into $\mathcal{E}'_1, \cdots, \mathcal{E}'_{s_b}$. Then, without any matching restrictions, the degrees of freedom of $\mathcal{C}$ is
\begin{align}
\mathcal{\Tilde{F}}_{\mathcal{C}} \ = \ \sum_{\textup{blocks } \mathcal{B}}(\textup{len}(\mathcal{B})+1) \ = \ \alpha+\beta+m.
\end{align}
To find the actual degree of freedom $\mathcal{F}_\mathcal{C}$ of $\mathcal{C}$, we can choose two indices from each equivalence class. However, adjacent $a$'s and $b$'s from different equivalence classes (which we call cross-overs) place restrictions on the indices and cause additional loss of degrees of freedom. Let the loss of degrees of freedom due to cross-overs be $\gamma$, then $\mathcal{F}_\mathcal{C}=2s_a+2s_b-\gamma$. Thus, the degree of freedom lost per block is
\begin{equation}\label{degreeoffreedom}
\mathcal{\overline{L}}_\mathcal{C} \ = \ \frac{\mathcal{\Tilde{F}_C}-\mathcal{F}_C}{m} \ = \ 1+\frac{\alpha+\beta+\gamma-2s_a-2s_b}{m}.
\end{equation}
Since $|\mathcal{E}_{i}|,|\mathcal{E}_{j}'|\geq 2$ for $1\leq i\leq s_a$ and $1\leq j\leq s_b$, then $s_a\leq \frac{\alpha}{2}$ and $s_b\leq \frac{\beta}{2}$, and so $\mathcal{\overline{L}}_\mathcal{C}\geq 1$. We've shown that if $\mathcal{C}$ satisfies the conditions (1) and (2), then $\mathcal{\overline{L}}_\mathcal{C}=1$. Hence, it suffices to show that if $\mathcal{C}$ with a matching $\sim$ loses one degree of freedom per block (or equivalently, satisfies $\frac{\alpha+\beta+\gamma}{s_a+s_b}=2$), then it must satisfy the conditions (1) and (2). Since $|\mathcal{E}_{i}|,|\mathcal{E}'_{j}|\geq 2$, we get $\alpha\geq 2s_a$ and $\beta\geq 2a_b$. Hence, if some $|\mathcal{E}_{i}|>2$ or $|\mathcal{E}'_{j}|> 2$, then $\frac{\alpha+\beta+\gamma}{s_a+s_b}>2$. Moreover, if $\gamma>0$, then $\frac{\alpha+\beta+\gamma}{s_a+s_b}>2$. Therefore, if $\mathcal{C}$ with a matching $\sim$ loses one degree of freedom per block, then all the blocks are paired up and there can be no cross-overs from different equivalence classes. Thus, the only possible $S$-classes and matchings are those satisfying conditions (1) and (2).
\end{proof}

\begin{lemma}\label{Sclasscontribution} Suppose that $A_N$ is an $N\times N$ matrix sampled from GOE and $B_N$ is an $N\times N$ matrix sampled from the $k$-checkerboard ensemble. For $m_1, m_2\in \Z_{\geq 0}$, the total contribution to $\mathbb{E}[\textup{Tr}\{A_N, B_N\}^\eta]$ of an $S$-class with $m_1$ 1-blocks of $a$ and $m_2$ 2-blocks of $a$ is
\begin{align}
\frac{p(\eta)N^{\frac{3\eta}{2} -\frac{m_1}{2}}}{k^{\eta}}\mathbb{E}_k[\textup{Tr }C^{m_1}]+O\left(N^{\frac{3\eta}{2} -\frac{m_1}{2}-1}\right),
\end{align}
where $p(\eta)=\frac{2\eta^{m_1}}{m_1!}+O(\eta^{m_1-1})$ and $C$ is a $k\times k$ hollow GOE matrix.
\end{lemma}

\begin{proof}
We begin by noting that by Lemma \ref{blockslemma}, an $S$-class with $m_1$ $1$-blocks of $a$ and $m_2$ $2$-blocks of $a$ loses at least $(\eta+m_1)/2$ degrees of freedom, and if it contains any $b$ would have fewer degrees of freedom and hence would contribute at most $O(N^{3\eta/2-m_1/2-1})$. Thus, it suffices to consider the case when $m_2=(\eta-m_1)/2$ and there are no $b$'s. The rest of the proof is divided into two parts: we first count the number of ways to arrange a prescribed number of blocks into a cyclic product of length $2\eta$; we then count the number of ways to pair together $1$-blocks and assign indices that are consistent throughout the cyclic product.

Given $m_1=o(\eta)$, we claim that the number of ways $p(\eta)$ of arranging $m_1$ $1$-blocks and $(\eta-m_1)/2$ $2$-blocks of $a$'s into a cyclic product of length $2\eta$ is $\frac{2\eta^{m_1}}{m_1!}+O(\eta^{m_1-1})$. Indeed, there are two ways to choose the $(\eta-m_1)/2$ $2$-blocks since we can either start with $aw$ or $wa$, and there are $2^{m_1}\binom{\frac{\eta-m_1}{2}}{m_1}=\frac{\eta^{m_1}}{m_1!} + O(\eta^{m_1-1})$ ways to choose the $m_1$ 1-blocks between adjacent $2$-blocks assuming that $m_1=o(\eta)$. Moreover, as we shall see later in the proof, we require that mixed pairs containing the $1$-blocks are not placed adjacent to each other, which is possible since the number of ways of having at least one $2$-block formed from adjacent mixed pairs is $2\left(\frac{\eta-m_1}{2}\right)\binom{\frac{\eta-m_1}{2}-2}{m_1-2}=O(\eta^{m_1-1})$ and thus a lower order term. Hence, overall we get $\frac{2\eta^{m_1}}{m_1!}+O(\eta^{m_1-1})$ ways to arrange the prescribed blocks.

Now, we observe that the second and the first index respectively of two adjacent $1$-blocks are congruent mod $k$, as illustrated in the example below.

\begin{example}
Consider the configuration
\begin{align}
\cdots v_{i_1i_2}a_{i_2i_3}v_{i_3i_4}v_{i_4i_5}a_{i_5i_6}a_{i_6i_7}v_{i_7i_8}a_{i_8i_9}\cdots.
\end{align}
Since $a$'s within a $2$-block are matched together, we have $i_5=i_7$ with $i_6$ being free, and that 
\begin{align}
i_3 \equiv i_4 \equiv i_5 \equiv i_7 \equiv i_8 \textup{ (mod $k$)}.
\end{align}
\end{example}

Thus, all the indices of terms between a pair of $1$-blocks, except for those within $2$-blocks, share a congruence class.  The number of ways to specify the congruence classes of $1$-blocks and to pair the $1$-blocks up is 
\begin{align}
\sum_{1\leq i_1,i_2,...,i_{m_1}\leq k}\mathbb{E}[a_{i_1i_2}a_{i_2i_3}...a_{i_{m_1}i_1}],
\end{align}
where each $a_{ij}$ are i.i.d. $\sim \mathcal{N}(0,1)$ under the restriction that $a_{ij}=a_{ji}$ and $a_{ii}=0$ for all $i,j$. The above expression is simply the $m_1$\textsuperscript{th} expected moment of $k\times k$ hollow GOE.

Since the $m_1$ 1-blocks are paired together, then there are only $m_1$ free indices. As the congruence class of these indices are fixed, the number of choices of these indices is $(N/k)^{m_1}$. Similarly, the number of choices of indices for all the 2-blocks is $(N^2/k)^{\frac{\eta-m_1}{2}}$, since the indices of each 2-block $a_{i_\ell i_{\ell+1}}a_{i_{\ell+1}i_{\ell+2}}$ must satisfy $i_\ell=i_{\ell+2}$, and there are $N/k$ choices for $i_\ell=i_{\ell+2}$ whose congruence class is fixed and $N$ choices for $i_2$ that is free. The remaining indices are those of the weight blocks, which must satisfy congruence mod $k$ and hence are each restricted to $N/k$ choices. By the structure imposed by the anticommutator, the total number of indices of all weight blocks is $\eta-\left(\frac{\eta-m_1}{2}+m_1\right)=\frac{\eta-m_1}{2}$. Thus, the total number of ways to assign indices is
$\left(N/k\right)^{m_1} \left(N^2/k\right)^{\frac{\eta-m_1}{2}}\left(N/k\right)^{\frac{\eta-m_1}{2}}=\frac{N^{3\eta/2-m_1/2}}{k^\eta}$. After combining all these pieces, we arrive at the desired result for the contribution of a fixed $S$-class.
\end{proof}
The following two identities from \cite{split} will be useful for us in the expected blip moment calculation.
\begin{lemma}\label{inequalities}
For any $0\leq p<m$,
\begin{align}
\sum_{i=0}^m(-1)^i\binom{m}{i}i^p &\ = \ 0, \\
\sum_{i=0}^m(-1)^{m-i}\binom{m}{i}i^m &\ = \ m!.
\end{align}
\end{lemma}

Observe that if $m_1>m$, then the contribution of an $S$ class with $m_1$ 1-blocks of $a$ is
\begin{align}
&\frac{1}{2k}\sum_{\alpha=2n}^{4n}c_\alpha\left(\frac{k^2}{N^{3}}\right)^\alpha\frac{1}{N^{5m/2}}\sum_{i=0}^m\binom{m}{i}\left(-\frac{N^{3}}{k^2}\right)^{m-i}p(2(\alpha+i))\left(\frac{N^{\frac{3}{2}\cdot 2(\alpha+i)-\frac{1}{2}m_1}}{k^{2(\alpha+i)}}\right)\mathbb{E}_k\left[\textup{Tr }C^{m_1}\right] \nonumber \\
&\ = \ \frac{C_{k,m,m_1}}{N^{\frac{1}{2}(m_1-m)}}\sum_{\alpha=2n}^{4n}c_\alpha \sum_{i=0}^m\binom{m}{i}(-1)^{m-i}p(2(\alpha+i)) \nonumber \\
&\ = \ \frac{C_{k,m,m_1}}{N^{\frac{1}{2}(m_1-m)}}\sum_{\alpha=2n}^{4n}c_\alpha \sum_{i=0}^m\binom{m}{i}(-1)^{m-i}\left(\frac{2(2(\alpha+i))^{m_1}}{m_1}+\left((\alpha+i)^{m_1-1}\right)\right) \nonumber \\
&\ \ll \ \frac{C_{k,m,m_1}}{N^{\frac{1}{2}(m_1-m)}}\sum_{\alpha=2n}^{4n}c_\alpha \alpha^{m_1}.
\end{align}
Since $f^{(2n)}(x)=x^{2n}(2-x)^{2n }$, then $|c_\alpha|\ll C_0^{2n}$ for some $C_0>0$. Moreover, $\alpha=O(\log\log(N))$, then for some $\epsilon>0$ 
\begin{align}
\sum_{\alpha=2n}^{4n}c_\alpha \alpha^{m_1} \ \ll \ n^{m_1+2}C_0^{2n} \ \ll \ (\log\log(N))^{m_1+2}\log(N) \ \ll \ N^{(m_1-m)/2-\epsilon}
\end{align}

Hence, as $N\rightarrow 0$, the contribution of an $S$-class with $m_1>m$ total 1-blocks of $a$ and $m_2$ total 2-blocks of $a$ is negligible. Moreover, if $m_1<m$, then the contribution of an $S$-class with $m_1$ total 1-blocks of $a$  is
\begin{align}
&\frac{1}{2k}\sum_{\alpha=2n}^{4n}c_\alpha\left(\frac{k^2}{N^{3}}\right)^\alpha\frac{1}{N^{5m/2}}\sum_{i=0}^m\binom{m}{i}\left(-\frac{N^{3}}{k^2}\right)^{m-i}p(2(\alpha+i))\left(\frac{N^{\frac{3}{2}\cdot 2(\alpha+i)-\frac{m_1}{2}}}{k^{2(\alpha+i)}}\right)\mathbb{E}_k\left[\textup{Tr }C^{m_1}\right] \nonumber \\
&\ = \ \frac{C_{k,m,m_1}}{N^{\frac{1}{2}(m_1-m)}}\sum_{\alpha=2n}^{4n}c_\alpha\sum_{i=0}^m\binom{m}{i}(-1)^ip(2(\alpha+i)) \nonumber \\
&\ = \ \frac{C_{k,m,m_1}}{N^{\frac{1}{2}(m_1-m)}}\sum_{\alpha=2n}^{4n}c_\alpha \sum_{q=0}^{m_1}c_q \alpha^{m_1-q} \sum_{i=0}^m(-1)^i\binom{m}{i} i^q \ = \ 0,
\end{align}
where the last step follows from Lemma \ref{inequalities}. Thus, any contributing $S$-class must satisfy $m_1=m$.

\begin{theorem}\label{GOE-checkerboard Moments} Let $A_N$ be an $N\times N$ matrix sampled from GOE and $B_N$ be an $N\times N$ matrix sampled from the $k$-checkerboard ensemble, where the samplings occur independently. Then the limiting expected $m$\textup{\textsuperscript{th}} moment associated to the empirical blip spectral measure of $\{A_N, B_N\}$ around $\pm N^{3/2}/k$ is
\begin{align}
\lim_{N\rightarrow\infty}\mathbb{E}\left[\mu_{\{A_N,B_N\}}^{(m)}\right] \ = \ \frac{1}{k}\left(\frac{2}{k^2}\right)^{m}\mathbb{E}_k[\textup{Tr }C_k^m],
\end{align}
where $C_k$ is a $k\times k$ GOE and $\mathbb{E}_k[\textup{Tr }C_k^m]$ is taken over all $k\times k$ hollow GOE.
\end{theorem}
Note that we drop the dependency of the empirical blip spectral measure on $f^{(2n)}$ since the limiting expected moment is independent of the choice of the weight function.
\begin{proof}
By the discussion above, we know that $m_1=m$. Then for large $N$, the error term is negligible, and the main term gives
\begin{align}
&\lim_{N\rightarrow\infty}\mathbb{E}\left[\mu_{\{A_N,B_N\},f^{(2n)}}^{(m)}\right] \ = \ \lim_{N\rightarrow\infty}\frac{1}{2k}\sum_{\alpha=2n}^{4n}c_\alpha\left(\frac{k^2}{N^{3}}\right)^\alpha\frac{1}{N^{5m/2}}\sum_{i=0}^m \binom{m}{i}(-1)^{m-i}\left(\frac{N^3}{k^2}\right)^{m-i} \nonumber\\
&\times\frac{2(2(\alpha+i))^mN^{3/2\cdot 2(\alpha+i)-1/2m}}{m!\cdot k^{2(\alpha+i)}}\mathbb{E}_k[\textup{Tr }C^{m}] \nonumber \\
&\ = \ \frac{1}{m!\cdot k}\left(\frac{2}{k^2}\right)^{m}\mathbb{E}_k[\textup{Tr }C^m]\sum_{\alpha=2n}^{4n}c_\alpha\sum_{i=0}^m \binom{m}{i}(-1)^{m-i}(\alpha+i)^m \nonumber \\
&\ = \ \frac{1}{m!\cdot k}\left(\frac{2}{k^2}\right)^{m}\mathbb{E}_k[\textup{Tr }C^m] \sum_{\alpha=2n}^{4n}c_\alpha\sum_{i=0}^m \binom{m}{i}(-1)^{m-i}\sum_{p=0}^m \binom{m}{p}\alpha^{p}i^{m-p} \nonumber \\
&\ = \ \frac{1}{m!\cdot k}\left(\frac{2}{k^2}\right)^{m}\mathbb{E}_k[\textup{Tr }C^m]\sum_{\alpha=2n}^{4n}\sum_{p=0}^m\binom{m}{p}c_\alpha\alpha^p\sum_{i=0}^m\binom{m}{i}(-1)^{m-i}i^{m-p}.\label{MomentCalc2}
\end{align}
Since the inner sum is 0 if $p>0$ and $m!$ if $p=0$ by Lemma \ref{inequalities} and $f^{(2n)}(1)=\sum_{\alpha=2n}^{4n}c_\alpha=1$, then
\begin{align}
\lim_{N\rightarrow\infty}\mathbb{E}\left[\mu_{\{A_N,B_N\}}^{(m)}\right] &\ = \ \frac{1}{m!\cdot k}\left(\frac{2}{k^2}\right)^{m}\mathbb{E}_k[\textup{Tr }C^m]\sum_{\alpha=2n}^{4n}c_\alpha m! \nonumber \\
&\ = \ \frac{1}{k}\left(\frac{2}{k^2}\right)^{m}\mathbb{E}_k[\textup{Tr }C^m].
\end{align}
\end{proof}

%\textcolor{red}{To do: need to add the error term too}
\subsection{Limiting Expected Moments of Blips of the Anticommutator of Checkerboard and Checkerboard} By Lemma \ref{blockslemma}, for the expected moments of a blip regime, it suffices to only consider $S$-classes with $1$-blocks or $2$-blocks of $a$'s or $b$'s. In the following proposition, we determine the contributions of these $S$-classes:

\begin{prop}\label{j,k-checkerboard trace}
For $m_{1a}, m_{2a}, m_{1b}, m_{2b}\in \mathbb{Z}_{\geq 0}$, define $m_1:=m_{1a}+m_{1b}$ and $m_2:=m_{2a}+m_{2b}$. If $m_1+m_2=o(\eta)$, then the total contribution to $\E{\textup{Tr}\{A_N, B_N\}^\eta}$ of an $S$-class with $m_{1a}$ 1-blocks of $a$, $m_{1b}$ 1-blocks of $b$, $m_{2a}$ 2-blocks of $a$, and $m_{2b}$ 2-blocks of $b$ is 
\begin{equation*}
\resizebox{1.0\hsize}{!}{$\frac{2^{\eta-2m_2}\eta^{m_1+m_2}}{m_{1a}!m_{1b}!m_{2a}!m_{2b}!}2^{\frac{m_1}{2}}(m_{1a})!!(m_{1b})!!\left(\frac{1}{k}\right)^{\eta-m_{1a}-2m_{2a}}\left(\frac{1}{j}\right)^{\eta-m_{1b}-2m_{2b}}\left(1-\frac{1}{k}\right)^{\frac{m_{1a}}{2}+m_{2a}}\left(1-\frac{1}{j}\right)^{\frac{m_{1b}}{2}+m_{2b}}N^{2\eta-(m_1+m_2)}$}.
\end{equation*} 
\end{prop}

\begin{proof}
The proof is divided into two parts. First, we count the number of ways to arrange the prescribed blocks into a cyclic product of length $2\eta$ and assign the weight pairs; second, we count the number of ways to pair up all the 1-blocks and assign indices that ensures consistent indexing throughout the cyclic product.

Given $m_1+m_2=o(\eta)$, we claim that the number of ways $q(\eta)$ of arranging the prescribed blocks into a cyclic product of length $2\eta$ and assigning the adjacent weight pairs is $\frac{2^{\eta-2m_2}\eta^{m_1+m_2}}{m_{1a}!m_{1b}!m_{2a}!m_{2b}!}+O(2^\eta\eta^{m_1+m_2-1})$. 
Naively, $q(\eta)$ is simply
\begin{align}
\binom{\eta-m_2}{m_1+m_2}\binom{m_1+m_2}{m_2}2^{\eta-2m_2}\binom{m_1}{m_{1a}}\binom{m_2}{m_{2a}} \ = \ \frac{2^{\eta-2m_2}\eta^{m_1+m_2}}{m_{1a}!m_{1b}!m_{2a}!m_{2b}!}+O(2^\eta\eta^{m_1+m_2-1}).
\end{align}

That is, we first choose all the $m_1+m_2$ blocks, viewing each of the $m_2$ 2-block as a $1$-block (where $\eta-m_2$ comes from), which can be done in $\binom{\eta-m_2}{m_1+m_2}$ ways. Then, we choose the $m_2$ 2-blocks from all the $m_1+m_2$ blocks, $m_{1a}$ 1-blocks of $a$ from all the $m_1$ 1-blocks, $m_{2a}$ 2-blocks of $a$ from all the $m_2$ 2-blocks, and finally specifying the mixed pairs the 1-blocks belong to and the weight pairs, all of which can be done in $\binom{m_1+m_2}{m_2}2^{\eta-2m_2}\binom{m_1}{m_{1a}}\binom{m_2}{m_{2a}}$ ways.

However, this naive counting method fails to account for the restriction that different mixed pairs of $a$ and $v$ and of $b$ and $w$ cannot be placed adjacent to each other to form a 2-block (e.g. if two mixed pairs $va$ and $av$ are adjacent to each other, then we have a 2-block of $a$). The number of ways of having at least one 2-block formed from different mixed pairs of $a$ and $v$ and of $b$ and $w$ is $2^{\eta-2m_2-1}(\eta-m_2)\binom{\eta-m_2-2}{m_1+m_2-2}\binom{m_1+m_2}{m_2}\binom{m_1}{m_{1a}}\binom{m_2}{m_{2a}}=O(2^\eta \eta^{m_1+m_2-1})$ if $m_1+m_2=o(\eta)$. Thus,
\begin{align}
q(\eta) \ = \ \frac{2^{\eta-2m_2}\eta^{m_1+m_2}}{m_{1a}!m_{1b}!m_{2a}!m_{2b}!}+O(2^\eta\eta^{m_1+m_2-1}).
\end{align}

Similarly, we also need to guarantee that none of the 1-blocks are adjacent to each other. The number of ways of having at least two adjacent 1-blocks is $2^{\eta-2m_2}(\eta-m_2)\binom{\eta-m_2-2}{m_1+m_2-2}\binom{m_1+m_2}{m_2}\binom{m_1}{m_{1a}}\binom{m_2}{m_{2a}}=O(2^\eta \eta^{m_1+m_2-1})$ if $m_1+m_2=o(\eta)$.

Now, we count the number of ways to assign the indices for the $S$-class. In contrast to \cite{split} where there are restrictions on the indices of the 1-blocks, we demonstrate in the example below that we can remove such restrictions.

\begin{example}
Consider a weight block surrounded by two non-weight terms $c_{i_{\ell-1} i_{\ell}} c_{i_{\ell}i_{\ell+1}}\cdots c_{i_{t}i_{t+1}} c_{i_{t+1} i_{t+2}}$, where $t-\ell$ is sufficiently large. We assume without loss of generality that $c_{i_{\ell-1} i_{\ell}}=a_{i_{\ell-1} i_{\ell}}$ and $c_{i_{t+1} i_{t+2}}=a_{i_{t+1} i_{t+2}}$. After specifying $i_{\ell}$, if $c_{i_{\ell}i_{\ell+1}}=w_{i_{\ell}i_{\ell+1}}$, then $i_{\ell}\equiv i_{\ell+1}\textup{ (mod $k$)}$ and there are $N/k$ choices of indices for $i_{\ell+1}$; if $c_{i_{\ell}i_{\ell+1}}=v_{i_{\ell}i_{\ell+1}}$, then $i_{\ell}\equiv i_{\ell+1}\textup{ (mod $j$)}$ and there are $N/j$ choices of indices for $i_{\ell+1}$. After specifying $i_{t+1}$, there are similar number of choice of indices for $i_{t}$. Since $t-\ell$ is sufficiently large, then with high probability there exists $\ell+1\leq s\leq t-2$ such that $\{c_{i_s i_{s+1}}, c_{i_{s+1}i_{s+2}}\}=\{w_{i_s i_{s+1}}, v_{i_{s+1}i_{s+2}}\}$. We can specify the indices $i_{\ell+2}, \dots, i_s$ and $i_{s+2}, \dots, i_{t-1}$ the same way as before. Then we have $i_{s+1}\equiv i_s\textup{ (mod $k$)}$ and $i_{s+1}\equiv i_{s+2}\textup{ (mod $j$)}$. Since $\textup{gcd}(k, j)=1$ and $jk\mid N$, then by Chinese remainder theorem, there are $N/kj$ choices of indices $i_{s+1}$. If the number of $w$'s and $v$'s in the weight block is $r$ and $t-\ell+1-r$, respectively, then there are $\left(1/k\right)^{r}\left(1/j\right)^{t-\ell+1-r}N^{t-\ell}$ ways of specifying the $i_{\ell+1}, \dots, i_{t}$. Thus, regardless of the indices we specify for the two non-weight terms surrounding a weight block, we can guarantee with high probability consistency of indexing throughout the weight block.
\end{example}

We know that the total number of $w$'s and $v$'s in a cyclic product of length $2\eta$ is $\eta-m_{1a}-2m_{2a}$ and $\eta-m_{1b}-2m_{2b}$, respectively. Then the number of choices of congruence classes of indices for all the $w$'s and $v$'s has the corresponding factors $\left(1/k\right)^{\eta-m_{1a}-2m_{2a}}$ and $\left(1/j\right)^{\eta-m_{1b}-2m_{2b}}$. Now, by Lemma \ref{blockslemma}, each 1-block is paired up with another 1-block and the two terms of each 2-block are paired up with each other. Moreover, the indices of any non-weight terms $a_{i_\ell i_{\ell+1}}$ and $b_{i_t i_{t+1}}$ must satisfy the modular restrictions $i_\ell\not\equiv i_{\ell+1}\textup{ (mod $k$)}$ and $i_t\not\equiv i_{t+1}\textup{ (mod $j$)}$. Then similarly, the number of choices of congruence classes of indices for all the $a$'s and $b$'s has the corresponding factors $\left(1-1/k\right)^{\frac{m_{1a}}{2}+m_{2a}}$ and $\left(1-1/j\right)^{\frac{m_{1b}}{2}+m_{2b}}$. Since the loss of degrees of freedom per block in a contributing configuration is 1, then the contribution of actually specifying all the indices is $N^{2\eta-(m_1+m_2)}$. Thus, the number of ways of assigning the indices that guarantees consistent indexing is
\begin{align}
\left(\frac{1}{k}\right)^{\eta-m_{1a}-2m_{2a}}\left(\frac{1}{j}\right)^{\eta-m_{1b}-2m_{2b}}\left(1-\frac{1}{k}\right)^{\frac{m_{1a}}{2}+m_{2a}}\left(1-\frac{1}{j}\right)^{\frac{m_{1b}}{2}+m_{2b}}N^{2\eta-(m_1+m_2)}.
\end{align}

Finally, since there are $m_{1a}$ 1-block of $a$ and $m_{1b}$ 1-block of $b$, and there are no restrictions on their indices, then the number of ways of matching up all the 1-blocks is $2^{\frac{m_1}{2}}(m_{1a})!!(m_{1b})!!$. Note that the $2^{\frac{m_1}{2}}$ factor is due to the fact that for any two paired 1-blocks $c_{i_\ell i_{\ell+1}}, c_{i_t, i_{t+1}}$, we either have $i_\ell=i_{t+1}$ and $i_{\ell+1}=i_t$ or $i_{\ell}=i_t$ and $i_{\ell+1}=i_{t+1}$. This completes the proof.
\end{proof}

We are now ready to calculate the limiting expected moments for the largest blip.

\begin{theorem}
Let $A_N$ be an $N\times N$ matrix sampled from the $k$-checkerboard ensemble and $B_N$ be an $N\times N$ matrix sampled from the $j$-checkerboard ensemble, where the samplings occur independently. Then the limiting expected $m$\textup{\textsuperscript{th}} moment of the largest blip spectral measure of $\{A_N, B_N\}$ around $2N^2/kj$ is
\begin{align}
\lim_{N\rightarrow\infty}\mathbb{E}\left[\mu_{\{A_N,B_N\}}^{(m)}\right] \ = \ \sum_{\substack{m_{1a}+m_{1b}+m_{2a}+m_{2b}=m; \\ m_{1a},m_{1b}\textup{ even}}}&C(m, m_{1a}, m_{2a}, m_{1b}, m_{2b}),\end{align}
where 
\begin{align}
&C(m, m_{1a}, m_{2a}, m_{1b}, m_{2b}) \ = \ \nonumber\\
&m!\left(\frac{2}{jk}\right)^m\frac{2^{\frac{m_{1a}+m_{1b}}{2}-2(m_{2a}+m_{2b})}m_{1a}!!m_{1b}!!}{m_{1a}!m_{1b}!m_{2a}!m_{2b}!}\left(k\sqrt{1-\frac{1}{k}}\right)^{m_{1a}+2m_{2a}}\left(j\sqrt{1-\frac{1}{j}}\right)^{m_{1b}+2m_{2b}}.  
\end{align}

\end{theorem}

\begin{proof}

By Lemma \ref{blockslemma}, it suffices to consider the contributions from $S$-classes with 1-blocks and 2-blocks. Applying Proposition \ref{j,k-checkerboard trace} to \eqref{momentOfCNDN} by letting $\eta=\alpha+i$, we have
\begin{align}\label{largestblipmoment}
&\lim_{N\rightarrow\infty}\mathbb{E}\left[\mu^{(m)}_{\{A_N,B_N\},f^{(2n)}}\right] \ = \ \lim_{N\rightarrow\infty}\sum_{\alpha=2n}^{4n}c_\alpha \left(\frac{jk}{2N^2}\right)^\alpha \frac{N^{-(m_1+m_2)}}{N^m}\sum_{i=0}^m\binom{m}{i}\left(-\frac{2N^2}{jk}\right)^{m-i}\left(\frac{2N^2}{kj}\right)^{\alpha+i}
\nonumber \\
&\sum_{\substack{m_{1a}, m_{1b}, m_{2a}, m_{2b} \\ m_{1a}, m_{1b}\textup{ even}}} \frac{2^{-2m_2+\frac{m_1}{2}}(m_{1a})!!(m_{1b})!!}{m_{1a}!m_{1b}!m_{2a}!m_{2b}!}\left(k\sqrt{1-\frac{1}{k}}\right)^{m_{1a}+2m_{2a}}\left(j\sqrt{1-\frac{1}{j}}\right)^{m_{1b}+2m_{2b}}(\alpha+i)^{m_1+m_2}.
\end{align}
In \eqref{largestblipmoment}, the power of $N$ in each term of the sum is $N^{m_1+m_2-m}$, which is independent of $\alpha$. We require $m_1+m_2\leq m$, since otherwise the contribution from $m_1+m_2>m$ vanishes in the limit. Moreover, by Lemma \ref{inequalities}, the sum $\sum_{i=0}^m\binom{m}{i}(-1)^{m-i}(\alpha+i)^{m_1+m_2}$ vanishes except for the $m^{th}$ power of $i$. Hence, the contribution to the moment in the limit only comes from $m_1+m_2=m$. After combining terms and canceling out the dependency on $i$, and noting that the weight function $g(x)$ satisfies $g(1)=\sum_{\alpha=2n}^{4n}c_{\alpha} x^\alpha$, we have
\begin{align}
&\mathbb{E}\left[\mu^{(m)}_{\{A_N,B_N\}}\right] \ = \ \sum_{\substack{m_{1a}+m_{1b}+m_{2a}+m_{2b}=m; \\ m_{1a},m_{1b}\textup{ even}}}C(m, m_{1a}, m_{2a}, m_{1b}, m_{2b}).
\end{align}
\end{proof}

Even though the combinatorics and cancellation techniques have successfully yielded the limiting expected moments of the empirical blip spectral measures for the blip regime of $\{\textup{GOE, $k$-checkerboard}\}$ and the largest blip regime of $\{\textup{$k$-checkerboard, $j$-checkerboard}\}$, they fail to provide us with the limiting expected moments of the intermediary blip regimes of $\{\textup{$k$-checkerboard, $j$-checkerboard}\}$. To see this, we need to first define the empirical spectral measure for each intermediary blip regime.

\begin{definition}
Let $n=\log\log(N)$ and consider the weight function $g_s^{(2n)}$ as defined in Appendix \ref{weight polynomial}. Let $h_s=k-1$ if $s=1$ and $h_s=j-1$ if $s=2$. Suppose that $A_N$ is an $N\times N$ matrix sampled from the $k$-checkerboard ensemble and $B_N$ is an $N\times N$ matrix sampled from the $j$-checkerboard ensemble, where the samplings occur independently. Then the $g_s^{(2n)}$-\textbf{weighted intermediary blip spectral measure} associated to $\{A_N, B_N\}$ around $\pm w_sN^{3/2}$ is
\begin{align}
\mu_{\{A_N, B_N\},g_s^{(2n)}}(x)dx \ = \ \frac{1}{2h_s}\sum_{\lambda\textup{ eigenvalue}}g^{(2n)}_s\left(\frac{\lambda^2}{w_s^2N^3}\right)\delta\left(x-\frac{\lambda^2-w_s^2N^3}{N^{5/2}}\right)dx.
\end{align}
\end{definition}
Then, by eigenvalue trace lemma,
\begin{align}
\mathbb{E}\left[\mu^{(m)}_{\{A_N, B_N\},g_s^{(2n)}}\right] &\ = \ \mathbb{E}\left[\frac{1}{2h_s}\sum_{\lambda\textup{ eigenvalue}}\sum_{\alpha=2n}^{14n}c_{\alpha}\left(\frac{\lambda^2}{w_s^2N^3}\right)^\alpha\left(\frac{\lambda^2-w_s^2N^3}{N^{5/2}}\right)^m\right] \\
&\ = \ \frac{1}{2h_s}\sum_{\alpha=2n}^{14n}c_\alpha\left(\frac{1}{w_s^2N^3}\right)^\alpha\frac{1}{N^{5m/2}}\sum_{i=0}^m (-w_s^2N^3)^{m-i}\mathbb{E}\left[\textup{Tr}\left(\{A_N, B_N\}^{2(\alpha+i)}\right)\right].
\end{align}
Hence, the intermediary blip moment calculation has again become a combinatorial problem of counting and leveraging the expected value of different configurations of cyclic products. Since a blip eigenvalue in either intermediary blip regimes is of size $\Theta(N^{3/2})$, then the contribution of such eigenvalue to $\mathbb{E}\left[\textup{Tr}\left(\{A_N, B_N\}^{2(\alpha+i)}\right)\right]$ is of size $\Theta(N^{3(\alpha+i)})$, thus a contributing $S$-class for either intermediary blip regime should be $\Theta(N^{3(\alpha+i)-m})$ for some nonnegative integer $m=o(\alpha)$. One may be tempted to apply the same combinatorial and cancellation techniques to the intermediary blip moment calculation. However, there are two problems with this approach.

\begin{enumerate}
\item Among all the $S$-classes in $\mathbb{E}\left[\textup{Tr }\left(\{A_N, B_N\}\right)^{2(\alpha+i)}\right]$, there are $\Theta(\alpha)$ $S$-classes whose contribution is $\omega(N^{3/2(\alpha+i)-m})$ for each nonnegative $m=o(\alpha)$. These $S$-classes, which are typically a mixture of $wv$, $vw$ pairs, $2$-blocks of $a$ and $b$, and $1$-blocks of $a$ and $b$, are particularly difficult to cancel out using current cancellation techniques.
\item Since a blip eigenvalue in either intermediary blip regimes is of size $\Theta(N^{3/2})$, then there are $\Theta(\alpha)$ $S$-classes in $\mathbb{E}\left[\textup{Tr}\left(\{A_N, B_N\}^{2(\alpha+i)}\right)\right]$ whose contributions are all $\Theta(N^{3/2(\alpha+i)-m})$ for each nonnegative $m=o(\alpha)$. These $S$-classes typically consists of a mixture of $2$-blocks of $a$ and $b$ and $1$-blocks of $a$ and $b$. It remains difficult to differentiate which $S$-classes among these are contributing to $\mathbb{E}\left[\mu^{(m)}_{\{A_N,B_N\},g_s^{(2n)}}\right]$ and which are not.
\end{enumerate}
Note that one would encounter the same problems even if they take $k=O(N^\epsilon)$ or $j=O(N^\epsilon)$ for $0< \epsilon<1/2$, despite creating a difference in the order of $N$ of the aforementioned $S$-classes. Nevertheless, if we take $k=\Theta(\sqrt{N})$ or $j=\Theta(\sqrt{N})$, then either all the eigenvalues are in the bulk, or there are $\Theta(\sqrt{N})$ blip eigenvalues of size $\Theta(N)$ outside the bulk. In either case, standard techniques (without weight function) will allow us to find the limiting expected moments as all the eigenvalues are of the same size $\Theta(N)$.
\appendix
\section{Moments of anticommutators}\label{AppendixMomentsAnti}
\subsection{Moments of Anticommutators}
\begin{prop}
For sequences $f_0,f_1,\dots,f_\ell$ defined such that $f_0(0)=1$, and $f_0(1)=f_1(1)=f_2(1)=\dots=f_\ell(1)=1$, and with recurrence relations for $m>1$ given by
\begin{align}
&f_0(m) \ = \ f_1(m)+ \ell! \sum_{j=1}^{m-1}f_1(j)f_0(m-j),\nonumber\\
&f_k(m) \ = \ f_{k+1}(m)+\sum_{\substack{1\leq x_1,x_2\leq m\\ x_1+x_2\leq m}}(\ell-k)!(k-1)!f_{k+1}(x_1)f_{k+1}(x_2)f_{\ell-k-1}(m-x_1-x_2+1)
\end{align}
for any $0<k<\ell-1$, and by
\begin{align}
&f_{\ell-1}(m) \ = \ f_\ell(m)+\nonumber\\&\sum_{\substack{0\leq x_1,x_2<m-1\\ x_1+x_2<m-1}} (\ell-1)!(1+(\ell!-1)\cdot \mathbbm{1}_{x_1>0})(1+(\ell!-1)\cdot \mathbbm{1}_{x_2>0})f_0(x_1)f_0(x_2)f_1(m-x_1-x_2-1),\nonumber
\\
&f_\ell(m) \ = \ \ell! \cdot f_0(m-1),
\end{align}
the $2m$\textup{\textsuperscript{th}} moment of the $\ell$-anticommutator is 
\begin{align}
M_{2m} \ = \ \ell! \cdot f_0(m)
\end{align}
\end{prop}

\begin{proof}
The proof follows similarly as with the $2$-anticommutator. Let $f_0(m)$ be the number of non-crossing matchings with respect to all $(\ell, 2\ell m)$-configurations starting with $a^{(1)}_{i_1i_2}a^{(2)}_{i_2i_3}\cdots a^{(\ell)}_{i_\ell i_{\ell+1}}$ and let $f_k(m)$ be the number of such matchings where the first $k$ terms are paired with the last $k$ terms in a nested fashion (i.e., for $k=3$, we would have configurations of the form $a^{(1)}_{i_1i_2}a^{(2)}_{i_2i_3}a^{(3)}_{i_3i_4}\cdots a^{(3)}_{i_{2\ell m-2}i_{2\ell m-1}}a^{(2)}_{i_{2\ell m-1}i_{2\ell m}}a^{(1)}_{i_{2\ell m}i_{1}}$ and matchings such that $i_2=i_{2\ell m}$, $i_3=i_{2\ell m-1}$, and $i_4=i_{2\ell m-2}$).

We first find the recurrence relation for $f_0(m)$. To ensure non-crossing matchings, $a^{(1)}_{i_1 i_2}$ must be paired with some $a^{(1)}_{i_{2\ell j} i_{2\ell j+1}}$ with $j\leq m$ (in the case when $j=m$, we identify $2\ell m+1$ as $1$). When $j=m$, the number of non-crossing matchings is simply $f_1(m)$ by definition. When $j<m$, the number of non-crossing matchings within $a^{(1)}_{i_1 i_2}\cdots a^{(1)}_{i_{2\ell j} i_{2\ell j+1}}$ is $f_1(j)$, while the number of non-crossing matchings within the rest of the cyclic product for which we have no restrictions is $\ell!f_0(m-j)$, with the $\ell!$ accounting for different possible arrangements of the first $\ell$ terms. Multiplying these together and summing over all possible $j$'s gives
\begin{align}
f_0(m) \ = \ f_1(m)+\ell! \sum_{j=1}^{m-1}f_1(j)f_0(m-j).
\end{align}

Now turning to $f_k(m)$, we look separately at when $0<k<\ell-1$, when $k=\ell-1$, and when $k=\ell$.

When $0<k<\ell-1$, we have either that $a^{(k+1)}_{i_{k+1}i_{k+2}}$ is paired with $a^{(k+1)}_{i_{2\ell m-k}i_{2\ell m-k+1}}$, or that $a^{(k+1)}_{i_{k+1}i_{k+2}}$ is paired with $a^{(k+1)}_{i_{2\ell x_1 -k}i_{2\ell x_1-k+1}}$ and $a^{(n)}_{i_{2\ell m-k}i_{2\ell m-k+1}}$ is paired with $a^{(n)}_{i_{2\ell (m-x_2) + k+1}i_{2\ell(m- x_2) + k+2}}$, where $n\in \{k+1, \dotsc, \ell\}$, with both $x_1, x_2 \geq 1$ and $2\ell x_1-k+1 < 2\ell (m-x_2) + k+1$, or $x_1+x_2 \leq m$. The first case is precisely the definition of $f_{k+1}(m)$. In the second case, the number of non-crossing matchings of terms between $a^{(k+1)}_{i_{k+1}i_{k+2}}$ and $a^{(k+1)}_{i_{2\ell x_1 -k}i_{2\ell x_1-k+1}}$ is $f_{k+1}(x_1)$, the number of non-crossing matchings of terms between $a^{(n)}_{i_{2\ell (m-x_2) + k+1}i_{2\ell(m- x_2) + k+2}}$ and $a^{(n)}_{i_{2\ell m-k}i_{2\ell m-k+1}}$ is $(\ell-k)!f_{k+1}(x_2)$, with the $(\ell-k)!$ accounting for different possible arrangements of the last $\ell$ terms, and the number of non-crossing matchings of terms between $a^{(k+1)}_{i_{2\ell x_1 -k}i_{2\ell x_1-k+1}}$ and $a^{(n)}_{i_{2\ell (m-x_2) + k+1}i_{2\ell(m- x_2) + k+2}}$ is $(k-1)!f_{\ell-k+1}(m-x_1-x_2+1)$. The last statement follows from viewing the $2\ell(m-x_1-x_2)+2k+2$ terms between $a^{(k+1)}_{i_{2\ell x_1 -k}i_{2\ell x_1-k+1}}$ and $a^{(n)}_{i_{2\ell (m-x_2) + k+1}i_{2\ell(m- x_2) + k+2}}$ as $2\ell (m-x_1-x_2+1)$ terms where the $(l-k-1)$ terms on both end are matched to each other and hence fixed, with the $(k-1)!$ accounting for different permutations of the remaining $k+1$ of the first  $\ell$ terms. We sum over all possible $x_1$ and $x_2$'s to get the desired result:
\begin{align}
f_k(m) \ = \ f_{k+1}(m)+
 \sum_{\substack{1\leq x_1,x_2\leq m\\ x_1+x_2\leq m}}(\ell-k)!(k-1)!f_{k+1}(x_1)f_{k+1}(x_2)f_{\ell-k-1}(m-x_1-x_2+1).
\end{align}

When $k=\ell-1$, we either have that $a^{(\ell)}_{i_{\ell}i_{\ell+1}}$ is paired with $a^{(\ell)}_{i_{2\ell m - \ell+1}i_{2\ell m - \ell+2}}$, which is simply $f_\ell(m)$ by definition, or that $a^{(\ell)}_{i_{\ell}i_{\ell+1}}$ is paired with $a^{\ell}_{i_{2\ell x_1 +\ell+1}i_{2\ell x_1+\ell+2}}$ and $a^{(n)}_{i_{2\ell m-\ell+1}i_{2\ell m-\ell+2}}$ is paired with $\\a^{(n)}_{i_{2\ell (m-x_2) - \ell}i_{2\ell(m- x_2) - \ell+1}}$, with both $x_1, x_2 \geq 0$. The number of non-crossing matchings of terms between $a^{(\ell)}_{i_{\ell}i_{\ell+1}}$ and $a^{(\ell)}_{i_{2\ell x_1 +\ell+1}i_{2\ell x_1+\ell+2}}$ is $(1+(\ell!-1)\mathbbm{1}_{x_1>0})f_0(x_1)$, the number of non-crossing matchings of terms between $a^{(n)}_{i_{2\ell (m-x_2) - \ell}i_{2\ell(m- x_2) - \ell+1}}$ and $a^{(n)}_{i_{2\ell m-\ell+1}i_{2\ell m-\ell+2}}$ is  $(1+(\ell!-1)\mathbbm{1}_{x_2>0})f_0(x_2)$, with a factor of $\ell!$ when either $x_1,x_2 > 0$ due to different possible arrangements of the first $\ell$ terms starting at $a^{(n)}_{i_{\ell+1}i_{\ell+2}}$, where $n \in \{1, \dotsc, \ell-1\}$, and the number of non-crossing matchings of terms between $a^{(\ell)}_{i_{2\ell x_1 +\ell+1}i_{2\ell x_1+\ell+2}}$ and $a^{(n)}_{i_{2\ell (m-x_2) - \ell}i_{2\ell(m- x_2) - \ell+1}}$ is $(\ell-1)!f_1(m-x_1-x_2-1)$. For the last statement, we view the $2\ell(m-x_1-x_2)-2\ell-2$ terms between $a^{(\ell)}_{i_{2\ell x_1 +\ell+1}i_{2\ell x_1+\ell+2}}$ and $a^{(n)}_{i_{2\ell (m-x_2) - \ell}i_{2\ell(m- x_2) - \ell+1}}$ as $2\ell(m-x_1-x_2-1)$ terms with the first and last term matched with each other, with $(l-1)!$ accounting for different arrangements of the remaining $l-1$ terms. We once again sum over all possible $x_1$ and $x_2$'s to reach the desired result:
\begin{align}
&f_{\ell-1}(m) \ = \ f_\ell(m)+\nonumber
\\
&\sum_{\substack{0\leq x_1,x_2<m-1\\ x_1+x_2<m-1}} (\ell-1)!(1+(\ell!-1)\cdot \mathbbm{1}_{x_1>0})(1+(\ell!-1)\cdot \mathbbm{1}_{x_2>0})f_0(x_1)f_0(x_2)f_1(m-x_1-x_2-1).
\end{align}

Lastly, when $k=\ell$, with no matching conditions on the terms between the first and last $\ell$ terms, for each possible permutation of the next $\ell$ terms, we have $f_0(m-1)$ non-crossing matchings, amounting to $\ell!\cdot f_0(m-1)$ total non-crossing matchings.

We have now fully defined our recurrences for $f_0(k)$, which represents the number of non-crossing matchings with respect to $(\ell, 2\ell m)$-configurations where the first $\ell$ terms are fixed to be $a^{(1)}_{i_1i_2}a^{(2)}_{i_2i_3}\cdots a^{(\ell)}_{i_\ell i_{\ell+1}}$. Applying any permutation to these $\ell$ terms preserves the non-crossing property of these matchings. Hence, we multiply $f_0(m)$ by $\ell!$ to obtain all possible non-crossing partitions with respect to $(\ell, 2\ell m)$-configurations, and we arrive at the even moments being $M_{2k}=\ell! \cdot f_0(k)$.
\end{proof}

\subsection{Bulk Moments of the Anticommutator of GOE and Checkerboard}

\begin{proposition}\label{bulkGOEcheckerboard} The $2m$\textup{\textsuperscript{th}} bulk moment of $\{\textup{GOE, }k\textup{-checkerboard}\}$ is $M_{2m}=2(1-\frac{1}{k})^mf(m)$, where $f(0)=f(1)=1$, $g(1)=1$, and
\begin{align}
f(m) \ = \ g(m)+2\sum_{j=1}^{m-1}g(j)f(m-j)
\end{align}
and 
\begin{align}
g(m) \ = \ 2f(m-1) + \sum_{\substack{0\leq x_1,x_2\leq m-2\\ x_1+x_2\leq m-2}}(1+\mathbbm{1}_{x_1>0})(1+\mathbbm{1}_{x_2>0})f(x_1)f(x_2)g(m-1-x_1-x_2).
\end{align}
\end{proposition}

\begin{proof}
By a result in \cite{Tao1}, the limiting distribution of the bulk of $\{\textup{GOE, }(k,1)\textup{-checkerboard}\}$ is given by the limiting distribution of $\{\textup{GOE, }(k,0)\textup{-checkerboard}\}$. Because in a contributing cyclic product, every term $c_{i_\ell i_{\ell+1}}$ from the $(k, 0)$-checkerboard must be non-weight with the modular restriction $i_\ell\not\equiv i_{\ell+1}\textup{ (mod $k$)}$, then the $2m$\textsuperscript{th} bulk moment of $\{\textup{GOE, }k\textup{-checkerboard}\}$ is essentially the $2m$\textsuperscript{th} moment $\{\textup{GOE,}\\ \textup{GOE}\}$, except that we have to account for all the modular restrictions. Since the $2m$ non-weight terms are paired together, the probability that all the terms from the $(k,0)$-checkerboard are non-weights is $\left(1-\frac{1}{k}\right)^m$. This completes the proof.
\end{proof}

\subsection{Bulk Moments of the Anticommutator of Checkerboard and Checkerboard}
\begin{corollary} The $2m$\textup{\textsuperscript{th}} bulk moment of $\{k\textup{-checkerboard, }j\textup{-checkerboard}\}$ is $M_{2m}=2\left(1-\frac{1}{k}\right)^m\\ \left(1-\frac{1}{j}\right)^m f(m)$, where
\begin{align}
f(m) \ = \ g(m)+2\sum_{j=1}^{m-1}g(j)f(m-j)
\end{align}
and 
\begin{align}
g(m) \ = \ 2f(m-1) + \sum_{\substack{0\leq x_1,x_2\leq m-2\\ x_1+x_2\leq m-2}}(1+\mathbbm{1}_{x_1>0})(1+\mathbbm{1}_{x_2>0})f(x_1)f(x_2)g(m-1-x_1-x_2).
\end{align}
\end{corollary}

\begin{proof}
The proof is essentially the same as the proof of Proposition \ref{bulkGOEcheckerboard}.
\end{proof}

\section{Proof of Multiple Regimes}\label{multipleregimes} 

In this section, we prove the existence of multiple regimes of eigenvalues for $\{\textup{GOE, }k\textup{-checkerboard}\}$ and $\{k\textup{-checkerbord, }j\textup{-checkerboard}\}$. Our method involves decomposing each checkerboard matrix into the sum of its mean matrix and perturbation matrix and applying Weyl's inequality to bound the eigenvalue of the matrix ensemble in terms of the eigenvalue of its components. For the sake of simplicity, throughout this section we assume that the weight $w=1$ and that $k|N$ for $\{\textup{GOE, }k\textup{-checkerboard}\}$ and $jk|N$ for $k\{\textup{-checkerboard, }j\textup{-checkerboard}\}$.

\begin{definition}[Mean Matrix]
The mean matrix $\overline{A}_N$ of the $k$-checkerboard matrix $A_N=(a_{ij})$ is given by
\begin{align}
\overline{a}_{ij} \ = \ 
\begin{cases}
0,& \text{if }i\not\equiv j\mod{k}\\
1,& \text{if }i\equiv j\mod{k}.
\end{cases}
\end{align}
We note that the rank of $\overline{A}_N$ is $k$.
\end{definition}

\begin{definition}[Perturbation Matrix]
The perturbation matrix $\Tilde{A}_N$ of the k-checkerboard matrix $A_N=\Tilde{A}_N=(\Tilde{a}_{ij})$ is given by
\begin{align}
\Tilde{a}_{ij} \ = \
\begin{cases}
a_{ij},& \text{if }i\not\equiv j\mod{k}\\
0,& \text{if }i\equiv j\mod{k}.
\end{cases}
\end{align}
\end{definition}

Thus, we can write the $k$-checkerboard matrix simply as $A_N=\overline{A}_N+\Tilde{A}_N$. As shown in \cite{split},  all the eigenvalues of $\Tilde{A}_N$ are $O(N^{1/2})$. Moreover, $\overline{A}_N$ has $k$ eigenvalues at $\frac{N}{k}$ and $N-k$ eigenvalues at $0$.

\begin{lemma} 
Let $\Tilde{A}_N$ be the perturbation matrix as defined above and $B_N$ an $N\times N$ GOE matrix, then with probability $1-o(1)$, all the eigenvalues of $\{\Tilde{A}_N, B_N\}$ are $O(N)$.
\end{lemma}

\begin{proof}

We know that the maximum eigenvalue of a matrix is equal to the operator norm of the matrix. Since $||\Tilde{A}_N||=O(N^{1/2})$ and $||B_N||=O(N^{1/2})$, then by submultiplicativity of the matrix norm, $||\Tilde{A}_NB_N||\leq ||\Tilde{A}_N||||B_N||=O(N)$. Similarly, $||B_N\Tilde{A_N}||=O(N)$. By Weyl's inequality, $\lambda_N(\{\Tilde{A}_N, B_N\})\leq \lambda_N(\Tilde{A}_NB_N)\\+\lambda_N(B_N\Tilde{A}_N)=O(N)$. The argument for the smallest negative eigenvalues follows from considering $-\Tilde{A}_N$ and $-B_N$.
\end{proof}

\begin{lemma}
Let $\overline{A}_N$ be the mean matrix as defined above and $B_N$ the $N\times N$ GOE matrix, then the largest eigenvalue of $\{\overline{A}_N, B_N\}$ is bounded above by $\frac{4N^{3/2}}{k}$, the smallest eigenvalue is bounded below by $-\frac{4N^{3/2}}{k}$, and there are at least $N-2k$ eigenvalues at 0.
\end{lemma}

\begin{proof}
First, observe that $\textup{rank}(\overline{A}_N B_N)\leq \textup{min}(\textup{rank}(\overline{A}_N), \textup{rank}(B_N))=k$. Similarly, $\textup{rank}(B_N\overline{A}_N)\leq k$. By the subadditivity of rank, $\textup{rank}(\{\overline{A}_N, B_N\})\leq 2k$. Thus, at least $N-2k$ eigenvalues are $0$. For the highest eigenvalues, we see that $||\{\overline{A}_N, B_N\}||\leq 2||\overline{A}_N||||B_N||=2\cdot \frac{N}{k}\cdot 2N^{1/2}=\frac{4N^{3/2}}{k}$. Similarly, the smallest eigenvalue is bounded below by $-\frac{4N^{3/2}}{k}$.
\end{proof}

Empirically, we observe that $\overline{A}_N$ has $k$ blip eigenvalues at $\frac{N^{3/2}}{k}+O(N)$ and $k$ blip eigenvalues at $-\frac{N^{3/2}}{k}+O(N)$. By assuming this, we are able to prove the existence of multiple regimes for $\{A_N, B_N\}$, as follows:

\begin{lemma}
Let $A_N$ be a $k$-checkerboard matrix and $B_N$ an $N\times N$ GOE matrix, then $\{A_N,B_N\}$ has a blip containing $k$ eigenvalues at $\frac{N^{3/2}}{k} + O(N)$, a blip containing $k$ eigenvalues at $-\frac{N^{3/2}}{k}+ O(N)$, and $N-2k$ eigenvalues of $O(N)$.
\end{lemma}

\begin{proof}
First note that we can write $\{A_N,B_N\}=\{\overline{A}_N,B_N\} + \{\Tilde{A}_N,B_N\}$. By Weyl's inequality, we see that 
\begin{align}
\lambda_{N-k+1}(\{A_N,B_N\})\ \geq \ \lambda_{N-k+1}(\{\overline{A}_N,B_N\}) + \lambda_1(\{\Tilde{A}_N,B_N\})\ = \ \frac{1}{k}N^{3/2}+O(N)
\end{align}
and 
\begin{align}
\lambda_N(\{A_N,B_N\})\ \leq \ \lambda_N(\{\overline{A}_N,B_N\}) + \lambda_N(\{\Tilde{A}_N,B_N\})\ = \ \frac{1}{k}N^{3/2}+O(N).
\end{align}

So this proves the existence of $k$ blip eigenvalues at $\frac{N^{3/2}}{k}$. Similarly, we can use Weyl's inequality to show the existence of blip eigenvalues at $-\frac{N^{3/2}}{k}$. For the bulk, we see that 
\begin{align}
\lambda_{N-k}(\{A_N,B_N\})\ \leq \ \lambda_{N-k}(\{\overline{A}_N,B_N\}) + \lambda_N(\{\Tilde{A}_N,B_N\})\ = \ O(N),
\end{align}
and 
\begin{align}
\lambda_{k+1}(\{A_N,B_N\})\ \geq \ \lambda_{k+1}(\{\overline{A}_N,B_N\}) + \lambda_1(\{\Tilde{A}_N,B_N\})\ = \ O(N).
\end{align}
This completes the proof for the existence of three different regimes. 
\end{proof}

Now we consider $\{A_N, B_N\}$, where $A_N$ is a $k$-checkerboard matrix and $B_N$ is a $j$-checkerboard matrix. We assume $\textup{gcd}(k,j)=1$, $N\mid kj$. Then we can write $\{A_N,B_N\}=\{\Tilde{A}_N,\Tilde{B}_N\}+\{\overline{A}_N,\Tilde{B}_N\}+\{\Tilde{A}_N,\overline{B}_N\}+\{\overline{A}_N, \overline{B}_N\}$. Similarly, we see that all eigenvalues of $\{\Tilde{A}_N,\Tilde{B}_N\}$ are of $O(N)$. Empirically, $\{\overline{A}_N,\Tilde{B}_N\}$ has $k$ eigenvalues at $\frac{1}{k}\sqrt{1-\frac{1}{j}}N^{3/2}$, $k$ eigenvalues at $-\frac{1}{k}\sqrt{1-\frac{1}{j}}N^{3/2}$, and the remaining $N-2k$ eigenvalues at $0$. Heuristically, this can be seen from the fact that $\overline{A}_N$ has $k$ eigenvalues at $\frac{1}{k}N$ and the eigenvalues of $\Tilde{B}_N$ are bounded above and below by $\pm 2\sqrt{1-\frac{1}{j}}N^{1/2}$. Similarly, empirically $\{\Tilde{A}_N, \overline{B}_N\}$ has $j$ eigenvalues at $\frac{1}{j}\sqrt{1-\frac{1}{k}}N^{3/2}$, $j$ eigenvalues at $-\frac{1}{j}\sqrt{1-\frac{1}{k}}N^{3/2}$, and the remaining $N-2j$ eigenvalues are at 0. For $\{\overline{A}_N, \Tilde{B}_N\}+\{\Tilde{A}_N, \overline{B}_N\}$, we observe that the largest eigenvalue is of $O(N^{3/2})$ but larger than $\frac{1}{k}\sqrt{1-\frac{1}{j}}N^{3/2}$ and $\frac{1}{j}\sqrt{1-\frac{1}{j}}N^{3/2}$, the smallest eigenvalue is of $O(N^{3/2})$ but smaller than $-\frac{1}{k}\sqrt{1-\frac{1}{j}}N^{3/2}$ and $-\frac{1}{j}\sqrt{1-\frac{1}{k}}N^{3/2}$. Furthermore, There are $k-1$ eigenvalues at each of $\pm \frac{1}{k}\sqrt{1-\frac{1}{j}}N^{3/2}$, and $j-1$ eigenvalues at each of $\pm \frac{1}{j}\sqrt{1-\frac{1}{k}}N^{3/2}$, and the remaining $N-2k-2j+3$ eigenvalues of $O(N)$.

\begin{lemma}
Let $\overline{A}_N$ and $\overline{B}_N$ be average matrices as defined above, then $\{\overline{A}_N, \overline{B}_N\}$ has 1 eigenvalue exactly at $\frac{2N^2}{jk}$ and $N-1$ eigenvalues at 0.
\end{lemma}
\begin{proof}
Since $j$ and $k$ are relatively prime, then from matrix multiplication, we see that $\overline{A}_N\overline{B}_N$ and $\overline{B}_N\overline{A}_N$ are both the constant matrix where every entry is $\frac{N}{kj}$. Hence, $\{\overline{A}_N, \overline{B}_N\}$ is the constant matrix where every entry is $\frac{2N}{kj}$. Such matrix has 1 eigenvalue exactly at $\frac{2N^2}{kj}$ and $N-1$ eigenvalues at $0$.  
\end{proof}

\begin{lemma}\label{multipleregimeskjcheckerboard}
Let $A_N$ be an $N\times N$ $k$-checkerboard matrix, and $B_N$ an $N\times N$ $j$-checkerboard matrix such that $\gcd(j,k)=1$ and $jk|N$. Assume without loss of generality that $2\leq k<j$. Then the eigenvalues of $\{A_N,B_N\}$ are distributed as follows: 
\begin{enumerate}
\item $1$ eigenvalue at $\frac{2}{jk}N^2+O(N^{3/2})$,
\item $k-1$ eigenvalues at each of $\pm \frac{1}{k}\sqrt{1-\frac{1}{j}}N^{3/2}+O(N)$,
\item $1$ eigenvalue between $\frac{1}{j}\sqrt{1-\frac{1}{k}}N^{3/2}+O(N)$ and $\frac{1}{k}\sqrt{1-\frac{1}{j}}N^{3/2}+O(N)$ and $1$ eigenvalue between $-\frac{1}{k}\sqrt{1-\frac{1}{j}}N^{3/2}+O(N)$ and $-\frac{1}{j}\sqrt{1-\frac{1}{k}}N^{3/2}+O(N)$,
\item $j-2$ eigenvalues at $\pm \frac{1}{j}\sqrt{1-\frac{1}{k}}N^{3/2}+O(N)$,
\item 1 eigenvalue between $O(N)$ (positive) and $\frac{1}{j}\sqrt{1-\frac{1}{k}}N^{3/2}+O(N)$ and 1 eigenvalue between $O(N)$ (negative) and $-\frac{1}{j}\sqrt{1-\frac{1}{k}}N^{3/2}+O(N)$,
\item The remaining $N-2k-2j+1$ eigenvalues of $O(N)$.
\end{enumerate}
\end{lemma}

\begin{proof}
By assumption, we have $2\leq k<j$, so $\frac{1}{k}\sqrt{1-\frac{1}{j}}> \frac{1}{j}\sqrt{1-\frac{1}{k}}$. Since $\{A_N,B_N\}=\{\Tilde{A}_N,\Tilde{B}_N\}+\{\overline{A}_N,\Tilde{B}_N\}+\{\Tilde{A}_N,\overline{B}_N\}+\{\overline{A}_N,\overline{B}_N\}$, then
\begin{align}
\lambda_N(\{A_N,B_N\}) &\ \geq \ \lambda_N(\{\overline{A},\overline{B}\}) + \lambda_1(\{\overline{A},\Tilde{B}\} + \{\Tilde{A},\overline{B}\} + \{\Tilde{A},\Tilde{B}\}) \nonumber \\
&\ \geq \ \lambda_N(\{\overline{A},\overline{B}\}) + \lambda_1(\{\overline{A},\Tilde{B}\}) + \lambda_1(\{\Tilde{A},\overline{B}\}) + \lambda_1(\{\Tilde{A},\Tilde{B}\}) \ = \ \frac{2N^2}{jk} + O(N^{3/2}).
\end{align}
This establishes the existence of the largest blip. Then, we establish the existence of the intermediary blip containing $k-1$ eigenvalues at $\frac{1}{k}\sqrt{1-\frac{1}{j}}N^{3/2}+O(N)$:

\begin{alignat}{2}
&\lambda_{N-1}(\{A_N,B_N\}) &&\ \leq \ \lambda_{N-1}(\{\overline{A}_N,\overline{B}_N\})+\lambda_{N}(\{\overline{A},\Tilde{B}\} + \{\Tilde{A},\overline{B}\} + \{\Tilde{A},\Tilde{B}\}) \nonumber \\
 & &&  \ \leq \ \lambda_N(\{\overline{A}_N,\Tilde{B}_N\}+\{\Tilde{A}_N,\overline{B}_N\}) + \lambda_N(\{\Tilde{A}_N,\Tilde{B}_N\}) \ = \ \frac{1}{k}\sqrt{1-\frac{1}{j}}N^{3/2} + O(N),\nonumber
\\
&\lambda_{N-k+1}(\{A_N,B_N\}) && \ \geq \ \lambda_{1}(\{\overline{A}_N,\overline{B}_N)+\lambda_{N-k+1}(\{\overline{A},\Tilde{B}\} + \{\Tilde{A},\overline{B}\} + \{\Tilde{A},\Tilde{B}\}) \nonumber \\
 & &&\ \geq \ 
\lambda_{N-k+1}(\{\overline{A}_N,\Tilde{B}_N\}+\{\Tilde{A}_N,\overline{B}_N\})+\lambda_1(\{\Tilde{A}_N,\Tilde{B}_N)=\frac{1}{k}\sqrt{1-\frac{1}{j}}N^{3/2} + O(N).
\end{alignat}
Next, we show that there is one eigenvalue between $\frac{1}{j}\sqrt{1-\frac{1}{k}}N^{3/2}+O(N)$ and $\frac{1}{k}\sqrt{1-\frac{1}{j}}N^{3/2}+O(N)$ as well as $j-2$ eigenvalues at $\frac{1}{j}\sqrt{1-\frac{1}{k}}N^{3/2}+O(N)$:

\begin{alignat}{2}
&\lambda_{N-k}(\{A_N,B_N\}) &&\ \leq \ \lambda_{N-k+1}(\{\overline{A}_N,\Tilde{B}_N\} + \{\Tilde{A}_N,\overline{B}_N\}) + \lambda_{N-1}(\{\overline{A}_N,\overline{B}_N\}+\{\Tilde{A}_N, \Tilde{B}_N\}) \nonumber \\
& &&\ = \ \frac{1}{k}\sqrt{1-\frac{1}{j}}N^{3/2}+O(N),\nonumber \\
&\lambda_{N-k-1}(\{A_N,B_N\}) &&\ \leq \ \lambda_{N-k}(\{\overline{A}_N,\Tilde{B}_N\} + \{\Tilde{A}_N,\overline{B}_N\}) + \lambda_{N-1}(\{\overline{A}_N,\overline{B}_N\}+\{\Tilde{A}_N, \Tilde{B}_N\}) \nonumber \\
& &&\ = \ \frac{1}{j}\sqrt{1-\frac{1}{k}}N^{3/2}+O(N),\nonumber \\
&\lambda_{N-k-j+2}(\{A_N,B_N\}) &&\ \geq \ \lambda_1(\{\overline{A}_N,\overline{B}_N\}+\{\Tilde{A}_N,\Tilde{B}_N\}) + \lambda_{N-j-k+2}(\{\overline{A}_N,\Tilde{B}_N\} + \{\Tilde{A}_N,\overline{B}_N\}) \nonumber \\
& &&\ = \ \frac{1}{j}\sqrt{1-\frac{1}{k}}N^{3/2} + O(N).
\end{alignat}
By symmetry argument, we can use Weyl's inequality to establish the existence of their negative counterpart. Finally, for the bulk, we have 
\begin{alignat}{2}
&\lambda_{N-j-k+1}(\{A_N,B_N\}) &&\ \leq \ \lambda_{N-j-k+2}(\{\overline{A}_N,\Tilde{B}_N\} + \{\Tilde{A}_N,\overline{B}_N\}) + \lambda_{N-1}(\{\overline{A}_N,\overline{B}_N\}+\{\Tilde{A}_N,\Tilde{B}_N\}) \nonumber \\
& &&\ \leq \ \frac{1}{j}\sqrt{1-\frac{1}{k}}N^{3/2} + O(N),\nonumber \\
&\lambda_{N-j-k}(\{A_N, B_N\}) &&\ \leq \ \lambda_{N-j-k+1}(\{\overline{A}_N, \Tilde{B}_N\}+\{\Tilde{A}_N, \overline{B}_N\})+\lambda_{N-1}(\{\overline{A}_N, \overline{B}_N\}+\{\Tilde{A}_N, \Tilde{B}_N\})\nonumber  \\
& &&\ = \ O(N).
\end{alignat}
By symmetry argument, we can use Weyl's inequality to bound the bulk from the below. This completes the proof.
\end{proof}

Note that in the proof of Lemma \ref{multipleregimeskjcheckerboard}, Weyl's inequality fails to provide an accurate bound on the four eigenvalues between different regimes. Empirically, we observe that among those eigenvalues, the positive ones belong to the regimes corresponding to their lower bound, and the negative ones belong to the regimes corresponding to their upper bound.

\section{Convergence of Blip and Bulk Regimes}\label{sec:convergence}
In this section, we establish the weak convergence almost surely of the empirical blip spectral measure of $\{\textup{GOE, $k$-checkerboard}\}$ and the weak convergence almost surely of the bulk of all the anticommutator ensembles previously considered, following \cite{split}. Convergence results for the empirical largest blip spectral measure of $\{\textup{$k$-checkerboard, $j$-checkerboard}\}$ can be established analogously, so here we will only focus on the case of $\{\textup{GOE, $k$-checkerboard}\}$. By \cite{split}, we can treat the $m$\textsuperscript{th} moment of the empirical blip spectral measure of $\{\textup{GOE, $k$-checkerboard}\}$, $\mu^{(m)}_{\{A_N, B_N\},N}$, as a random variable on $\Omega$. Here, we define $\Omega:=\prod_{N \in \N} \Omega_N$. where $\Omega_N$ is the probability space of $N\times N$ $\{\textup{GOE, $k$-checkerboard}\}$ with the natural probability measure.
\begin{definition}
We define the random variable $X_{m,N}$ on $\Omega$
\begin{equation}\label{eq_xmn}
X_{m,N}(\{A_N, B_N\})\ :=\ \mu_{\{A_N, B_N\},N}^{(m)},
\end{equation}
which have centered $r$\textup{\textsuperscript{th}} moment
\begin{equation}
X_{m,N}^{(r)}\ :=\ \E{(X_{m,N}-\E{X_{m,N}})^r}.
\end{equation}
\end{definition}

\begin{defi}\label{def_average_blip_measure}
Fix a function $g: \N \rightarrow \N$. The \textbf{averaged empirical blip spectral measure} associated to $\overline{\{A, B\}} \in \Omega^\N$ is
\begin{equation}
\mu_{N,g,\overline{\{A, B\}}}\ :=\ \frac{1}{g(N)}\sum_{i=1}^{g(N)} \mu_{\{A_N,B_N\}^{(i)},N}.
\end{equation}
\end{defi}

In other words, we project onto the $N$\textsuperscript{th} coordinate in each copy of $\Omega$ and then average over the first $g(N)$ of these $N \times N$ matrices.
\begin{definition}
We denote by $Y_{m,N,g}$ the random variable on $\Omega^\N$ defined by the moments of the averaged empirical blip spectral measure
\begin{equation}
Y_{m,N,g}(\overline{\{A, B\}})\ :=\ \mu_{N,g,\overline{\{A, B\}}}^{(m)}.
\end{equation}
The centered $r$\textsuperscript{th} moment (over $\Omega^\N$) of this random variable is denoted by $Y_{m,N,g}^{(r)}$.    
\end{definition}

\begin{definition}
A sequence of measures $\{\mu_N\}_{N \in \N}$ converges \textbf{weakly} to a measure $\mu$ if
\begin{equation}
\lim_{N \rightarrow \infty} \int f d\mu_N = \int f d\mu
\end{equation}
for all $f \in \mathcal{C}_b(\R)$ (continuous and bounded functions).
\end{definition}

\begin{thm}\label{thm_as_convergence}
Let $g: \mathbb{N} \rightarrow \mathbb{N}$ be such that there exists an $\delta>0$ for which $g(N) = \omega(N^\delta)$. Then, as $N\to\infty$, the averaged blip empirical spectral measures $\mu_{N,g,\overline{\{A, B\}}}$ of $N\times N$ $\{\textup{GOE, $k$-checkerboard}\}$ converge weakly almost-surely to the measure with moments $M_{m}=\frac{1}{k}\left(\frac{2}{k^2}\right)^m \mathbb{E}\left[\textup{Tr }C^m\right]$, the limiting expected moments computed in Theorem \ref{GOE-checkerboard Moments}.
\end{thm}
To prove this theorem, it suffices to prove the following lemma, and the rest follows from the proof of Theorem 5.5 in \cite{split}.

\begin{lem}\label{lem_moments_of_moments}
Let $X_{m,N}$ be as defined in \eqref{eq_xmn}. Then the $r$\textup{\textsuperscript{th}} centered moment of $X_{m,N}$ satisfies
\begin{equation}
X_{m,N}^{(r)} \ =\ \E{\left(X_{m,N}-\E{X_{m,N}}\right)^r}\ =\ O_{m,r}(1)
\end{equation}
as $N$ goes to infinity.
%exists.
\end{lem}
\begin{proof}
Note that
\begin{align}
 \E{\left(X_{m,N}-\E{X_{m,N}}\right)^r}\ &\ =\  \mathbb{E}\left[\sum_{\ell=0}^r \binom{r}{\ell}(X_{m,N})^\ell \left(\E{X_{m,N}}\right)^{r-\ell}\right]\  \nonumber \\
 &\ =\ \ \sum_{\ell=0}^r \binom{r}{\ell}\mathbb{E}\left[(X_{m,N})^\ell\right] \left(\E{X_{m,N}}\right)^{r-\ell}.
\end{align}
By Theorem \ref{GOE-checkerboard Moments}, we have $\E{X_{m,N}}=O_m(1)$, so $\left(\E{X_{m,N}}\right)^{r-\ell}=O_{m,r,\ell}(1)$ for all $\ell$. As such, it suffices to show that $\E{(X_{m,N})^\ell}=O_{m,\ell}(1)$. We know that
\begin{align}
&\mathbb{E}\left[(X_{m,N})^\ell\right] \nonumber \\ 
&\ = \ \mathbb{E}\left[\left(\frac{1}{2k}\sum_{\alpha = 2n} ^{4n}c_{\alpha} \left(\frac{k^2}{N^3}\right)^{\alpha}\frac{1}{N^{5m/2}}\sum_{i = 0}^{m}\binom{m}{i} \left(-\frac{N^{3}}{k^2}\right)^{m-i}\text{Tr}\left(\{A_N,B_N\}^{2(\alpha+i)}\right)\right)^l\right]  \nonumber \\
&\ = \  \left(\frac{1}{2kN^{5m/2}}\right)^\ell\sum_{\substack{2n \leq \alpha_1 \leq 4n\\ 0 \leq i_1 \leq m}} \dots\sum_{\substack{2n \leq \alpha_l \leq 4n\\ 0 \leq i_l \leq m}}\prod_{v=1}^lc_{\alpha_v}\left(\frac{k^2}{N^3}\right)^{\alpha_v} \binom{m}{i_v} (-1)^{m-i_v}\left(\frac{N^{3}}{k^2}\right)^{m-i_v} \nonumber\\
&\quad\quad\mathbb{E} \left[\prod_{v=1}^l \textup{Tr}\left(\{A_N,B_N\}^{2(\alpha_v+i_v)}\right)\right]. \label{eq_finite_moments}
\end{align}
By the eigenvalue trace lemma,
\begin{align}
\mathbb{E}\left[\prod_{v=1}^\ell\textup{Tr} \left(\{A_N, B_N\}^{2(\alpha_v+i_v)} \right)\right) \ = \  \sum_{t^1_1,\dots,t^1_{2(\alpha_1+i_1)} \leq N} \dots \sum_{t^\ell_1,\dots,t^\ell_{2(\alpha_\ell+i_\ell)}\leq N} \mathbb{E}\left [ \prod_{v=1}^\ell c_{t^v_1,t^v_2}\dots c_{t^v_{2(\alpha_v+i_v)},t^v_1} \right].
\end{align}
Thus, we have again reduced the moment problem into a combinatorial one. By Lemma \ref{blockslemma} and \ref{Sclasscontribution}, the main contribution from configurations of length $2(\alpha_v+i_v)$ with $B_v$ $1$-blocks of $a$ is $\frac{(2(\alpha_v+i_v))^{B_v}}{B_v!}$. The number of ways to choose the number of blocks having 1-blocks of $a$ and 2-blocks of $a$ as well as the number of ways to choose matchings across the $\ell$ cyclic products are independent of $N$, $v$'s and $t_j^v$'s. Furthermore, the contribution from choosing the indices of all the blocks and $w$'s is $\frac{N^{3\sum_{v=1}^l(\alpha_v+i_v)-1/2(B_1+\cdots+B_l)}}{k^{2\sum_{v=1}^l(\alpha_v+i_v)}}$. As such, if $B_1,\dots,B_\ell \geq m$, the total contribution is $O(1)$.
If there exists $B_{v}<m$, then the overall contribution is
\begin{equation}
CN^{1/2\ell m-1/2(B_1+\dots+B_\ell)}\sum_{\substack{2n\leq \alpha_1\leq 4n \\ 0\leq i_1\leq m}}\cdots \sum_{\substack{2n\leq \alpha_l\leq 4n \\ 0\leq i_l\leq m}}\prod_{v=1}^l c_{\alpha_v}\binom{m}{i_v}(-1)^{m-i_v}\frac{(2(\alpha_v+i_v))^{B_v}}{B_v!} \ = \ 0
\end{equation}
by Lemma \ref{inequalities}. Thus, the total contribution of $\E{\left(X_{m,N}\right)^\ell}$ is simply $O(1)$.
\end{proof}
By \cite{split}, to establish weak convergence almost surely of the bulk of $\{\textup{GOE, $k$-checkerboard}\}$ (resp. $\{\textup{$k$-checkerboard, $j$-checkerboard}\}$), it suffices to prove the limiting variance of the $r$\textsuperscript{th} moment of empirical spectral measures of $\{\textup{GOE, $(k,0)$-checkerboard}\}$ (resp. $\{\textup{$(k,0)$-checkerboard, $(j,0)$-checkerboard}\}$) is $O\left(\frac{1}{N^2}\right)$. This can be established using similar techniques in Lemma A.2. of \cite{split}. 
\begin{theorem}\label{checkerboard convergence} Let $\{A_N, B_N\}$ be an $N\times N$ $\{\textup{GOE, $(k,0)$-checkerboard}\}$. Then for any fixed $r$,
\begin{align}
\lim_{N\to\infty}\textup{Var}\left[\nu_{\{A_N,B_N\}}^{(r)}\right]\ = \ O\left(\frac{1}{N^2}\right).
\end{align}
\end{theorem}
By \cite{palindromicToeplitz}, to establish weak convergence almost surely of the bulk of $\{\textup{$k$-checkerboard, $j$-checkerboard}\}$, we only need to prove the fourth moment bound. Applying the same proof method, we arrive at the following theorem which establishes the weak convergence almost surely of all the anticommutator ensembles we considered in Section \ref{sec: anticommutator Combinatorics}.
\begin{theorem}\label{section2Bulk}Let $\{A_N, B_N\}$ be an $N\times N$ anticommutator matrix from Section \ref{sec: anticommutator Combinatorics}. Then we have the fourth moment bound
\begin{align}
\lim_{N\to\infty}\mathbb{E}\left[|M_m(\{A_N,B_N\})-\mathbb{E}\left[M_m(\{A_N,B_N\})\right]|^4\right] \ = \ O_m\left(\frac{1}{N^2}\right).
\end{align}
\end{theorem}

\section{Polynomial Weight Functions for Intermediary Blips}\label{weight polynomial}
The exact choice of the weight function does not affect the values of limiting expected moments of the blip regimes, but only the details of the moment calculation. For the sake of completeness, we include here the expression for the weight functions for the intermediary blips of $\{k\textup{-checkerboard}, j\textup{-checkerboard}\}$. Let $w_1=\frac{1}{k}\sqrt{1-\frac{1}{j}}$, $w_2=\frac{1}{j}\sqrt{1-\frac{1}{k}}$, and $w_3=2/kj$ Then, the weight function for the intermediary blip around $\pm w_sN^{3/2}+\Theta(N)$ is
\begin{align}
g_s^{(2n)}(x) \ = \ \frac{x^{2n}\left(x^2-\frac{w^2_{t_s}}{w^2_s}\right)^{2n}\left(x-\frac{w_3\sqrt{N}}{w_s}\right)^{10n}}{\left(1-\frac{w_{t_s}^2}{w_s^2}\right)^{2n}\left(1-\frac{w_3 \sqrt{N}}{w_s}\right)^{10n}},
\end{align}
where $t_s=2$ if $s=1$ and $t_s=1$ if $s=2$. It's clear that $g_s^{(2n)}$ has zeros of order $2n$ at $0$, $\pm w_{t_s}/w_s$, and zero of order $10n$ at $w_3\sqrt{N}/w_s$. We prove that $g_s^{(2n)}$ vanishes at $O\left(1/\sqrt{N}\right)$, $\pm w_{t_s}/w_s+O\left(1/\sqrt{N}\right)$, and $w_3\sqrt{N}/w_s+O(1/\sqrt{N})$ and is equal to $1$ at $1+O(1/\sqrt{N})$. The key to the proof is the evaluation of the expression $\lim_{N\rightarrow\infty}(1+O(1/\sqrt{N}))^{2n}$. Since $\lim_{y\rightarrow \infty}(1+1/y)^y=e$ and $n= \log\log(N)\ll \sqrt{N}$, then
\begin{align}
1 \ \leq \ \lim_{N\rightarrow\infty}(1+O(1/\sqrt{N}))^{2n} \ = \ \lim_{N\rightarrow\infty}\left((1+O(1/\sqrt{N}))^{O(\sqrt{N})}\right)^{2n/O(\sqrt{N})} \ = \ \lim_{N\rightarrow\infty} e^{2n/O(\sqrt{N})} \ = \ 1.
\end{align}
Hence, $\lim_{N\rightarrow\infty}(1+O(1/\sqrt{N}))^{2n}=1$.

Now, we first evaluate the function at $x=1+O(1/\sqrt{N})$ as $N\rightarrow\infty$,
\begin{align}
&\lim_{N\rightarrow\infty}g_s^{(2n)}\left(1+O\left(\frac{1}{\sqrt{N}}\right)\right) \nonumber \\ 
&\ = \ \lim_{N\rightarrow\infty}\left(1+O\left(\frac{1}{\sqrt{N}}\right)\right)^{2n}\cdot \frac{\left(\left(1+O\left(\frac{1}{\sqrt{N}}\right)\right)^2-\frac{w^2_{t_s}}{w_s^2}\right)^{2n}}{\left(1-\frac{w^2_{t_s}}{w_s^2}\right)^{2n}}\cdot\frac{\left(1+O\left(\frac{1}{\sqrt{N}}\right)-\frac{w_3\sqrt{N}}{w_s}\right)^{10n}}{\left(1-\frac{w_3\sqrt{N}}{w_s}\right)^{10n}}\nonumber\\
&\ = \ 1.
\end{align}
Note that we repeatedly apply the evaluation $\lim_{N\rightarrow\infty}(1+O(1/\sqrt{N}))^{2n}=1$ above to simplify the expression. As an example, $\left(\frac{(1+O(1/\sqrt{N}))^2-A^2}{1-A^2}\right)^{2n}=\left(1+\frac{O(1/\sqrt{N})}{1-A^2}\right)^{2n}=\left(1+O(1/\sqrt{N})\right)^{2n}=1$.

Then, we consider the evaluation of the function at $O(1/\sqrt{N})$, and the evaluation of the function at other vanishing points similarly follows,
\begin{align}
&\lim_{N\rightarrow\infty}g_s^{(2n)}\left(O\left(\frac{1}{\sqrt{N}}\right)\right) \nonumber \\
&\ = \ \lim_{N\rightarrow\infty}\left(O\left(\frac{1}{\sqrt{N}}\right)\right)^{2n}\cdot\frac{\left(O\left(\frac{1}{\sqrt{N}}\right)^2-\frac{w_{t_s}^2}{w_s^2}\right)^{2n}}{\left(1-\frac{w^2_{t_s}}{w^2_s}\right)^{2n}}\cdot \frac{\left(O\left(\frac{1}{\sqrt{N}}\right)-\frac{w_3\sqrt{N}}{w_s}\right)^{10n}}{\left(1-\frac{w_3\sqrt{N}}{w_s}\right)^{10n}}
\nonumber \\
&\ = \ \lim_{N\rightarrow\infty} C_N \cdot \left(O\left(\frac{1}{\sqrt{N}}\right)\right)^{2n} \ = \ 0.
\end{align}

\section{Lower even moments of the Anticommutator of GOE and BCE and of BCE and BCE}\label{lower even moments involving BCE}

We provide here explicit expression of the first few even moments of $\{\textup{GOE}, k\textup{-BCE}\}$ and $\{k\textup{-BCE}, k\textup{-BCE}\}$ based on genus expansion formulae, where distributions are rescaled and normalized to have mean zero and variance one. This means that the second method of both distributions are 1. Theoretically, with enough computing power, we should be able to obtain closed form expressions for any even moments for the two distribution. However, in reality, the computation is extremely complicated and time-consuming. Hence, we only provide the moments of $\{\textup{GOE, }k\textup{-BCE}\}$ up to the 10\textsuperscript{th} moment and $\{k\textup{-BCE}, k\textup{-BCE}\}$ up to the 8\textsuperscript{th} moment.

\begin{table}[ht]
\centering
\caption{First Few Even Moments of $\{\textup{GOE}, k\textup{-BCE}\}$.}
\renewcommand{\arraystretch}{3}
\begin{tabular}{|>{\centering\arraybackslash}m{1.5cm}|>{\centering\arraybackslash}m{4cm}|} 
\hline
Moment & Value \\
\hline
4\textsuperscript{th} & $\dfrac{5}{2}+\dfrac{1}{2k^2}$ \\ 
\hline
6\textsuperscript{th} & $\dfrac{33}{4}+\dfrac{19}{4k^2}$ \\ 
\hline
8\textsuperscript{th} & $\dfrac{249}{8}+\dfrac{34}{k^2}+\dfrac{27}{8k^4}$ \\ 
\hline
10\textsuperscript{th} & $\dfrac{2033}{16}+\dfrac{875}{4k^2}+\dfrac{1043}{16k^4}$ \\
\hline
\end{tabular}
\end{table}

\begin{table}[ht]
\centering
\caption{First Few Even Moments of $\{k\textup{-BCE}, k\textup{-BCE}\}$.}
\renewcommand{\arraystretch}{3}
\begin{tabular}{|>{\centering\arraybackslash}m{1.5cm}|>{\centering\arraybackslash}m{10cm}|} 
\hline
Moment & Value \\
\hline
4\textsuperscript{th} & $\dfrac{10k^4+86k^2+48}{4k^4+8k^2+4}$ \\ 
\hline
6\textsuperscript{th} & $\dfrac{66k^6+1890k^4+9084k^2+3360}{8k^6+24k^4+24k^2+8}$ \\ 
\hline
8\textsuperscript{th} & $\dfrac{498k^8+33236k^6+529634k^4+1759064k^2+499968}{16k^8+64k^6+96k^4+64k^2+16}$ \\
\hline
\end{tabular}
\end{table}
\newpage

\begin{table}[ht]
\centering
\caption{Comparison of Theoretical and Empirical Even Moments of $\{\textup{GOE, $2$-BCE}\}$.}
\renewcommand{\arraystretch}{2}
\begin{tabular}{|>{\centering\arraybackslash}m{1.5cm}|>{\centering\arraybackslash}m{3cm}|>{\centering\arraybackslash}m{3cm}|} 
\hline
Moment & Theoretical Value & Empirical Value\\
\hline
4\textsuperscript{th} & $10.5$ & $10.47$\\ 
\hline
6\textsuperscript{th} & $75.5$ & $79.47$\\ 
\hline
8\textsuperscript{th} & $637.38$ & $668.60$\\
\hline
10\textsuperscript{th} & $5946.38$ & $6657.43$ \\
\hline
\end{tabular}
\end{table}

\begin{table}[ht]
\centering
\caption{Comparison of Theoretical and Empirical Even Moments of $\{\textup{$2$-BCE, $2$-BCE}\}$.}
\renewcommand{\arraystretch}{2}
\begin{tabular}{|>{\centering\arraybackslash}m{1.5cm}|>{\centering\arraybackslash}m{3cm}|>{\centering\arraybackslash}m{3cm}|} 
\hline
Moment & Theoretical Value & Empirical Value\\
\hline
4\textsuperscript{th} & $5.52$ & $10.19$\\ 
\hline
6\textsuperscript{th} & $74.16$ & $70.75$\\ 
\hline
8\textsuperscript{th} & $1826.50$ & $1179.23$\\
\hline
\end{tabular}
\end{table}

%%%%%%%%%%%%%%%%%%%%%%%%%%%%%%%%%%%%%%%%%%%%%%%%%%%%%%%%%%
%%%%%%%%%%%%%%%%%%%%%%%%%%%%%%%%%%%%%%%%%%%%%%%%%%%%%%%%%%%%%%%%%%%%%%%%%%%%%%%%%%%%%%%%%%%%%%%%%%%%%%%%%%%%%%%%%%%%%%%%%%%%%%
%%%%%%%%%%%%%%%%%%%%%%%%%%%%%%%%%%%%%%%%%%%%%%%%%%%%%%%%%%%%%%%%%%%%%%%%%%%%%%%%%%%%%%%%%%%%%%%%%%%%%%%%%%%%%%%%%%%%%%%%%%%%%%
%%%%%%%%%%%%%%%%%%%%%%%%%%%%%%%%%%%%%%%%%%%%%%%%%%%%%%%%%%%%%%%%%%%%%%%%%%%%%%%%%%%%%%%%%%%%%%%%%%%%%%%%%%%%%%%%%%%%%%%%%%%%%%

\ \\
\end{document}